\documentclass{amsart}
\usepackage[utf8]{inputenc}

\usepackage{geometry}
\usepackage{epsdice}

\usepackage{enumitem}

\usepackage{color}
\usepackage{comment}

\usepackage[normalem]{ulem}

\usepackage[OT2,T1]{fontenc}

\usepackage[final]{microtype}

\usepackage{tikz}
\usepackage{tikz-cd}
\usetikzlibrary{arrows,shapes,chains}
\usepackage{xypic}

\usepackage[draft=false,linktocpage=true]{hyperref}
\usepackage[capitalise]{cleveref}

\usepackage{relsize}
\usepackage[bbgreekl]{mathbbol}
\usepackage{amsfonts}
\DeclareSymbolFontAlphabet{\mathbb}{AMSb} 
\DeclareSymbolFontAlphabet{\mathbbl}{bbold}
\newcommand{\Prism}{{\mathlarger{\mathbbl{\Delta}}}} 

\usepackage{mathtools}
\usepackage{amsmath}

\usepackage{amsthm,amssymb}
\numberwithin{equation}{section}

\theoremstyle{plain}
\newtheorem{theorem}[equation]{Theorem}

\newtheorem{proposition}[equation]{Proposition}
\newtheorem{lemma}[equation]{Lemma}
\newtheorem{claim}[equation]{Claim}
\newtheorem{corollary}[equation]{Corollary}
\newtheorem{porism}[equation]{Porism}

\theoremstyle{definition}
\newtheorem{definition}[equation]{Definition}
\newtheorem{notation}[equation]{Notation}
\newtheorem{construction}[equation]{Construction}
\newtheorem{question}[equation]{Question}

\newtheorem{example}[equation]{Example}

\newtheorem{remark}[equation]{Remark}

\newcommand{\rH}{{\mathrm{H}}}
\newcommand{\qsyn}{{\mathrm{qSyn}}}

\newcommand*{\s}{{\mathfrak{S}}}

\DeclareMathOperator{\gr}{gr}
\DeclareMathOperator{\id}{id}

\DeclareMathOperator{\Hom}{Hom}
\DeclareMathOperator{\Ker}{Ker}

\DeclareMathOperator{\Ext}{Ext}
\DeclareMathOperator{\Gal}{Gal}
\DeclareMathOperator{\GL}{GL}

\DeclareMathOperator{\Spec}{Spec}
\DeclareMathOperator{\Spf}{Spf}

\DeclareMathOperator{\coker}{coker}
\DeclareMathOperator{\et}{\acute{e}t}

\DeclareMathOperator{\dR}{dR}
\DeclareMathOperator{\Cris}{Cris}
\DeclareMathOperator{\cris}{crys}

\DeclareMathOperator{\tors}{tors}
\DeclareMathOperator{\free}{fr}

\DeclareMathOperator{\Mod}{Mod}

\DeclareMathOperator{\Fil}{Fil}

\DeclareMathOperator{\MF}{MF}

\DeclareMathOperator{\Rep}{Rep}
\DeclareMathOperator{\tor}{tor}
\DeclareMathOperator{\tf}{tf}

\DeclareMathOperator{\Tor}{Tor}

\DeclareMathOperator{\FL}{{FM}}

\newcommand{\LL}{{\mathbb{L}}}

\newcommand*{\Z}{\ensuremath{\mathbb{Z}}}

\newcommand*{\Kbar}{\overline{K}}

\newcommand*{\m}{\mathfrak{M}}

\newcommand*{\C}{\mathbf{C}}

\newcommand*{\F}{\mathbf{F}}

\newcommand*{\calF}{\mathcal{F}}

\newcommand*{\calM}{\mathcal{M}}
\renewcommand*{\O}{\mathcal{O}}
\newcommand*{\calK}{\mathcal{K}}
\newcommand*{\calN}{\mathcal{N}}

\newcommand{\rN}{{\rm N}}

\newcommand*{\calL}{\mathcal{L}}

\newcommand*{\M}{\ensuremath{\mathbf{M}}}

\renewcommand*{\int}{\ensuremath{\mathrm{int}}}

\renewcommand*{\u}[1]{\underline{#1}}

\newcommand*{\inj}{\hookrightarrow}
\newcommand*{\onto}{\twoheadrightarrow}

\renewcommand{\tilde}{\widetilde}
\renewcommand{\bar}{\overline}

\newcommand{\ku}{{k[\![u]\!]}}
\newcommand{\gt}{{\mathfrak t}}

\newcommand{\cI}{\mathcal I}
\newcommand{\cX}{\mathcal X}

\newcommand{\upi}{\underline{\pi}}

\title{On the $u^{\infty}$-torsion submodule of prismatic cohomology}

\author{Shizhang Li}

\address{Department of Mathematics,University of Michigan, 530 Church Street,
  Ann Arbor, MI 48109}

\email{shizhang@umich.edu}

\author{Tong Liu}

\address{Department of Mathematics
                 150 N. University Street
                 Purdue University
                 West Lafayette, IN 47907}

\email{tongliu@purdue.edu}

\date{First version March 22nd, 2022; Revised version \today.}

\begin{document}

\maketitle

\begin{abstract}
We investigate the maximal finite length submodule of the Breuil--Kisin prismatic cohomology
of a smooth proper formal scheme over a $p$-adic ring of integers.
This submodule governs pathology phenomena in integral $p$-adic cohomology theories.
Geometric applications include a control, in low degrees and mild ramifications, of 
(1) the discrepancy between two naturally associated Albanese varieties in
characteristic $p$, and
(2) kernel of the specialization map in $p$-adic \'{e}tale cohomology.
As an arithmetic application, we study the boundary case of the theory due to
Fontaine--Laffaille, Fontaine--Messing, and Kato.
Also included is an interesting example, 
generalized from a construction in Bhatt--Morrow--Scholze's work, which
(1) illustrates some of our theoretical results being sharp, and
(2) negates a question of Breuil.
\end{abstract}

\tableofcontents

\section{Introduction}

Let $\mathcal{O}_K$ be a mixed characteristic DVR with
perfect residue field $k$ and fraction field $K$.
Let $\mathcal{X}$ be a smooth proper (formal) scheme over $\mathcal{O}_K$,
it is natural to ask how the geometry of $\mathcal{X}_k$ and $\mathcal{X}_K$ are related.
Recall that proper base change theorem~\cite[\href{https://stacks.math.columbia.edu/tag/0GJ2}{Tag 0GJ2}]{stacks-project} says that,
for any prime $\ell$, there is a specialization map
\[
\mathrm{Sp} \colon \mathrm{R\Gamma}_{\et}(\mathcal{X}_{\bar{k}}, \mathbb{Z}_\ell) \to \mathrm{R\Gamma}_{\et}(\mathcal{X}_{\bar{K}}, \mathbb{Z}_\ell).
\]
The smooth base change theorem says~\cite[\href{https://stacks.math.columbia.edu/tag/0GKD}{Tag 0GKD}]{stacks-project} that the above map is an isomorphism for any $\ell \not= p$.

The lack of the smooth base change theorem when $\ell = p$ is related to many interesting ``pathology'' phenomena
in $p$-adic cohomology theories.
In this paper, we try to investigate these pathologies using the recent advances of prismatic cohomology theory.

The driving philosophy in this article is the following: 
recall in \cite{BMS1}, \cite{BMS2}, and \cite{BS19}, the authors
attached a natural cohomology theory, known as the prismatic cohomology, to the mixed characteristic family $\mathcal{X}/\mathcal{O}_K$.
This cohomology can be thought of as ``the universal $p$-adic cohomology theory'',
therefore we expect certain well-defined piece inside prismatic cohomology to be ``the universal source of pathology''
in all $p$-adic cohomology theory.
Before explicating the above, let us first say that the comparison between
\'{e}tale torsion and crystalline torsion as in \cite[Theorem 1.1.(ii)]{BMS1}
serves as the initial inspiration.
Now let us showcase two more such pathologies and state what our
main theorem specializes to in these two cases.

\subsection*{Albanese and reduction}
\addtocontents{toc}{\protect\setcounter{tocdepth}{1}}
Let us assume, in addition to above, that $\mathcal{X}$ possesses an $\mathcal{O}_K$-point $x$.
Associated with the pair $(\mathcal{X}, x)$ is a functorial map of abelian varieties
$f \colon \mathrm{Alb}(\mathcal{X}_k) \to \mathcal{A}_k$ over $k$,
where $\mathcal{A}$ is the N\'{e}ron model of the Albanese of $(\mathcal{X}_K, x_K)$.
The smooth and proper base change theorem tells us that $f$ is a $p$-power isogeny.
What can we say about the $\ker(f)$?


\begin{theorem}[{\Cref{defect controlled by ramification index}}]
\label{Thm 1}
Let $e$ be the ramification index of $\mathcal{O}_K$.
\begin{enumerate}
\item If $e < p-1$ then the map $f \colon \mathrm{Alb}(\mathcal{X}_0) \to \mathrm{Alb}(X)_0$ is an isomorphism.
\item If $e < 2(p-1)$ then $\ker(f)$ is $p$-torsion.
\item If $e = p-1$ then $\ker(f)$ is $p$-torsion and of multiplicative type, hence must be a form of several copies of $\mu_p$.
Moreover there is a canonical injection of $\mathcal{O}_K$-modules
\[
\mathbb{D}(\ker(f)) \otimes_k \big(\mathcal{O}_K/p\big) \hookrightarrow \rH^2(\mathcal{X}, \mathcal{O}_{\mathcal{X}}).
\]
\end{enumerate}
\end{theorem}

Here $\mathbb{D}(-)$ denotes the Dieudonn\'{e} module of said finite flat group scheme.
If one translates this result to a statement concerning maps between Picard schemes, 
then our result slightly refines an old result by Raynaud~\cite[Th\`{e}or\'{e}me 4.1.3]{Ray79}
in the setting of smooth central fiber, see \Cref{translation to Picard remark}.

\subsection*{Kernel of specialization}
The $p$-adic specialization map is not an isomorphism, as it is almost never going to be surjective,
for the rank of the source is at most half of the rank of the target.
One can still ask whether the $p$-adic specialization map is injective or not.
\begin{theorem}[{\Cref{control ker Cosp}}]
Let $e$ be the ramification index of $\mathcal{O}_K$, 
and let $i \in \mathbb{N}$.
Consider the specialization map
$\mathrm{Sp}^i \colon 
\rH^i_{\et}(\mathcal{X}_{\bar{k}}, \mathbb{Z}_p) \to \rH^i_{\et}(\mathcal{X}_{\bar{K}}, \mathbb{Z}_p)$.
\begin{enumerate}
\item If $e \cdot (i-1) < p-1$, then $\mathrm{Sp}^i$ is injective.
\item If $e \cdot (i-1) < 2(p-1)$, then $\ker(\mathrm{Sp}^i)$ is annihilated by $p^{i-1}$.
\item If $e \cdot (i-1) = p-1$, then $\ker(\mathrm{Sp}^i)$ is $p$-torsion,
and there is a $\mathrm{Gal}(\bar k /k)$-equivariant injection:
\[
\ker(\mathrm{Sp}^i) \otimes_{\mathbb{F}_p} \big(\mathcal{O}_K \otimes_W W(\bar{k})\big)/p
\hookrightarrow \rH^i(\mathcal{O}_{\mathcal{X}}) \otimes_{W} W(\bar{k}).
\]
\end{enumerate}
\end{theorem}

The above two theorems are of similar shape, and that is because
they are shadows of the same result concerning prismatic cohomology,
which we explain next.

\subsection*{Prismatic input}
Choose a uniformizer $\pi \in \mathcal{O}_K$, there is a canonical surjection
$\s \coloneqq W(k)[\![u]\!] \twoheadrightarrow \mathcal{O}_K$
with kernel generated by the Eisenstein polynomial of $\pi$,
which has degree given by the ramification index $e$.
Let $\varphi_{\s}$ be the endomorphism on $\s$ which restricts to usual
Frobenius on $W(k)$ and sends $u$ to $u^p$.
The triple $(\s, (E), \varphi_{\s})$ is known as the Breuil--Kisin prism associated with
$(\mathcal{O}_K, \pi)$ \cite[Example 1.3.(3)]{BS19}.

In \cite{BMS2}, and \cite{BS19}, the authors attached an $\s$-perfect complex
$\mathrm{R\Gamma}_{\Prism}(\mathcal{X}/\s)$ with a Frobenius operator.
Similar to the classical crystalline story, the Frobenius operator 
is also an isogeny. A concrete consequence of having an isogenous Frobenius map
is that the torsion submodule in $\rH^i_{\Prism}(\mathcal{X}/\s)$
is $p$-power torsion \cite[Proposition 4.3.(i)]{BMS1}.
Hence the torsion must be supported on
$\Spec(\s/p)$, note that $\s/p \cong k[\![u]\!]$ is a DVR.
An upshot of the above discussion is that we have three descriptions of
a submodule in $\rH^i_{\Prism}(\mathcal{X}/\s)$:
\begin{enumerate}
\item the $u^{\infty}$-torsion submodule in $\rH^i_{\Prism}(\mathcal{X}/\s)$,
from now on we denote it as $\rH^i_{\Prism}(\mathcal{X}/\s)[u^{\infty}]$;
\item the maximal finite length submodule in $\rH^i_{\Prism}(\mathcal{X}/\s)$; and
\item the submodule in $\rH^i_{\Prism}(\mathcal{X}/\s)$ supported at the closed point in $\Spec(\s)$.
\end{enumerate}

To convey readers the above \emph{is} the universal source of pathology in $p$-adic cohomology theory,
let us exhibit the connection between $u^{\infty}$-torsion and our
results before.

\begin{theorem}
\leavevmode
\begin{enumerate}
\item \emph{(}\Cref{defect of Albanese}\emph{)}
Concerning the natural map $f \colon \mathrm{Alb}(\mathcal{X}_0) \to \mathrm{Alb}(X)_0$,
we have a natural isomorphism of Dieudonn\'{e} modules:
\[
\mathbb{D}(\ker(f))^{(-1)} \cong \rH^2_{\Prism}(\mathcal{X}/\s)[u],
\]
where $(-)^{(-1)}$ denote the Frobenius \emph{untwist} and $(-)[u]$ denotes
the $u$-torsion submodule.
\item \emph{(}\Cref{ker Cosp and u-torsion}\emph{)}
As for the kernel of $p$-adic specialization map, we have a natural isomorphism
of $\mathrm{Gal}(\bar k/k) $-representations:
\[
\ker(\mathrm{Sp}^i) \cong 
\big(\rH^i_{\Prism}(\mathcal{X}/\s)[u^\infty]/u \otimes_{W(k)} W(\bar{k})\big)^{\varphi = 1}.
\]
\end{enumerate}
\end{theorem}

In view of aforementioned statements, the reader can probably guess what our main result,
concerning the structure of $u^\infty$-torsion in prismatic cohomology, should look like.
\begin{theorem}[{\Cref{cor control u-torsion} and~\Cref{torsion in O bounds prismatic torsion}}]
\label{Thm 1.4}
Let us write $\mathfrak{M}^i \coloneqq \rH^i_{\Prism}(\mathcal{X}/\s)[u^\infty]$,
and write $\rm Ann(-)$ for the annihilator ideal of an $\s$-module.
\begin{enumerate}
\item If $e \cdot (i-1) < p-1$, then $\mathfrak{M}^i = 0$.
\item If $e \cdot (i-1) < 2(p-1)$, then ${\rm Ann}(\mathfrak{M}^i) + (u) \supset (p^{i-1},u)$.
\item If $e \cdot (i-1) = p-1$, then ${\rm Ann}(\mathfrak{M}^i) \supset (p,u)$.
Moreover the semi-linear Frobenius on $\mathfrak{M}^i$ is bijective,
and there is a natural injection 
$\mathfrak{M}^i \otimes_{k} \big(\mathcal{O}_K/p\big)
\hookrightarrow \rH^i(\mathcal{X}, \mathcal{O}_{\mathcal{X}})$.
\end{enumerate}
\end{theorem}

\begin{remark}
\label{Rmk 1.5}
\leavevmode
\begin{enumerate}
\item We also prove the mod $p^n$ analogs as well.
As a consequence we obtain the following \Cref{finite free in low degree and ramification}
concerning the shape of prismatic cohomology:
Let $i$ be an integer satisfying $e \cdot (i-1) < p-1$,
then there exists a (non-canonical) isomorphism of $\s$-modules:
\[
\rH^i_{\Prism}(\mathcal{X}/\s) \simeq \rH^i_{\et}(\mathcal{X}_{\bar{K}}, \mathbb{Z}_p) \otimes_{\mathbb{Z}_p} \s.
\]
\item Previously Min has obtained vanishing of $\mathfrak{M}^i$
with the assumption $e \cdot i < p-1$ \cite[Theorem 0.1]{Min20}.
Let us briefly explain the appearance of $(i-1)$ in our result,
which might seem odd at first glance. 
It is due to the fact that the prismatic Verschiebung operator $V_i$ becomes
canonically divisible by $E$
when restricted to the $p^\infty$-torsion submodule or the $u^\infty$-torsion submodule,
and these submodules with the usual prismatic Frobenius and the ``divided Verschiebung''
is canonically a (generalized) Kisin module of height $(i-1)$ instead of $i$.
For more details, see \Cref{induced Vers Corollary 2}.
\item One may ask if there can be a better trick/argument showing better
bounds on vanishing of $u$-torsion.
Later on we shall explain a generalization of a construction in
\cite[Subsection 2.1]{BMS1} with $u$-torsion in cohomological degree $2$
and ramification index $p-1$.
Hence our result is actually sharp in terms of largest $e \cdot (i-1)$ allowed.
\end{enumerate}
\end{remark}

\subsection*{Special fiber telling Hodge numbers of the generic fiber}
As a third geometric application of our result,
we revisit the question discussed in \cite{Li20}:
what mild condition on $\mathcal{X}$ guarantees that the Hodge numbers of the generic fibre $X$
can be read off from the special fibre $\mathcal{X}_0$?
In loc.~cit.~the first named author obtained a result along this line, with technical input
of prismatic cohomology and the structural result in \cite{Min20}.
However it was noted that the results in loc.~cit.~are
not optimal already in the unramified case, when compared with what one got from
results by Fontaine--Messing, Kato, and Wintenberger.
We analyzed the situation and concluded that it is because we lack knowledge of the shape of
$u^\infty$-torsion in prismatic cohomology in the boundary degree.
This paper is partially motivated by the hope to improve results in \cite{Li20},
and our improvement is:
\begin{theorem}[{\Cref{improving Li20}, Improvement of \cite[Theorem 1.1]{Li20}}]
Let $\mathcal{X}$ be a smooth proper $p$-adic formal scheme over $\Spf(\mathcal{O}_K)$
of ramification index $e$.
Let $T$ be the largest integer such that $e \cdot (T-1) \leq p-1$.
\begin{enumerate}
\item Assume there is a lift of $\mathcal{X}$ to $\s/(E^2)$,
then for all $i,j$ satisfying $i + j < T$, we have equalities
\[
h^{i,j}(X) = \mathfrak{h}^{i,j}(\mathcal{X}_0)
\]
where the latter denotes virtual Hodge numbers of $\mathcal{X}_0$, defined as in \cite[Definition 3.1]{Li20}.
\item Assume furthermore that $e \cdot (\dim{\mathcal{X}_0}-1) \leq p-1$.
Then the special fibre $\mathcal{X}_0$ knows the Hodge numbers of the rigid generic fibre $X$.
\end{enumerate}
\end{theorem}

Along the way, we also improve the results in \cite{Li20} concerning the integral Hodge--de Rham spectral sequence
(see~\Cref{integral HdRSS}),
as well as obtain a curious degeneration statement of the ``Nygaard--Prism'' spectral sequence
(see \Cref{Nygaard--Prism degeneration}) in the unramified case.

\subsection*{Application to integral $p$-adic Hodge theory}
It is a central theme in integral $p$-adic Hodge theory to understand Galois representations such as
$\rH^i_{\et}(\mathcal{X}_{\bar{K}}, \mathbb{Z}_p)$ in terms of linear algebraic data
such as certain crystalline cohomology of $\mathcal{X}$ together with natural structures.
The first result along such lines is that of Fontaine--Messing \cite{FontaineMessing}
and Kato \cite{Katovanishingcycles}, which treats the case of $e=1$ (namely unramified base) and 
$i < p-1$.\footnote{See also \cite{AMMN21} for an approach of different flavor.}
Later on Breuil \cite{Bre98} generalized the above to semistable $\mathcal{X}$,
whereas Faltings \cite{Faltings} studied the analogue for $p$-divisible groups allowing arbitrary ramification index $e$.
A few years later, Caruso \cite{CarusoInvent} made progress allowing
$e > 1$ as long as $e \cdot (i+1) < p-1$,\footnote{For the mod $p$ analogue, Caruso's work even allows $e \cdot i < p-1$.}

To our interest is Breuil's question \cite[Question 4.1]{BreuilIntegral}:
\begin{question}
\label{Breuil question}
Assuming $i < p-1$ and let $S$ be the $p$-adic divided power envelope of $\s \twoheadrightarrow \mathcal{O}_K$, 
then the (torsion-free) crystalline cohomology $\rH^i_{\cris}(\mathcal{X}/S)/\mathrm{tors}$ together with 
its natural structure (such as divided Frobenius operator, filtration and connection) should be a ``strongly divisible lattice''
and ``corresponds'' to the Galois representation $\rH^i_{\et}(\mathcal{X}_{\bar{K}}, \mathbb{Z}_p)$.
\end{question}

All works mentioned above can be thought of as solving various special cases of the above question.
In \cite[Theorem 7.22]{LL20}, a connection with $u$-torsion in prismatic cohomology is observed:
We showed that, fix an $i < p-1$ and a smooth proper formal scheme $\mathcal{X}/\mathcal{O}_K$,
the mod $p^n$ analogue of the above question has a positive answer in degree $i$
if and only if both of $i$-th and $(i+1)$-st mod $p^n$ prismatic cohomology of $\mathcal{X}/\s$
are $u$-torsion free.
In loc.~cit., we then used Caruso's result on the mod $p$ analog as a starting point
to do an induction to show the vanishing as in \Cref{Thm 1.4} (1) and \Cref{Rmk 1.5} (1),
which in turn implies the mod $p^n$ analog of Breuil's question for all $n$ and $e \cdot i < p-1$,
see \cite[Corollary 7.25]{LL20}.
In particular this gives an affirmative answer to Breuil's original question when $e \cdot i < p-1$.
In this paper, the aforementioned vanishing of $u$-torsion is easily deduced,
hence gives a ``shortcut'' to the above result bypassing Caruso's work.

In private communications with Breuil, we were encouraged to study his question beyond the above bound.
To our surprise, we discovered that the construction in \cite[Subsection 2.1]{BMS1} can be generalized to
a counterexample with $e = p-1$ and $i=1$ to Breuil's question, see \Cref{intro-example}.
Note that in this example, we have $e \cdot i = p-1$, hence our previous result was actually sharp.

The other extreme of $(e,i)$ with $e \cdot i = p-1$ is $e=1, i=p-1$.
In this case, Fontaine--Messing \cite{FontaineMessing} and Kato \cite{Katovanishingcycles}
showed that the crystalline cohomology $\rH^{p-1}_{\cris}(\mathcal{X}_n/W_n)$
together with its natural structure is still a Fontaine--Laffaille module,
which according to \cite{Fontaine--Laffaille} can be attached a Galois representation $\rho^{p-1}_{n, \mathrm{FL}}$.
It is only natural to ask:
\begin{question}
What is the relation between $\rho^{p-1}_{n, \mathrm{FL}}$ and $\rH^{p-1}_{\et}(\mathcal{X}_{\bar{K}}, \mathbb{Z}/p^n)$?
\end{question}
Although we have not found any discussion on this question, it seems consensus among experts that these two Galois
representations are different. We are not aware of any particular expectation made in the past.
Our entire \Cref{arithmetic applications} is more or less devoted to this question,
and we arrive at the following statement.
\begin{theorem}[\Cref{Thm-FM-p-1}]
\label{Thm 1.9}
There exists a natural map 
$\eta \colon \rH ^{p-1} _{\et} (\cX_\C, \Z/ p ^n \Z)(p-1)  \to \rho^{p-1}_{n, \mathrm{FL}}$ of $G_K$-representations so that 
$\ker(\eta)$ is an unramified representation of $G_K$ killed by $p$, 
and $\coker(\eta)$ sits in a natural exact sequence $ 0 \to \ker(\eta) \to \coker (\eta)\to \ker(\mathrm{Sp}^{p-1}_n)$.
\end{theorem}
Here $\mathrm{Sp}^{p-1}_n$ denotes the specialization map of mod $p^n$ \'{e}tale cohomology in degree $p-1$,
which is also known to be an unramified $G_K$-representation killed by $p$, see \Cref{control ker Cosp} (3). 
The appearance of $\ker(\eta)$ is due to the defect of a key functor in integral $p$-adic Hodge theory,
which is well-known to experts; whereas the potential $u$-torsion in degree $p$ of mod $p^n$ prismatic cohomology of $\cX$
is solely responsible for the appearance of $\ker(\mathrm{Sp}^{p-1}_n)$.

\subsection*{Example and open questions}
Now let us discuss an interesting example, generalized from \cite[Subsection 2.1]{BMS1}.
\begin{example}
\label{intro-example}
Let $\mathcal{E}/W(\bar{k})$ be the canonical lift of an ordinary elliptic curve over an algebraically closed
field $\bar{k}$ of characteristic $p$.
Fix an $n \in \mathbb{N}$ and let $\mathcal{O}_K \coloneqq W(\bar{k})[\zeta_{p^n}]$.
Over $\mathcal{O}_K$ we have a tautological map of group schemes $\chi \colon \mathbb{Z}/p^n \to \mu_{p^n}$
sending $1$ to $\zeta_{p^n}$.

With the above notation, we consider the following smooth proper Deligne--Mumford stack
$\mathcal{X} \coloneqq [\mathcal{E}_{\mathcal{O}_K}/(\mathbb{Z}/p^n)]$
where the action of $\mathbb{Z}/p^n$ is via the character $\chi$
and the embedding $\mu_{p^n} \subset \mathcal{E}[p^n]$ (as $\mathcal{E}$
is the canonical lift).
Note that its special fiber is $\mathcal{E}_{\bar{k}} \times B(\mathbb{Z}/p^n)$
and its generic fiber is an elliptic curve
$\mathcal{E}'_K \coloneqq \big(\mathcal{E}_{\mathcal{O}_K}/\mu_{p^n}\big)_K$.
In view of the pathologies discussed in the beginning of the introduction,
let us record some facts concerning this example:
\begin{itemize}
\item The Albanese map has N\'{e}ron model given by the ``further quotient'' map:
$\mathcal{X} \to \mathcal{E}' \coloneqq \mathcal{E}_{\mathcal{O}_K}/\mu_{p^n}$,
and the special fiber of this map factors as
$\mathcal{E}_{\bar{k}} \times B(\mathbb{Z}/p^n) \to \mathcal{E}_{\bar{k}} 
\xrightarrow{f} \mathcal{E}_{\bar{k}}/\mu_{p^n}$.
Here $\mathcal{E}'$ is abstractly isomorphic to $\mathcal{E}_{\mathcal{O}_K}$
since $\mathcal{E}$ was chosen to be the canonical lift.
Note that $\ker(f) = \mu_{p^n}$.
\item The fundamental group of $\mathcal{X}_{\bar{k}}$ is abelian,
with torsion given by $\mathbb{Z}/p^n$ due to the factor of $B(\mathbb{Z}/p^n)$.
By universal coefficient theorem, we have 
$\rH^2_{\et}(\mathcal{X}_{\bar{k}}, \mathbb{Z}_p)_{\mathrm{tors}} \cong \mathbb{Z}/p^n$
whereas $\rH^2_{\et}(\mathcal{X}_{\bar{K}}, \mathbb{Z}_p)$ is torsion-free.
Hence we have $\ker(\mathrm{Sp}^2) = \mathbb{Z}/p^n$.
\item One can go through the Leray spectral sequence for the cover $\mathcal{E} \to \mathcal{X}$
to compute the prismatic cohomology of $\mathcal{X}/\s$.
The most relevant computation is:
$\rH^2_{\Prism}(\mathcal{X}/\s)[u^{\infty}] \cong \s/((u+1)^{p^{n-1}} - 1, p^n)$.
\item Finally we compute the crystalline cohomology of $\mathcal{X}/S$
and to our surprise we have
$\rH^1(\mathcal{X}/S) \cong S \oplus J$
where $J$ is the ideal 
\[
\{x \in S \mid p^n \text{ divides } x \cdot ((u+1)^{p^n} - 1)\}.
\]
In particular it is torsion-free of rank $2$ yet not free.
This gives a counterexample to \Cref{Breuil question}.
\end{itemize}
By standard approximation technique, one can cook up schematic examples having all the above
features.
When $n=1$, we have $e = p-1$, therefore our aforesaid results (which was only
stated and proved for formal schemes) are sharp.
For more details, see \Cref{subsection example}.
\end{example}

Combining a generalized version of the above construction with our \Cref{Thm 1},
we get a geometric proof of Raynaud's theorem 
\cite[Th\'{e}or\`{e}me 3.3.3]{Ray74}
on prolongations of finite flat commutative group schemes
over mixed characteristic DVR, see \Cref{Raynaud subsection}.

Finally, let us end the introduction with two natural questions awaiting explorations.
We consider them to be the next step in understanding pathological torsion
in $p$-adic cohomology theory.

\begin{question}
Is there a smooth proper (formal) scheme $\mathcal{X}$ over an unramified base $W$
which has $u$-torsion in its $p$-th prismatic cohomology?
Note that $p$ is the smallest possible cohomological degree according to our result,
and when $p=2$ this is achieved by the above example.
\end{question}

\begin{question}[see \Cref{general question}]
Recall $\m^i \coloneqq \rH^i_{\qsyn}(\mathcal{X}, \Prism)[u^\infty]$.
\begin{enumerate}
\item Let $\beta$ be the smallest exponent such that $p^\beta \in \mathrm{Ann}(\m^i)$,
let $\gamma$ be the exponent such that $\mathrm {Ann}(\m^i) + (u) = (u, p^{\gamma})$.
Is there a bound on $\beta$ and $\gamma$ in terms of $e$ and $i$?
\item In light of the above example, we guess $\beta$ and/or $\gamma$ are bounded above by
$\log_p(\frac{e \cdot (i-1)}{p-1}) + 1$ when $p$ is odd.
\end{enumerate}
\end{question}

\begin{remark}
Confirming the above guess will give us results along the following line:
If $\rH^i_{\cris}(\mathcal{X}_k/W)$ has torsion \emph{not} annihilated by $p^N$,
then $\rH^i_{\et}(\mathcal{X}_{\bar{K}}, \mathbb{Z}_p)$ has torsion
\emph{not} annihilated by $p^{N - c(e,i)}$ with $c(e,i)$ being some constants
depending only on $e$ and $i$.
Note that this would be a relation between
torsion in \'{e}tale and crystalline cohomology ``converse'' 
to the one established in \cite[Theorem 1.1.(ii)]{BMS1}.
When $e \cdot i < 2(p-1)$, our \Cref{Thm 1.4}.(2) can be translated
to such a statement.
Since our \Cref{Thm 1.4}.(2) does not seem to be optimal, we do not pursue 
that direction in this paper.
\end{remark}

The logic between each section is as below.

{\large
    \begin{center}

        \centering
                    {\LARGE \textbf{Leitfaden}} \\
        \vspace{0.3in}
		\scriptsize
		\tikzstyle{format}=[rectangle,draw,thin,fill=white,minimum width=3cm,
                        minimum height = 1cm]
		\tikzstyle{test}=[diamond,aspect=2,draw,thin]
		\tikzstyle{point}=[coordinate,on grid,]

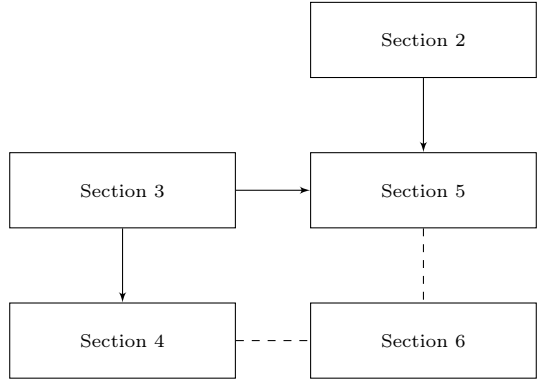
		\begin{tikzpicture}[node distance=2cm,
				auto,>=latex',
				thin,
				start chain=going below,
				every join/.style={norm},]
			\node[format] (n2){\Cref{modules and Galois rep}};
			\node[format,below of=n2] (n5){\Cref{arithmetic applications}};
			\node[point,left of=n5] (p3){};
			\node[format,left of=p3] (n3){\Cref{boundary degree}};
			\node[format,below of=n5] (n6){\Cref{subsection example}};
			\node[point,left of=n6] (p4){};
			\node[format,left of=p4] (n4){\Cref{geometric applications}};
			
			\draw[->] (n2.south) -- (n5);
			\draw[->] (n3.east) -- (n5);
			\draw[->] (n3.south) -- (n4);
			\draw[-,dashed] (n4.east) -- (n6);
			\draw[-,dashed] (n5.south) -- (n6);

		\end{tikzpicture}
    
    \end{center}
    
}

\vspace{0.3cm}

\subsection*{Notation and Conventions}
Let $k$ be a perfect field of characteristic $p > 0$ with $W = W(k)$ its Witt ring.
Let $K$ be a totally ramified degree $e$ finite extension of $W(k)[1/p]$, let $\mathcal{O}_K$ be its ring of integers.
Choose a uniformizer $\pi \in \mathcal{O}_K$ whose Eisenstein polynomial we denote by $E$ with $E(0) = a_0 p$, we get a surjection
$\s \coloneqq W[\![u]\!] \twoheadrightarrow \mathcal{O}_K$ sending $u$ to $\pi$.
We equip $\s$ the $\delta$-structure with $\varphi_{\s}(u) = u^p$.
The pair $(\s, (E))$ is the so-called Breuil--Kisin prism, see~\cite[Example 1.3.(3)]{BS19}.
Denote the $p$-adic divided power envelope of $\s \twoheadrightarrow \mathcal{O}_K$ by $S$.

We always use $C$ and its cousins like $C^{\flat}$ or $A_{\inf}$,
to denote the usual construction associated with the completion of an algebraic closure $\overline K$ of $K$
in $p$-adic Hodge theory. We use $G_K: = \Gal(\overline K / K)$ denote the absolute Galois group. Similarly, $G_k : = \Gal (\bar k / k)$. 

We use $\mathcal{X}$ to denote a smooth proper $p$-adic formal scheme on $\Spf(\mathcal{O}_K)$,
use $\mathcal{X}_0$ to denote its reduction mod $\pi$ and use $X$ to denote its rigid generic fiber.

On $(\mathcal{O}_K)_{\qsyn}$ we have the sheaf ${\Prism}$ given by (left Kan extended) 
prismatic cohomology relative to $\s$. We use ${\Prism}^{(1)}$ to denote its $\varphi_{\s}$ twist,
this sheaf of Frobenius-twisted prismatic cohomology admits a decreasing filtration called the Nygaard filtration, see~\cite[Section 15]{BS19},
which shall be denoted by $\Fil^{\bullet}_{\rN}$.
Let us note that $\mathrm{R\Gamma}_{\Prism}(\mathcal{X}/\s) \cong \mathrm{R\Gamma}_{\qsyn}(\mathcal{X}, \Prism)$
and $\varphi^*_{\s}\mathrm{R\Gamma}_{\Prism}(\mathcal{X}/\s) \cong \mathrm{R\Gamma}_{\qsyn}(\mathcal{X}, \Prism^{(1)})$.

For any $n \in \mathbb{N} \cup \{\infty\}$, we use subscript $(-)_n$ to denote the derived mod $p^n$ of
a quasi-syntomic sheaf, e.g.~$\mathrm{R\Gamma}_{\qsyn}(\mathcal{X}, \Prism^{(1)}_n) \coloneqq \mathrm{R\Gamma}_{\qsyn}(\mathcal{X}, \Prism^{(1)}/p^n)$.

In this paper we only consider relative prismatic cohomology, and hopefully readers will not confuse our notation
with the absolute prismatic cohomology developed in \cite{APC}.

\subsection*{Acknowledgements}
We are very grateful to Bhargav Bhatt for so many helpful conversations
and for sharing excitements.
During the preparation of this paper, we have benefited from discussions
with the following mathematicians: Christophe Breuil,
Johan de Jong, Luc Illusie, Dmitry Kubrak, Yu Min, Shubhodip Mondal,
Sasha Petrov, and Ziquan Yang.

The first named author thanks the support of AMS-Simons Travel Grant 2021-2023. 
He is also very grateful to the department of Mathematics at the University of Michigan
for the support and environment they provide.

\section{Various modules and their Galois representations}
\label{modules and Galois rep}
In this section, we discuss 3 types of Frobenius modules: Kisin modules, Breuil modules and Fontaine--Laffaille modules,
and their associated Galois representations.  Roughly speaking, various cohomology discussed in this paper will have these structures and functors to Galois representations just model comparison to \'etale cohomology.   
The major difference between the current work and \cite{LL20} is that we
now focus on the boundary case $eh = p-1$. 
So it is necessary to discuss \emph{nilpotent} objects for Fontaine--Laffaille modules and Breuil modules when $e=1$ and $h = p -1$.  
\subsection{Kisin modules}
\addtocontents{toc}{\protect\setcounter{tocdepth}{2}}
We review (generalized) Kisin modules from \cite[\S 6.1]{LL20}. 
Let $(\s , E(u))$ be the Breuil--Kisin prism over $\O_K$ with $d = E(u )= E$ the Eisenstein polynomial of
a fixed uniformizer $\pi\in \O_K$. 
A $\varphi$-module $\m$ over $\s$ is an $\s$-module  $\m$ together $\varphi_\s$-semilinear map $\varphi_\m: \m \to \m$. 
Write $ \varphi^*\m = \s \otimes_{\varphi, \s} \m$. Note that $1 \otimes \varphi_\m  : \varphi ^* \m \to \m$ is an $\s$-linear map.   A \emph{(generalized) Kisin module $\m$ of height $h$} is a $\varphi$-module $\m$ of finite $\s$-type
together with an $\s$-linear map $\psi: \m \to \varphi ^* \m$ such that $ \psi \circ (1 \otimes \varphi) = E^h \id_{\varphi^*\m}$ 
and $(1\otimes \varphi) \circ \psi = E^h\id_{\m}$. The map between generalized Kisin modules is $\s$-linear map that compatible with $\varphi$ and $\psi$.  We denote by $\Mod^{\varphi, h}_{\s}$ the category of (generalized) Kisin module of height $h$. As explained in \cite{LL20}, the main difference between generalized Kisin modules and classical theory Kisin modules is that the classical theory only discuss the situation that $\m$ has no $u$-torsion, while Kisin module from prismatic cohomology could have $u$-torsion in general. In the following, when we need to restrict to the classical theory, we will call $\m$ either \emph{classical} or \emph{$u$-torsion free}. Let $\Mod_{\s , \tor}^{\varphi , h, {\rm c}}$ denote the full subcategory of $\Mod_{\s}^{\varphi, h}$ consists of classical Kisin modules of height $h$ and killed by $p^n$ for some $n \in \mathbb{N}$. 

Now we review some technologies to deal with classical Kisin modules on boundary case and extends them to generalized Kisin modules. Following \cite[(1.2.10)]{Kisin-Modularity},  \cite[\S 2.1]{GaoHui1}, we call a $\varphi$-module $\m$  \emph{multiplicative} (resp.~\emph{nilpotent}) if $(1 \otimes \varphi): \varphi ^* \m \to \m$ is surjective (resp.~if $\lim\limits_{n \to \infty}\varphi^n (x)= 0, \ \forall x \in \m$). 
\begin{remark}In \cite[(1.2.10)]{Kisin-Modularity} and \cite[\S 2.1]{GaoHui1},
the authors define \emph{multiplicative} to mean $(1 \otimes \varphi): \varphi ^* \m \to \m$ is \emph{bijective}.  For classical Kisin module $\m$ these two concepts are the same as $1\otimes \varphi$ is always injective.  
But for generalized Kisin modules, as $u$-torsion exists, bijection of $1 \otimes \varphi$ is too restrictive. For example, $\s/ (p , u)\s$ with the usual Frobenius is multiplicative but $1 \otimes \varphi$ is not injective. 
\end{remark} 


Let $\m$ be a $\varphi$-module over $\s$ of finite $\s$-type. Set $M \coloneqq \m / u \m$ and write $q \colon \m \to M = \m / u \m$. 
By Fitting lemma, we have $M = M^{\rm m} \oplus M ^{\rm n }$ where $\varphi$ is bijective on $ M ^{\rm m}$ and nilpotent on $M ^{\rm n}$.

\begin{lemma}\label{lem-construct-Mm} Notations as the above, there exists a unique $W(k)$-linear section $[\cdot]: M^{\rm m} \to  \m$ so that 
$[\cdot ]$ is $\varphi$-equivariant and $q \circ [\cdot]= \id|_{M^{\rm m}}$. 
\end{lemma}
\begin{proof}
Pick any $\bar x \in M ^{\rm m}$, since $\varphi$ on $M^{\rm m}$ is bijective,  there exists unique $\bar x_n \in M$ so that $\varphi^n (\bar x_n ) = \bar x$.  Select $x_n\in \m$ a lift of $\bar x_n$ and define $[\bar x] : = \lim\limits_{n \to \infty} \varphi ^n (x_n)$. 
We first check that $\varphi ^n (x_n)$ converges to an $x \in \m$ so that $q(x) = \bar x$. Indeed, since $\varphi(\bar x_{n +1}) = \bar x_n$, $\varphi(x_{n +1}) -x_n = u y _n$ with $y _n \in \m$. So $\varphi^{n +1} (x_{n +1}) - \varphi^n (x_n) = \varphi ^n (u) \varphi^n (y )$ and hence $\varphi ^n (x_n)$ converges to a $x \in \m$ and clearly $q(x)= \bar x$.  Suppose that $x_n' \in \m$ is another lift of $\bar x_n$, then $x_n' - x_n = u z_n$ with $z_n \in \m_n$. Then $ \varphi^n  (x'_n)-\varphi^n (x_n ) = u ^{p ^n} \varphi ^n (z_n)$. So $\{\varphi^n (x_n')\}$ also converges to $x$. This implies that $x= [\bar x]$ does not depend on the choice of lift $x_n$ of $\bar x_n = \varphi^{-n} (\bar x)$. Hence the section $[\cdot]: M^{\rm m}\to \m$ is well-defined and satisfies $ q \circ [\cdot] = \id_{M^{\rm m}}$. 
For any $a\in W(k)$, it is clear that $a [\bar x] = [a \bar x]$ from construction of $[\bar x]$. So $[\cdot]$ is $W(k)$-linear. If 
$\bar y = \varphi (\bar x)$ then $\varphi(x_n)$ is a lift of $ \varphi^{-n} (y ) $ and $[\bar y] = \lim \limits_{n \to \infty} \varphi^{n+1} (x_n) = \varphi (x)= \varphi ([\bar x])$. So $[\cdot]$ is $\varphi$-equivariant. Finally, suppose there is another sections $[\cdot]': M^{\rm m} \to \m $. Then $[ \bar x] - [\bar x]' \in u \m$ for any $\bar x \in M ^{\rm m}$. Then $[ \bar x] - [\bar x]'= \varphi ^n ([\bar x_n]- [\bar x_n]')\in u ^{p ^n}\m$. This forces that $ [\bar x]= [\bar x]'$. 
\end{proof}

\begin{lemma}\label{lem-multi} Let $\m$ be a $\varphi$-module with finite $\s$-type. Then there exists an exact sequence of $\varphi$-modules
\begin{equation}\label{eqn-multi-nil-seq}
    \xymatrix{ 0 \ar[r] & \m ^{\rm m} \ar[r]&   \m \ar[r]  & \m^{\rm n}\ar[r] & 0}
\end{equation}
so that $\m ^{\rm m}$ is multiplicative and $\m^{\rm n}$ has no nontrivial multiplicative submodule. Furthermore, the above exact sequence is functorial for $\m$, and if $\m  $ is in $  \Mod_{\s , \tor}^{\varphi , h, {\rm c}}$ then so are $ \m ^{\rm m}$ and $\m ^{\rm n}$. 
\end{lemma}
\begin{proof} Note that \cite[Prop. (1.2.11)]{Kisin-Modularity} has treated the situation that $\m$ has no $u$-torsion but our idea here is slightly different. By the above lemma, we can set $\m^{\rm m}$ to be the $\s$-submodule of $\m$ generated by $[M^{\rm m}]$ and $\m^{\rm n}: = \m / \m ^{\rm m}$. Clearly, $1 \otimes \varphi: \s \otimes_{\varphi , \s} \m^{\rm m} \to \m^{\rm m}$ is surjective. Consider the right exact sequence $ \s \otimes_{W(k)} [M^{\rm}] \to \m \to \m^{\rm n}\to 0$. By modulo $u$, we have the right exact (indeed exact) sequence $M^{\rm m}\to M \to \m^{\rm n}/ u \m ^{\rm n}\to 0$. So $\m^{\rm n}/ u \m ^{\rm n}\simeq M ^{\rm n}$ and also forces $\m^{\rm m}/ u \m ^{\rm m} = M^{\rm m}$. Hence 
$\varphi$ on $\m^{\rm n}$ is topologically nilpotent as $\varphi$ on $M^{\rm n}$ is nilpotent, thus $\m^{\rm n}$ can not have nontrivial multiplicative submodule. So we obtain exact sequence \eqref{eqn-multi-nil-seq} which is functorial for $\m$ because $[\cdot]$ is clearly functorial for $\m$ by the above lemma. 

If $\m \in \Mod_{\s , \tor}^{\varphi , h, {\rm c}}$ then $\m^{\rm m}$ has no $u$-torsion. Note that the exact sequence \eqref{eqn-multi-nil-seq} modulo $u$ becomes the exact sequence 
$0 \to M^{\rm m} \to M \to M ^{\rm n}\to 0$. Then $\m^{\rm n}$ can not have $u$-torsion as $\m$ has no $u$-torsion. Hence both 
$\m ^{\rm m}$ and $\m^{\rm n}$ have no $u$-torsion. Then both $\m ^{\rm m}$ and $\m^{\rm n}$ have $E$-height $h$ by \cite[Prop.  B 1.3.5]{Fontaine} as required. 
\end{proof}

But for a generalized Kisin module $\m$ with height $h$, it is unclear if we can define $\psi: \m^{\rm m} \to \varphi^*\m ^{\rm m} $ so that $\m^{\rm  m}$ has height $h$. Luckily, we will not need such a statement.  

Let $M[p^n]$ denote the $p^n$-torsion in $M$, for later application, we need the following
two statements.
\begin{lemma}
\label{lem Zp module structure}
Let $M$ be a finitely generated $\s$-module.
Assume $M/p^nM$ are $u$-torsion free for all $n > 0$.
Then $M/(M[p^n] + pM)$ are also $u$-torsion free for all $n > 0$.
\end{lemma}

\begin{proof}
Suppose $x \in M$ is a lift of a $u$-torsion in $M/(M[p^n] + pM)$, hence satisfies $u \cdot x = y + p \cdot z$ for some $y \in M[p^n]$ and $z \in M$.
Multiply the equation by $p^n$, we get $u \cdot p^n \cdot x = p^{n+1} \cdot z$.
As $M/p^{n+1}M$ also has no $u$-torsion by assumption, we see that $p^n \cdot x = p^{n+1} \cdot \tilde{z}$
for some $\tilde{z} \in M$. Write $x = (x - p \cdot \tilde{z}) + p \cdot \tilde{z}$ shows that in fact $x \in M[p^n] + pM$, as required.
\end{proof}

\begin{proposition}
\label{prop Zp module structure}
Let $M$ be a finitely generated generalized Breuil--Kisin module.
Assume $M/p^nM$ are $u$-torsion free for all $n > 0$.
Then there exists a $\mathbb{Z}_p$-module $N$ and an isomorphism of $\s$-modules
$M \simeq N \otimes_{\mathbb{Z}_p} \s$.
\end{proposition}

\begin{proof}
First let us treat the case when $M$ is torsion. In this case $M$ is killed by a power of $p$,
see~\cite[Proposition 4.3.(i)]{BMS1}. Denote ${\rm Im}(M \xrightarrow{p} M) = pM \eqqcolon M_1$.
We claim $M_1/p^nM_1$ are also $u$-torsion free for all $n > 0$.
Granting this claim, by induction on the exponent of $p$ annihilating $M$,
we know $M_1$ satisfies the conclusion. 
Here, for the starting point of induction, we used the fact that a finitely generated $\s/p$-module
is $u$-torsion free if and only if it is free.
Then by \cite[Lemma 5.9]{Min20}, we get the conclusion for $M$.

We now verify the claim. Applying the snake lemma to
\[\xymatrix{
0 \ar[r] & M_1 \ar[r] \ar[d]^{\cdot p^n} & M \ar[r] \ar[d]^{\cdot p^n} & M/M_1 \ar[r] \ar[d]^{0} & 0 \\
0 \ar[r] & M_1 \ar[r] & M \ar[r] & M/M_1 \ar[r] & 0
}
\]
yields an exact sequence
\[
0 \to M/(M[p^n] + pM) \to M_1/p^nM_1 \to M/p^nM.
\]
Here $M[p^n]$ denotes the $p^n$-torsion in $M$.
Since $M/p^nM$ has no $u$-torsion by assumption, it suffices to show the same for $M/(M[p^n] + pM)$.
Applying \Cref{lem Zp module structure} gives the claim.

Next we turn to the general case. By \cite[Proposition 4.3]{BMS1}, we have two short exact sequences of generalized Breuil--Kisin
modules
\[
0 \to M_{\tor} \to M \to M_{\tf} \to 0
\]
and
\[
0 \to M_{\tf} \to M_{\free} \to M_0 \to 0.
\]
Here $M_{\tor}$ is the torsion submodule, $M_{\tf}$ is the torsion free quotient, $M_{\free}$ is the reflexive hull of $M$
(which is free as $\s$ is a $2$-dimensional regular Noetherian domain), and $M_0$ has finite length.
The first sequence implies that $M_{\tor}/p^n$ injects into $M/p^n$, therefore $M_{\tor}$ satisfies the assumption.
Since we have treated the torsion case, we see that $M_{\tor}$ satisfies the conclusion.
Now we claim $M_0$ vanishes. This immediately implies that $M_{\tf} = M_{\free}$ is free,
hence the first sequence splits, and $M = M_{\tor} \oplus M_{\tf}$ has shape of a $\mathbb{Z}_p$-module.

Finally let us justify the claim that $M_0 = 0$.
Take the second sequence above, derived modulo $p$ gives an inclusion
$M_0[p] \subset M_{\tf}/p$. Since $M_0$ has finite length, we see that $M_0[p]$ must be $u^{\infty}$-torsion.
If we can show that $M_{\tf}/p$ is $u$-torsion free, 
then we get $M_0[p] = 0$ which implies $M_0 = 0$ as it must be $p^{\infty}$-torsion.
We now reduce ourselves to showing $M/(M_{\tor} + p \cdot M)$ is $u$-torsion free.
Since $M_{\tor} = M[p^n]$ for sufficiently large $n$, we finish the proof by appealing to \Cref{lem Zp module structure}.
\end{proof}

\subsection{Breuil modules} Fix $0 \leq h \leq p-1$. 
 Let $S$ be the  $p$-adically completed PD-envelope
of $\theta: \s \twoheadrightarrow \O_K, u\mapsto \pi$, and for $i\ge 1$
write $\Fil^i S\subseteq S$ for the (closure of the) ideal generated by $\{ \gamma_n (E) = E^n/n!\}_{n\ge i}$.
For $i \le p-1$,  one has $\varphi(\Fil^i S) \subseteq p^i S$,
so we may define $\varphi_i:\Fil^i S\rightarrow S$ where $\varphi_i \coloneqq p^{-i}\varphi$. We have  $c_1 : = \varphi(E(u))/p \in S^\times$. Note that $S\subset K_0 [\![u]\!]$. Define $I_+ : = S \cap u K_0 [\![u]\!]$. Clearly $S/ I_+ = W(k)$. Let $S_n \coloneqq S/ p ^n S$. 
Let $^\sim\textnormal{Mod}^{\varphi, h }_{S}$ denote the category
whose objects are triples $(\mathcal M, \Fil^h \calM, \varphi_h)$, consisting of
\begin{enumerate}
	\item two $S$-modules $\calM$ and $\Fil^h \calM$;
	\item an $S$-module map $\iota: \Fil^h  \calM \to \calM $ whose image contains $\Fil ^h S \cdot \calM $; and
	\item a $\varphi$-semi-linear map $\varphi_h : \Fil ^h  \calM  \to \calM $ such that for all $s \in \Fil^h S$ and $x\in \calM$
	we have $$\varphi_h (sx)= (c_1)^{-h }\varphi_h (s) \varphi_h(E(u)^hx).$$
\end{enumerate}
Morphisms  are given by $S$-linear maps  compatible with $\iota$'s and commuting with
$\varphi_h$.
Let $'\textnormal{Mod}^{\varphi, h }_{S}$ denote the full subcategory of  $^\sim\textnormal{Mod}^{\varphi, h }_{S}$ whose objects $(\mathcal M, \Fil^h \calM, \varphi_h)$ satisfy 
\begin{enumerate}
	\item  $\iota$ is injective so that $\Fil ^h \calM$ is regarded as a submodule of $\calM$. 
	\item $\varphi_h (\Fil ^h \calM)$ generates $\calM$ as $S$-modules. 
\end{enumerate}
 A sequence is defined to be \emph{short exact}
if it is short exact as a sequence of $S$-module, and induces a
short exact sequence on $\Fil^h$'s. Let  $\textnormal{Mod}^{\varphi, h }_{S, \tor}$ denote the full subcategory of  $'\textnormal{Mod}^{\varphi, h  }_{S}$ so that
the underlying module $\calM$ is killed by a $p$-power and the triple $\calM$ can be a written as successive
extensions of triples $\calM_i$ in  $'\textnormal{Mod}^{\varphi, h }_{S}$
with each underlying module $\calM_i \simeq \bigoplus_{\text{finite}} S_1 $. 


Let 
$\nabla : S \to S$ be $W(k)$-linear continuous derivation so that $\nabla(u)= 1$. Let  $\Mod_{S, \tor}^{\varphi, h, \nabla}$ denote the category of the object $(\calM, \Fil ^h \calM, \varphi _h, \nabla) $ where $(\calM , \Fil^h \calM , \varphi_h)$
is an object in $\Mod_{S, \tor}^{\varphi, h}$ and $\nabla$ is 
$W(k)$-linear morphism $\nabla: \calM \to \calM $ such that :
\begin{enumerate}
	\item for all $s\in S $ and $x\in \calM$, $\nabla (sx)=\nabla(s) x + s \nabla(x)$.
    \item $E \nabla(\Fil^h \calM )\subset \Fil^h \calM $. 
    \item  the following diagram  commutes:
	\begin{equation}
	\begin{split}
	\xymatrix{ \Fil^h  \calM\ar[d]_{E(u) \nabla} \ar[r]^-{\varphi_h } & \calM \ar[d]^{c_1 \nabla}\\
		\Fil ^h  \calM   \ar[r]^-{ u ^{p-1} \varphi_h } &\calM }
	\end{split}
	\end{equation}
\end{enumerate}
An object $\calM $ in  $\Mod_{S, \tor}^{\varphi, h }$  is called a (torsion) \emph{Breuil module}. 

Now let us recall the relation of classical torsion Kisin modules and objects in $\Mod_{S, \tor}^{\varphi, h}$.   
For each such $\m\in \Mod_{\s , \tor}^{\varphi, h , {\rm c}}$, we construct an object $\calM:= \u\calM (\m) \in \Mod^{\varphi, h}_{S, \tor}$ as the following: $\calM: = S \otimes_{\varphi, \s} \m$ and
$$\Fil ^h \calM: = \{x \in \calM | (1\otimes \varphi_\m ) (x) \in \Fil ^h S \otimes_{\s}\m\}; $$
and $\varphi_h : \Fil ^h \calM \to \calM$ is defined as the composite of following map

\begin{equation*}
\xymatrix{
	{\Fil ^h \calM} \ar[r]^-{1\otimes \varphi_{\m}} & {\Fil ^h S \otimes_{\s}\m }
	\ar[r]^-{\varphi_h \otimes 1} & S\otimes_{\varphi,\s} \m = \calM
} .
\end{equation*}

For any $\calM \in \Mod _{S, \tor}^{\varphi, h}$, define a semi-linear $\varphi: \calM \to \calM$ by $\varphi (x) = (c_1)^{-h} \varphi_h (E^h x)$. Similar to the situation of Kisin module, we say $\calM$ is  \emph{multiplicative} (resp.~\emph{nilpotent}) if  $1 \otimes \varphi: S \otimes_{\varphi, S}\calM \to \calM$ is surjective (resp.~$\lim\limits_{n\to \infty} \varphi^n (x) = 0, \ \forall x \in \calM$). Clearly if $\m \in \Mod_{\s , \tor}^{\varphi , h , {\rm c}}$ is multiplicative (resp.~nilpotent) then so is $\u \calM (\m)$. 
\begin{remark}Here our definition of multiplicative is different from that in \cite[Def. 2.2.2]{GaoHui1} where $\calM$ is called multiplicative if $\Fil^h \calM = \Fil^h S \calM$. Indeed these two definitions are equivalent. Suppose that $\Fil^h \calM = \Fil^h S \calM$. Since  $\varphi_h ( a x)= \varphi_h (a) \varphi (x)$ for any $a \in \Fil^h S$ and $x\in \calM$,  $\{\varphi (x)= c_1^{-h}\varphi_h (E^h x)\}$ and $\{\varphi_h (\Fil^h S \calM)\}$ generates the same subsets in $\calM$. This implies that $\varphi (\calM)$ generates $\calM$. Conversely, suppose that $\varphi (\calM)$ generates $\calM$.  To show that $\Fil^h\calM = \Fil ^hS \calM$, we can reduce to the case that $\calM$ is finite $S_1$-free by d\'evissage. See the last part of proof of Lemma \ref{lem-uni-etale-Breuil}. 
\end{remark}


\begin{lemma}
\label{lem-uni-etale-Breuil} 
For any object $\calM \in \Mod^{\varphi, h}_{S, \tor}$, there exists a  functorial  exact sequence
\begin{equation}\label{eqn-multi-nil-seq-Breuil}
    \xymatrix{ 0 \ar[r] & \calM ^{\rm m } \ar[r]&   \calM  \ar[r]  & \calM^{\rm n}\ar[r] & 0}
\end{equation}
with $\calM^{\rm m}$  a multiplicative submodule of $\calM$ and $\calM^{\rm n}$ being nilpotent. 
\end{lemma}

\begin{proof} Recall $I_+ = S \cap u K_0 [\![u]\!]$,  $S/ I _+\simeq W(k)$ and $\varphi (x) = c_1^{-h} \varphi_h (E^h x)$. Write $S_n : = S/ p ^n S$ and assume that $\calM$ is an $S_n$-module. 
We claim Lemma \ref{lem-construct-Mm} still holds by  replacing $\m$ by $\calM$,  $M = \calM/ I_+$ and $q : \calM \to M = \calM / I _+$. Indeed, the same proof goes through because $\varphi^\ell (I_+)=0$ in $S_n$ for sufficient large $\ell$. 
Now we can set $\calM^{\rm m}$ be $S$-submodule of $\calM$ generated by $[M^{\rm m}]$ and $\calM^{\rm n} : = \calM / \calM^{\rm m}$. Using the same argument as in Lemma \ref{lem-multi}, the  right exact sequence $ S \otimes_{W(k)} [M^{\rm m}] \to \calM \to \calM^{\rm n} \to 0 $ modulo $I_+$ becomes an exact sequence $0 \to M^{\rm m} \to M \to M^{\rm n} \to 0$. This forces to that $\calM^{\rm m}/ I_+ = M ^{\rm m}$ and $\calM^{\rm n} / I_+ = M ^{\rm n }$.  
Set $\Fil ^h \calM ^m = \Fil ^h S \cdot \calM^{\rm m}$ and $\Fil ^h \calM^{\rm n} = \Fil^h \calM / \Fil^h \calM^{\rm m}$.
It is clear that $\varphi_h: \Fil ^h \calM ^{\rm m} \to \calM ^{\rm m}$ and $\varphi_h : \Fil ^h \calM ^{\rm n} \to \calM ^{\rm n }$ are well defined.
So we obtain an exact sequence
 $ 0 \to \calM ^{\rm m } \to \calM \to \calM^{\rm n} \to 0$ in the category $^\sim \Mod^{\varphi, h}_S $. 
 
 To promote our exact sequence to the category $\Mod^{\varphi , h}_{S, \tor}$, 
 we make induction on $n$ where $p ^n$ kills $\calM$. The base case $n =1$ is most complicated and postpone to the end. For general $n$, by definition, $\calM$ sits in the exact sequence in $\Mod^{\varphi , h}_{S, \tor}: \ 0 \to \calM_1 \to \calM \to \calM_2 \to 0 $ with $\calM_1$, $\calM_2$ killed by $p ^{n -1}$ and $p$ respectively. Consider the following commutative diagram
\begin{equation}\label{diagram-devissage}
    \xymatrix{ 0 \ar[r] & \calM_1^{\rm m}\ar[d] \ar[r]  & \calM_1 \ar[r]\ar[d] &\calM^{\rm n } _1 \ar[r]\ar[d]  & 0 \\  0 \ar[r] & \calM^{\rm m}\ar[d]^f \ar[r]  & \calM \ar[r]\ar[d]  & \calM ^{\rm n} \ar[r]\ar[d]  & 0  \\ 0 \ar[r] & \calM_2^{\rm m} \ar[r]  & \calM_2 \ar[r] & \calM _2^{\rm n} \ar[r]  & 0}
\end{equation}    
 We need to show that the first columns is short exact. 
 Note that $\calM _2$ is finite $S_1$-free, the exact sequence in the second column yields the exact sequence $ 0 \to M _1 \to M \to M_2 \to 0$ where $M_i : = \calM_i / I _+ \calM_i$ for $i= 1, 2$.  So the sequence $ 0 \to \calM_1^{\rm m } / I _+ \to \calM ^{\rm m }/I_+ \to \calM_2^{\rm m }/ I_+ \to 0$ is also exact as it is the same as the exact sequence $0 \to M_1^{\rm m} \to M ^{\rm m} \to M_2^{\rm m} \to 0$. Note that $\calM_i^{\rm m }$ is finite $S$-generated as they are generated by $[M_i^{\rm m}]$. 
 Note that $S_n$ is coherent ring, see \cite[Lemma 7.15]{LL20}. Induction on $n$ and \cite[\href{https://stacks.math.columbia.edu/tag/05CW}{Tag 05CW}]{stacks-project}, 
 we see that $\calM$ is coherent and then $\calM^{\rm m}$ is coherent. Since $\calM_1^{\rm m}$ is coherent by induction,  $\calL = \calM^{\rm m}/ \calM_1^{\rm m}$ is also coherent by \cite[\href{https://stacks.math.columbia.edu/tag/05CW}{Tag 05CW}]{stacks-project} again.
 Note $f$ induces a map $f':  \calL \to \calM_2 ^{\rm m}$. We need to show that $ f'$ is an isomorphism. Let $L = \calL / I_+ $. Note that $\bar f' : = f' \mod I_+ : L \to M^{\rm m}_2 $ is an isomorphism. Nakayama Lemma shows that $f'$ is surjective. Let $\mathcal K: = \ker(f')$ which is still coherent. Since $\calM^{\rm m}_2$ is finite $S_1$-free by induction, $\Tor_1^S (\calM^{\rm m }_2, S/I_+ )= 0$. So we obtain an exact sequence 
 $0 \to \mathcal K / I_+ \to L \to M_2^{\rm m} \to 0. $ Hence $\mathcal K /I_+ = 0$ as $\bar f'$ is an isomorphism. 
 By Nakayama lemma, $\mathcal K = 0$ and first column is exact as finite $S$-module. Using that $\calM_2^{\rm m}$ is finite $S_1$-free, we see that the sequence 
 $ 0 \to \calM_1^{\rm m}/ \Fil ^ h S \to \calM ^{\rm m}/\Fil^h S\to \calM_2^{\rm m}/ \Fil ^h S \to 0$ is exact. So the sequence 
 $0 \to \Fil ^ h S \cdot  \calM_1^{\rm m}  \to \Fil ^ h S \cdot  \calM ^{\rm m} \to \Fil ^ h S \cdot  \calM_2^{\rm m} \to 0$ is exact. Therefore, the  first column of \eqref{diagram-devissage} is exact in $\Mod_{S, \tor}^{\varphi, h}$. Then it is standard to check that last column is also exact sequence in  $'\Mod_{S, \tor}^{\varphi, h}$. In particular, $\calM^{\rm n}$ is an object in $\Mod_{S, \tor}^{\varphi, h}$ by induction on $n$.  Once \eqref{eqn-multi-nil-seq-Breuil} is exact in $\Mod_{S, \tor}^{\varphi, h}$. Then $\varphi$ on $\calM^{\rm m}$, $\calM$ and $\calM^{\rm n}$ defined from $\varphi (x) = c_1^{-h}\varphi _h (x)$ are compatible with maps in the sequence. Since 
 $\calM^{\rm m}$ is generated by $[M^{\rm m}]$ and $\calM^{\rm n}/ I_+ = M^{\rm n}$, we see that 
 $\calM ^{\rm m}$ is multiplicative and $\calM^{\rm n}$ is nilpotent. 
 
 Now we discuss the case $n =1$. First we have shown that $\calM^{\rm m}$ is finite $S$-generated as the above.  Now the exact sequence $0 \to \calM^{\rm m}/ I_+ \to \calM/ I _+ \to \calM^{\rm n}/ I_+ \to 0 $ is an exact sequence of $k$-vector spaces. Pick $m_i \in \calM ^{\rm m}$ and $n_j \in \calM$ so that $ m_i \mod I_+ $ and $n_j\mod I_+$ are basis of $\calM^{\rm m}/ I_+ $ and $\calM^{\rm n}/ I_+$ respectively. Using that $\calM$ is finite $S_1$-free. It is easy to show that $m_i , n_j $ forms a basis of $\calM$ and then both $\calM^{\rm m}$ and $\calM^{\rm n}$ are finite $S_1$-free. Now it remains to show that $\Fil^h \calM \cap \calM ^{\rm m} = \Fil^h S \calM^{\rm m}$ so that $\Fil^h \calM^{\rm n}= \Fil^h \calM / \Fil^h \calM^{\rm m}$ is a submodule of $\calM ^{\rm n}$. Then it is easy to check that $(\calM^{\rm n}, \Fil^h \calM^{\rm n}, \varphi_h)$ is a object in $\Mod_{S, \tor}^{\varphi, h}$ and thus the sequence $ 0 \to \calM ^{\rm m } \to \calM \to \calM^{\rm n} \to 0$ is in the category $\Mod^{\varphi, h}_{S, \tor}$. To show that $\Fil^h \calM^{\rm n}= \Fil^h \calM / \Fil^h \calM^{\rm m}$,  consider $\calF: = \calM ^{\rm m}/ \Fil^ p S_1\calM^{\rm m}$.  Write $\tilde \Fil^h \calF : = (\Fil^h \calM \cap \calM ^{\rm m }) / \Fil ^pS_1$ and $\Fil^h \calF = \Fil^h S \calM^{\rm m}/\Fil^p S_1$.  Since $\Fil^h \calF= u ^{eh}\calF \subset \tilde \Fil^h \calF \subset \calF$ which is a finite free $\ku /u^{pe}$-module. There exists a basis $e_1 , \dots , e_d$ of $\calF$ so that $\tilde \Fil^h \calF$ is generated by $u ^{a_i} e_i$ with $0 \leq a_i \leq eh$. Suppose one of $a_i < eh$. Say $a_1< eh$. Let $\hat e_i \in \calM^{\rm m}$ be a basis which lift $e_i$. Then $u^{a_1} \hat e_1 \in \Fil ^h \calM \cap \calM ^{\rm m}$. 
So $\varphi _h (u ^{eh} \hat e _1) = \varphi_h (u^{eh -a_1 } u ^a_i \hat e_1)= \varphi(u^{eh -a_1}) \varphi_h (u ^{a_1} \hat e_1)\in I_+ \calM$. This contradicts to  that $\varphi_h ( u ^{eh}\hat e_i)\mod I_+$ is a basis $M^m\subset M= \calM / I_+$. So all $a_i =eh$ and we have $\Fil^h \calM^{\rm m} = \Fil^h S \cdot \calM ^{\rm m}= \Fil^h \calM \cap \calM ^{\rm m}$ as required.   
\end{proof}
\begin{corollary}\label{cor-sequence-canonical}The exact sequence \eqref{eqn-multi-nil-seq-Breuil} is canonical in the sense of the following: Suppose $\calM$ admits another exact sequence in $\Mod_{S, \tor}^{\varphi, h}$: 
\[0 \to \tilde \calM^{\rm m} \to \calM \to \tilde \calM^{\rm n} \to 0\] with $\tilde \calM^{\rm m }$ being multiplicative and $\tilde \calM ^{\rm n}$ being nilpotent. Then  $\tilde \calM^{\rm m} = \calM^{\rm m}$ and $\tilde \calM ^{\rm n} = \calM^{\rm n}$. 
\end{corollary}
\begin{proof}
Since $\tilde \calM^{\rm n}$ is successive extension of finite free $S_1$-modules, $\Tor^1_{S} (\calM^{\rm n}, S/I_+) = 0$. Hence the sequence $0 \to \tilde \calM^{\rm m}/I_+ \to \calM/I_+ \to \tilde \calM^{\rm n}/I_+  \to 0 $ is exact. Since $\tilde \calM^{\rm m}$ is multiplicative, $\tilde \calM^{\rm m}/I_+ \subset M^{\rm m}$ and thus $\tilde \calM^{\rm m}/ I_+ = M^{\rm m}$ otherwise $\varphi$ on $\tilde \calM^{\rm n}/I_+$ can not be nilpotent. So $[M^{\rm m }]\subset \tilde \calM^{\rm m}$. Hence $\calM^{\rm m}\subset \tilde \calM^{\rm m}$ as $\calM^{\rm m}$ is constructed as $S$-submodule of $\calM$ generated by $[M^{\rm m}]$. Since $\calM^{\rm m }/I_+= M^{\rm m}$, we have $\tilde \calM^{\rm m}= \calM^{\rm m}$ by Nakayama's lemma. By the definition of exact sequence in the category $\Mod_{S, \tor}^{\varphi, h}$, we see that \[\Fil^h \tilde \calM^{\rm m} = \tilde \calM ^{\rm m} \cap \Fil^h \calM = \calM^{\rm m} \cap \Fil^h \calM = \Fil ^h \calM^{\rm m},\] 
where the last equality was proved by the end of the proof of \Cref{lem-uni-etale-Breuil}.
Therefore we have the desired equality
$(\tilde \calM^{\rm m}, \Fil^h \tilde \calM^{\rm m}, \varphi_h )=( \calM^{\rm m}, \Fil^h \calM^{\rm m}, \varphi_h )$ as sub-object of $\mathcal{M}$.
\end{proof}

\subsection{Fontaine--Laffaille modules} Fix $h = p -1$ for this subsection. Let us review Fontaine--Laffaille theory from \cite{Fontaine--Laffaille}. 
Let $\FL_{W(k)}$ denote the category whose objects are finite $W(k)$-modules $M$ together with decreasing filtration $\{\Fil ^i M\}_{i \geq 0}$ and Frobenius semi-linear map $\varphi_i: \Fil^i M \to M$ satisfying: 
\begin{enumerate}
    \item $\Fil^{i+1} M$ is a
direct summand of $\Fil ^i M$ for all $i\in \mathbb N$, and $\Fil^0 M = M$, $\Fil^{h+1} M =
\{0 \}$;\footnote{It turns out that this condition follows from the next two conditions, see \cite[Proposition 1.4.1 (ii)]{Wintenberger}.} 
\item $\varphi_i |_{\Fil^{i +1} M} = p \cdot \varphi_{i+1}$; 
\item $\sum_{i \geq 0} \varphi_i (\Fil ^i M) = M$. 
\end{enumerate}
Morphisms in $\FL_{W(k)}$ are $W(k)$-linear homomorphisms compatible with filtration and $\varphi_i$. 
It turns out that the category $\FL_{W(k)}$ is abelian, see \cite[Proposition 1.8]{Fontaine--Laffaille});
and any morphism is automatically strict with respect to the filtrations, see \cite[1.10 (b)]{Fontaine--Laffaille}.
A sequence $0 \to  M_1 \to  M \to  M_2 \to  0$ in $\FL_{W(k)}$
is \emph{short exact} if the underlying $W(k)$-module is 
exact\footnote{Note that by the above result of Fontaine--Laffaille, the sequence of filtrations are forced to be exact as well.}.
In this case, we call $M_2$ a quotient of $M$. 
An object $M\in \FL_{W(k)}$ is called  \emph{multiplicative} if $\Fil ^1M= \{0\}$ and   $M$ is called  \emph{nilpotent} if does not have  multiplicative subobject. 
Just as in previous sections, we have the following: 
\begin{lemma}
Let $(M, \Fil^{\bullet} M, \varphi_{\bullet}) \in \FL_{W(k)}$. 
\begin{enumerate}
\item It is multiplicative (resp.~nilpotent) if and only if $\varphi_0$ is bijective (resp.~nilpotent).
\item There is a canonical multiplicative-nilpotent exact sequence in $\FL_{W(k), \tor}$: 
\begin{equation}\label{eqn-seq-m-n-FM}
\xymatrix{0 \ar[r] & M ^{\rm m}\ar[r]  & M \ar[r] & M^{\rm n} \ar[r] & 0 }    
\end{equation}
so that $M^{\rm m}$ is the maximal multiplicative subobject in $M$ and $M ^{\rm n}$ is nilpotent.
\end{enumerate}
\end{lemma}

\begin{proof}
(1): the condition (3) of being an object in $\FL_{W(k)}$ in the case of a multiplicative object translates
to $\varphi_0$ being surjective, which is equivalent to being bijective due to length consideration.
Conversely, if $\varphi_0$ is bijective, we let $M' \in \FL_{W(k)}$ be defined as:
the underlying module is $M$ itself, with $\Fil^0 M' = M \supset \Fil^1 M' = 0$ and $\varphi_0$.
Then there is an evident morphism $M' \to M$ in $\FL_{W(k)}$, which is necessarily strict with respect to filtrations 
(see \cite[1.10 (b)]{Fontaine--Laffaille}),
hence $\Fil^1 M = \Fil^1 M' = 0$. The proof for nilpotent object is in end of the proof of (2). 

(2): By Fitting lemma, we have $M = M^{\rm m} \oplus  M^{\rm n}$, only as $\varphi$-modules,
so that $\varphi_0$ on $M ^{\rm m}$ is bijective and $\varphi_0$ on $M ^{\rm n}$ is nilpotent. 
Let $\Fil^1 M^{\rm m} = 0$, we get the desired sequence. The fact that the \emph{quotient} $M^{\rm n}$ with the induced filtration is nilpotent
follows from (1). By the exact sequence \eqref{eqn-seq-m-n-FM}, $M$ is nilpotent if and only if $M = M^{\rm n}$, whose $\varphi_0$ is nilpotent. 
\end{proof}

For any object $M$ in $\FL_{W(k)}$, we can attach a Breuil module $\u \calM_{\FL} (M) \in \Mod_{S \tor}^{\varphi, h, \nabla}$ in the following ways: 
Let $\calM = \u \calM_{\FL} (M) : = S \otimes_{W(k)} M$; $\nabla_{\calM} = \nabla_S \otimes \id_{M}$; 
$\Fil^h\calM : = \sum\limits_{i = 0}^h \Fil^i S \otimes_{W(k)} \Fil^{h-i} M$. 
By definition $\Fil^h \calM$ is a submodule of $\calM$.  
We define $\varphi_{h , \calM} : \Fil^h \calM \to \calM$ by $\varphi_{h , \calM} \coloneqq
\sum\limits_{i = 0}^h \big(\varphi_{i} \mid_{\Fil^i S}\big) \otimes \big(\varphi_{h-i} \mid_{\Fil^{h -i}M}\big)$,
this is well-defined because $\Fil^{i +1} M$ is a direct summand of $\Fil ^i M$.
It is standard to check that $\u\calM_{\FL} (M)$ is a Breuil module in $\Mod_{S, \tor}^{\varphi, h,  \nabla}$.  

\begin{proposition}\label{prop-same-uni} \begin{enumerate}
    \item Let $M \in \FL_{W(k), \tor}$. Then $\u\calM _{\FL} (\eqref{eqn-seq-m-n-FM})$ is isomorphic to \eqref{eqn-multi-nil-seq-Breuil} with $\calM = \u \calM (M)$. In particular, $\u\calM (M^{\rm m}) = \u \calM (M )^{\rm m}$. 
\item Given an $M \in \FL_{W(k), \tor}$ and suppose that there exists a classical Kisin module $\m \in \Mod^{\varphi, h, {\rm c}}_{\s , \tor}$ so that $\u \calM (\m)\simeq \u \calM_{\FL} (M)$ in the category of $\Mod_{S, \tor}^{\varphi, h}$. Then we have isomorphism 
$\u \calM_{\FL} (\eqref{eqn-seq-m-n-FM})\simeq \u \calM (\eqref{eqn-multi-nil-seq})$. In particular, $\u \calM (\m^{\rm n}) =\u \calM (\m)^{\rm n } = \u\calM_{\FL} (M^{\rm n })$. 
\end{enumerate}

\end{proposition}
\begin{proof}It is easy to check that if $M\in \FL_{W(k), \tor}$ (resp.~$\m \in \Mod_{\s, \tor}^{\varphi , h , {\rm c}}$) is multiplicative or nilpotent then so is $\u \calM_{\FL} (M)$ (resp.~$\u \calM (\m)$). Then the Proposition follows Corollary \ref{cor-sequence-canonical}.  
\end{proof}

For later use, let us prove the following technical lemma which says that one can test an object in $\FL_{W(k)}$ after 
looking at its ``Breuil's counterpart''. This is well-known to experts.

\begin{lemma}
\label{FL module equivalent to Breuil module}
Let $(M, \Fil^{\bullet} M, \varphi_{\bullet})$ be a filtered module with divided Frobenius,
namely only assuming the condition (2) in the definition of $\FL_{W(k)}$ is satisfied.
Let $\calM = \u \calM_{\FL} (M) : = S \otimes_{W(k)} M$ and
$\Fil^h\calM : = \sum\limits_{i = 0}^h \Fil^i S \otimes_{W(k)} \Fil^{h-i} M$.
Suppose there is a semi-linear map $\varphi_h \colon \Fil^h \calM \to \calM$
satisfying 
\[\varphi_{h} =\sum\limits_{i = 0}^h \big(\varphi_{i} \mid_{\Fil^i S}\big) \otimes \big(\varphi_{h-i} \mid_{\Fil^{h -i}M}\big).\]
Then $(M, \Fil^{\bullet} M, \varphi_{\bullet})$ is an object in $\FL_{W(k)}$ if and only if
$\varphi_h(\Fil^h \calM)$ generates $\calM$ as an $S$-module.
\end{lemma}

\begin{proof}
``Only if'' part follows from the standard direction of going from Fontaine--Laffaille modules to Breuil modules
as discussed above, below we prove the ``if'' part, which is the only part that will be used later.
To that end, we simply observe that $\calM/(p,I_+) \cdot \calM \cong M/p$.
One checks that the induced map $\overline{\varphi_h} \colon \Fil^h\calM \to M/p$ has image given by
the image of $\sum_{i=0}^{h} \overline{\varphi_i} \colon \bigoplus_{i=0}^{h} \Fil^i M \to M/p$.
Our condition now implies the reduction map is surjective. Since $M$ is $p$-adically complete,
it follows that the map $\sum_{i=0}^{h} \varphi_i \colon \bigoplus_{i=0}^{h} \Fil^i M \to M$
before mod $p$ is also surjective, which is exactly what we need to show.
\end{proof}

\subsection{Relations to Galois representations} Fix $\pi_n \in \overline K$ so that $\underline \pi : = (\pi _n) \in \O_{\C} ^\flat$ and $\pi_0= \pi$;  $K _\infty : = \bigcup_{n \geq 0} K (\pi _n)$ and $G_\infty : = \Gal (\Kbar/ K_\infty)$. We embed $\s \to A_{\inf}$ via $u \mapsto [\underline \pi]$. 
As discussed in \cite[\S 6.2]{LL20},  for a classical Kisin module $\m\in \Mod_{\s }^{\varphi , h}$, we can associate Galois representation of $G_\infty $ via 
$T_\s (\m) = (\m\otimes_\s W(\O_{\C}^\flat) )^{\varphi=1}$ and $T^h_\s (\m)=  (\Fil^h \varphi ^* \m \otimes A_{\inf})^{\varphi_h =1}$ where $\Fil^h \varphi^* \m : = \{ x \in \varphi^* \m | (1\otimes \varphi) (x) \in E^h \m \}$ and $\varphi_h : \Fil^h \varphi ^* \m \to \varphi^* \m$ is given by $\varphi_h (x) = \frac{(1\otimes \varphi)(x)}{\varphi (a_0^{-1}E) ^h}.$
Please consult \cite[\S 6.2]{LL20} for more details of $T^h _\s $ and $T_\s$, for example,  $T_\s^h (\m) = T _\s (\m) (h)$ and both $T_\s$ and $T^h _\s$ are exact. 

Note that if $\m \otimes _\s A_{\inf}$ has an $A_{\inf}$-semi-linear $G_K$-action which extends the natural $G_{\infty}$-action and commutes with $\varphi$, then $T_\s (\m)$ is a  $G_K$-representations.  
In particular, this is the case when $\m = \rH^i_\Prism (\cX/ \s_n)$ modulo $u^\infty$-torsion. 

Now given a Breuil module  $\calM\in \Mod_{S, \tor}^{\varphi, h, \nabla}$, 
then as explained around \cite[Eqn (6.19)]{LL20},  we 
define $\Fil^h (\calM \otimes _S A_{\cris} ): = \Fil^h \calM \otimes _S A_{\cris}$ then $\varphi_h $ extends to $\calM \otimes_ S A_{\cris}$  and  define a $G_K$-action on $\calM \otimes_S A_{\cris}$:
for any $\sigma \in
G_K$, 
any  $x \otimes a \in A_{\cris}\otimes_S \calM$,
define
\begin{equation} \label{eqn-G-action-by-nabla}
\sigma(x \otimes a)= \sum_{i=0}^\infty \nabla^i (x) \otimes  \gamma_i
\left ( \sigma ([\upi])- [\upi] \right ) \sigma (a) .
\end{equation}

The above $G_K$-action on $\calM \otimes_ S A_{\cris}$ extends the $G_\infty$-action, preserves filtration and commutes with $\varphi_h$. As in \cite[\S 6.3]{LL20}, we define $$T_S(\calM): = (\Fil^h (\calM \otimes_S A_{\cris}))^{\varphi _h= 1},$$ which is a $\Z_p [G_K]$-module. 

 
 Now given $\m \in \Mod_{\s, \tor}^{\varphi, h, {\rm c}}$ and  let $\frak c : = \prod\limits_{n =1}^\infty \varphi ^n ({\frac{E}{E(0)}})\in S^\times$.  
As explained in the proof of \cite[Prop. 6.12]{LL20}, the map $m \mapsto \frak c^h (1\otimes m)$ induces to natural map $\iota : T^h _\s (\m) \to T_S(\u \calM (\m))$. 

Suppose that $\m \otimes _\s A_{\inf}$ has $G_K$-action which extends $G_\infty$-action and commutes with $\varphi$,  and the natural map 
$\m \otimes_\s A_{\inf} \to \u \calM(\m) \otimes _S A_{\cris}$ is compatible with $G_K$-actions on both sides. Then as explained in  \cite[Remark 6.14]{LL20}, the natural map $T_\s (\m)(h) \simeq T_\s^h (\m) \overset{\iota}{\longrightarrow} T_S (\u \calM(\m))$ is compatible with $G_K$-actions on the both sides. In particular, this will  happen (see the proof of Theorem \ref{Thm-FM-p-1}) when $\m = \rH^i_ \qsyn (\cX, \Prism_n)$ is an object in $ \Mod_{\s, \tor}^{\varphi, h, {\rm c}}$ and $\u\calM (\m)$ is subobject  of $\rH^i _{\cris} (\cX/ S_n )$ inside $\Mod_{S, \tor}^{\varphi, h , \nabla}$.

\begin{lemma}\label{lem-Galois-unipotent}If $\m\in \Mod_{\s, \tor}^{\varphi, h,  {\rm c}}$ is  nilpotent then the natural map  $\iota: T^h _\s(\m) \to T_S (\u \calM (\m))$ is an isomorphism. 
\end{lemma}

\begin{proof} Write $\calM : = \u \calM (\m)$. Then $\calM$ is also nilpotent by Proposition \ref{prop-same-uni}. 
When $h \leq p-2$, $\iota$  is known to be isomorphism (without assuming  nilpotency of $\m$) by  \cite[Prop.6.12]{LL20}. So in the following, we assume $h = p-1$.   

Since $T^h_\s$ and $\u \calM$ exact and $T_S $ is left exact, 
we can assume that $\m$ is killed by $p$ so that $\m$ is finite free $\ku$-module with basis $e_1, \dots , e_d$. Write $\varphi (e_1 , \dots, e_d)= (e_1 , \dots, e_d)A$ with $AB = B A =a_0^{-h} u ^{eh}I_d$. Let $\tilde e_i : = 1 \otimes e_i $ be basis of $\varphi ^*\m$ and $S_1$-basis of $\calM$. Then $\Fil ^h \varphi^* \m$ is generated by $(\alpha_1 , \dots , \alpha_d) = (\tilde e_1 , \dots , \tilde e_d)B$ and $\Fil^h \calM$ is generated by $(\alpha_1 , \dots , \alpha _d)$ and $\Fil^p S_1 \calM$. Note that $\iota (\tilde e_1, \dots , \tilde e_d)= \frak c^h (\tilde e_1, \dots ,\tilde  e_d)$, and any $x \in (\Fil^h \varphi^*\m \otimes_\s A_{\inf})$ can be written as $x = (\alpha_1 ,\dots,  \alpha_d) X$ with $X\in (\O_\C ^\flat)^d$ and any $y \in \Fil^h \calM \otimes_S A_{\cris}$ can be written as $y = \frak c^h (\alpha _1 , \dots , \alpha_d) Y + \frak c^h (\tilde e_1 , \dots , \tilde e_d) Z$ with $Y \in (\O _\C^\flat/ u ^{ep})^d , Z \in (\Fil^pA_{\cris, 1}) ^d$. 
Then $\iota$ is the same as the following: 
\[\{X \in (\O_\C^\flat)^d | \varphi (X) = B X \}\longrightarrow \{(Y, Z)|Y \in (\O _\C^\flat/ u ^{ep})^d , Z \in (\Fil^pA_{\cris, 1}) ^d, \ \ \varphi (Y) + \varphi (A)\varphi _h (Z) = B Y +Z \}\]
by sending $X\mapsto (X, 0)$. We must show the above map is bijective. For injectivity, note that $X\in \ker (\iota)$ if and only if $B X \in (u^{pe}\O_\C^\flat )^d$. Then $ {a_0 ^{-h }} u ^{eh} X=AB X \in (u^{pe}\O_\C^\flat )^d $.  Hence $Y=a_0 u^{-e} X\in (\O_C^\flat)^d $. Note that $\varphi (X) = B X$ implies that $A\varphi (X)  = a_0^{-h}u^{eh} X $ and then  $Y  =  A \varphi (Y)$. 
So  $Y = A \varphi (A) \cdots \varphi^m (A) \varphi^{m +1} (Y)$. Since $\m$ is nilpotent, $A \varphi (A) \cdots \varphi^m (A)\to 0$ for $m \to \infty$, we see that $Y = 0$. This proves the injectivity of $\iota$.  

To prove the surjectivity of $\iota$, consider the equation $ \varphi (Y) + \varphi (A)\varphi _h (Z) - B Y =Z$. Note that  $A_{\cris, 1} = (\O_\C^\flat/ u ^{pe}) [\{z_i\}_{i \geq 1}]/ \{z_i ^p, i \geq 1\}$ 
with $z_i$ the image of $\gamma_{p ^i} (E) $ in $A_{\cris,1}$. Since $\varphi_h (z_i) = a_0^{p ^i}$ or $0$, the left side of equation is in $(\O_\C^\flat/u ^{pe})^d$, this forces the right side $Z= 0$ and we only have $\varphi (Y) = BY $ inside $ (\O_\C^\flat/u ^{pe})^d$.   So it suffices  to show   there exists $\tilde Y\in (\O_C^\flat) ^d $ so that $\varphi (\tilde Y) = B \tilde Y$ and $ B \tilde Y =  B Y $ inside $(\O _\C^\flat/ u ^{pe})^d$. To prove the existence of $\tilde Y$,  pick any lift $Y_0 \in (\O_\C^\flat)^d $ of $Y $. Then $\varphi (Y_0) = BY_0 + u ^{pe} W_0$. Since $u ^{pe} I_d = B A (a_0u ^e) ^h I _d$, we have $\varphi (Y_0)= B Y_1$ with $Y_1 = Y_0 + u^e A a^h_0W_0$. Then $\varphi (Y_1)= B Y_1 + u ^{pe}\varphi (A) W_1$ for some $W_1 \in (\O_\C^\flat)^d$. Continue construct $Y_n$ in this way, we have $\varphi (Y_{n}) = B Y_n + u ^{pe} A \varphi (A) \cdots \varphi ^n (A) W_n $ for some $W_n \in (\O_\C^\flat)^d $ and then $Y_{n+1} = Y_n + u^e A\varphi (A) \cdots \varphi^n (A) W_n$. Then $Y_n$ converges to $\tilde Y$ as  $A\varphi (A) \cdots \varphi^n (A)\to 0$. Since $\tilde Y = Y_0 + u ^e A \tilde W$ for some $\tilde W \in  (\O_C^\flat)^d$, we see that 
$B \tilde Y = B Y_0 = B Y $ inside $ (\O_\C^\flat / u ^{pe})^d$. 
\end{proof}
For $h = p-1$ the following example show that the above lemma will fail without $\m$ being nilpotent.


\begin{example}\label{example-h=p-1} Let $h = p -1$. 
Consider rank $1$-Kisin module  $\m = \s \cdot e_1$ and $\varphi (e_1)= e_1$. Then $\tilde e_1 = 1\otimes e_1 $ is a basis of $\varphi^*\m$ with $\Fil^h \varphi ^* \m = E^h \varphi^* \m$. We have $\calM = \u \calM (\m)= S \cdot \tilde e_1$ with $\Fil^h \calM = \Fil^h S \tilde e_1$ and $\varphi_h (x \tilde e_1) = \varphi_h (x)\tilde e_1, \forall x \in \Fil^h S$. 
Hence $$T^h_\s (\m ) = (E^h  A_{\inf})^{\varphi _h =1}  \tilde e_1 {= \{ E^h x\in E^h A_{\inf}| \varphi (x) = a_0^{-h} E^h x \}} \tilde e_ 1=   E^h\gt^h \Z_p \tilde e_1 . $$
Here $\gt \in A_{\inf}$ is  discussed in Example 3.2.3 in \cite{liu-notelattice} which also shows that $\varphi (\gt)= a_0^{-1}E\gt$ and $t =\mathfrak c \varphi(\gt)$. On the other hand,   $T_S (\calM)= (\Fil^h A_{\cris})^{\varphi _h =1} \tilde e_1=  \frac{t ^{h}}{p}\Z_p \tilde e_1.$ Tracing the definition of $\iota: T^h _\s (\m ) \to T_S (\calM)$, we see that $\iota (E^h \frak t^h \tilde e_1 ) = t ^h \Z_p \tilde e_1 \subset T_S(\calM)= \frac{t^h}{p} \Z_p \tilde e_1$. So $\iota$ is not a surjection in this case. 
By modulo $p^n $, we see $\ker (\iota) \cong \coker (\iota)$ is unramified and killed by $p$.  
\end{example}
\begin{corollary}\label{cor-diff-killed-byp}
  Let $h = p-1 $ and $\m\in \Mod_{\s, \tor}^{\varphi, h , {\rm c}}$ be a classical Kisin module of height $p -1$. Then the kernel
  and cokernel of $\iota: T^h_\s (\m) \to T_S(\u \calM (\m))$ are canonically isomorphic
  and are unramified representations killed by $p$. 
\end{corollary}
\begin{proof} 
Note that $T^h_\s$ is exact, see \cite[\S 6.2]{LL20}.
Since $T_S$ is clearly left exact,
by the exact sequence \eqref{eqn-multi-nil-seq} and Lemma \ref{lem-Galois-unipotent}, it suffices to prove Corollary for $\m$ be multiplicative.
Clearly, we can assume $k = \bar k$ and then $\m$ is direct sum of the $\s_n \cdot e_1 $ with $\varphi (e_1) =e_1$. 
Now our desired conclusion just follows from the above Example.  
\end{proof}

Finally, given a Fontaine--Laffaille module $M \in \FL_{W(k)}$. Set 
\[ T_{\FL} (M): = T _S(\u\calM_{\FL} (M))= \Fil^h (M \otimes_{W(k)} A_{\cris})^{\varphi_h=1}\]
where $\Fil ^h (M \otimes_{W(k)}A_{\cris})= \sum \limits_{i = 0 }^h \Fil^i M \otimes_{W(k)} \Fil^{h-i}A_{\cris}$. 

\section{Boundary degree prismatic cohomology}
\label{boundary degree}
\subsection{Structure of $u^\infty$-torsion}
Let $\mathcal{X}$ be a smooth proper formal scheme over $\mathcal{O}_K$ which is a degree $e$ totally ramified extension of $W= W(k)$,
the Witt ring of a perfect field $k$ of characteristic $p > 0$.
Let $\mathfrak{M}^i_n$ denote $\rH^i_{\qsyn}(\mathcal{X}, \Prism_n)[u^\infty]$, where $n = \infty$ shall be understood as
not modulo any power of $p$ at all.

\begin{proposition}
\label{prop control u-torsion}
We have the following restriction on the annihilator ideal of $\mathfrak{M}^i_n$:
\[
E^{i-1} \cdot {\rm{Ann}}(\mathfrak{M}^i_n) \subset {\rm{Ann}}(\varphi^*\mathfrak{M}^i_n) = {\rm{Ann}} (\mathfrak{M}^i_n) \otimes_{\s, \varphi_\s} \s.
\]
\end{proposition}

\begin{proof}
The equality follows from the flatness of $\varphi_{\s}$.
To show the inclusion, we first observe that 
\[
\varphi^*\mathfrak{M}^i_n = \left(\varphi^*\rH^i_{\qsyn}(\mathcal{X}, \Prism_n)\right)[u^\infty].
\]
Indeed, this follows from the fact that $\varphi_{\s}$ is flat and sends $u$ to $u^p$.

To finish the proof, it suffices to show that multiplication by $E^{i-1}$ on 
$\varphi^*\rH^i_{\qsyn}(\mathcal{X}, \Prism_n)$ factors through a submodule of
$\rH^i_{\qsyn}(\mathcal{X}, \Prism_n)$,
as then the same thing is true after replacing the above modules with their $u^{\infty}$-torsion submodules.
Now let us stare at the following diagram
\[
\label{diag control u-torsion}
\tag{\epsdice{1}}
\xymatrix{
\varphi^*\rH^i_{\qsyn}(\mathcal{X}, \Prism_n) \otimes_{\s} (E^{i-1}) \ar[r]
& \rH^i_{\qsyn}(\mathcal{X}, \Fil^{i-1}_{\rN}/p^n) \ar[r] \ar@{_{(}->}[d]_{\varphi_{i-1}}
& \varphi^*\rH^i_{\qsyn}(\mathcal{X}, \Prism_n) \\
& \rH^i_{\qsyn}(\mathcal{X}, \Prism_n).
}
\]
Here the top row is given by (mod $p^n$ of) the following inclusions of quasi-syntomic sheaves 
\[
\Prism^{(1)} \otimes_{\s} (E^{i-1}) \subset \Fil^{i-1}_{\rN} \subset \Prism^{(1)},
\]
and $\varphi_{i-1}$ is the divided Frobenius.
Finally apply \cite[Lemma 7.8.(3)]{LL20}, we see that $\varphi_{i-1}$ is injective in degree $i$.
\end{proof}

\begin{corollary}
\label{cor exponent inequality}
Let $\alpha \in \mathbb{Z}_{\geq 0}$ satisfy ${\rm Ann} (\mathfrak{M}^i_n) + (p) = (u^\alpha, p)$, then we have
\[
\alpha \leq \frac{e(i-1)}{p-1}.
\]
\end{corollary}

\begin{proof}
Using \Cref{prop control u-torsion}, after modulo $(p)$, we get the inclusion
\[
E^{i-1} \cdot (u^\alpha) \subset \varphi^*(u^\alpha) = (u^{p\alpha})
\]
in $\s/p = k[\![u]\!]$.
Since $E \equiv u^e$ modulo $p$, the above inclusion translates to the inequality
\[
p \alpha \leq e(i-1) + \alpha
\]
which is exactly what we need to show.
\end{proof}

Later on we shall exhibit examples showing that the above bound is sharp, see \Cref{rmk example} (1).
Now let us conclude our current knowledge on the $u^{\infty}$-torsion submodules in prismatic cohomology.

\begin{theorem}
\label{cor control u-torsion}
Recall $\mathfrak{M}^i_n \coloneqq \rH^i_{\qsyn}(\mathcal{X}, \Prism_n)[u^\infty]$.
\begin{enumerate}
    \item If $e \cdot (i-1) < p-1$, then $\mathfrak{M}^i_n = 0$.
    \item If $e \cdot (i-1) < 2(p-1)$, then ${\rm Ann}(\mathfrak{M}^i_n) + (u) \supset (p^{i-1},u)$.
    \item If $e \cdot (i-1) = p-1$, then ${\rm Ann}(\mathfrak{M}^i_n) \supset (p,u)$.
    Moreover the semi-linear Frobenius on $\mathfrak{M}^i_n$ is bijective.
    In particular $\mathfrak{M}^i_n$ gives rise to an \'{e}tale $\varphi$-module on $k$, hence an $\mathbb{F}_p$-representation
    of $G_k$ or equivalently an unramified $\mathbb{F}_p$-representation of $G_K$.
\end{enumerate}
\end{theorem}

Later we shall give an interpretation of the $G_k$-representation in (3) above, see \Cref{ker Cosp and u-torsion}
and \Cref{control ker Cosp}.

\begin{proof}
In the situation of (1), the inequality in \Cref{cor exponent inequality} gives $\alpha = 0$,
hence ${\rm Ann}(\mathfrak{M}^i_n) + (p)$ is the unit ideal. Since $p$ is topologically nilpotent,
this shows that ${\rm Ann}(\mathfrak{M}^i_n)$ is already the unit ideal, hence $\mathfrak{M}^i_n = 0$.

In the situation of (2), the inequality in \Cref{cor exponent inequality} gives $\alpha < 2$.
Therefore we have either $\mathfrak{M}^i_n = 0$ or ${\rm Ann}(\mathfrak{M}^i_n) + (p) = (u,p)$.
Without loss of generality, we may assume $\mathfrak{M}^i_n \not= 0$ and ${\rm Ann}(\mathfrak{M}^i_n) + (p) = (u,p)$.
Let us pick an element $f = u + a \in {\rm Ann}(\mathfrak{M}^i_n)$ with $a \in p \cdot W(k)$,
whose existence is guaranteed by our assumption that ${\rm Ann}(\mathfrak{M}^i_n) + (p) = (u,p)$.
Let us compute:
\[
E^{i-1} \cdot f = (E(u))^{i-1} \cdot (u + a) = (u^{e \cdot (i-1)} + \ldots + a_1 \cdot u + a_0) \cdot (u+a)
= \sum_{j = 0}^{p-1} \varphi_{\s}(B_j) \cdot u^j.
\]
\Cref{prop control u-torsion} implies that all of $B_i$'s are in $\rm Ann(\mathfrak{M}^i_n)$.
Let us contemplate $C_1 = B_1(0)$: the above equation says $\varphi(C_1) = a_1 \cdot a + a_0$.
Since we know $v_p(a_1) \geq i-1$ and $v_p(a_0) = i-1$, we see that $v_p(C_1) = i-1$,
which implies $(B_1) + (u) \supset (u, p^{i-1})$.

Lastly we turn to (3). Similarly argued as above,
we may assume $\mathfrak{M}^i_n \not= 0$ and ${\rm Ann}(\mathfrak{M}^i_n) + (p) = (u,p)$,
and our first task is to show $u \in {\rm Ann} (\mathfrak{M}^i_n)$.
To that end, pick again an element $f = u + a \in {\rm Ann}(\mathfrak{M}^i_n)$ 
with $a \in p \cdot W(k)$.
Next we compute
\[
E^{i-1} \cdot f = (u^e + p \cdot g_1)^{i-1} \cdot (u + a)
= (u^{p-1} + p \cdot g_2) \cdot (u + a) = (u^p + p^{i-1} E(0)^{i-1} \cdot a) \cdot 1
+ \sum_{j = 1}^{p-1} b_j \cdot u^j 
\]
By \Cref{prop control u-torsion}, we see that another element of the form $u + b \in {\rm Ann}(\mathfrak{M}^i_n)$ with $b \in W(k)$
having a bigger $p$-adic valuation than that of $a$.
Consequently we have $u \in {\rm Ann} (\mathfrak{M}^i_n)$, as $a-b$ and $a$ differ by a unit in $W(k)$.
{
Now we do the trick again:  
\[
E(u)^{i-1} \cdot u = (u ^e + p g(u)) ^ {i -1} \cdot u = u^p + \sum_{j = 1}^{i -1} {i-1 \choose j} u ^{1+ e (i -1 -j)} (p g(u)) ^{j} = u ^p + \sum_{j = 1}^{p-1} B_j u^j
\]
with $B_j\in W(k)$. Since $u ^p \in {\rm Ann} (\varphi^* \m_n ^i) $, we see that $\sum\limits_{j = 1}^{p-1}  B_j u^j  \in {\rm Ann} (\varphi^* \m_n ^i)$ and hence each 
$\varphi^{-1}(B_j)$ is in ${\rm Ann} (\m_n ^i)$. 
From the above expansion, we see that
\[
E(u)^{i-1} \cdot u \equiv u^p + (i-1) u ^{1+ e (i - 2)} (p g(u)) \mod p^2. 
\]
 Since $E(u)$ is an Eisenstein polynomial, we see that $g(0)$
is a $p$-adic unit.
This implies that $v_p (B_{1+ e (i-2)}) =1$, so $p \in {\rm Ann} (\m_n ^i)$. }

Lastly, we need to show the semi-linear Frobenius on $\mathfrak{M}^i_n$ is a bijection.
Previous paragraph tells us that $\mathfrak{M}^i_n \simeq k^{\oplus r}$.
Let us look at the $u^{\infty}$-torsion part of diagram~\ref{diag control u-torsion}
\[
\xymatrix{
\varphi^*\mathfrak{M}^i_n \otimes_{\s} (E^{i-1}) \ar[r] \ar@{->>}[d]
& \rH^i_{\qsyn}(\mathcal{X}, \Fil^{i-1}_{\rN}/p^n)[u^{\infty}] \ar[d] \ar@{^{(}->}[r]^-{\varphi_{i-1}}
& \mathfrak{M}^i_n \\
E^{i-1} \cdot \varphi^*\mathfrak{M}^i_n \ar@{^{(}->}[r] & \varphi^*\mathfrak{M}^i_n.  & 
}
\]
We claim the first arrow in the top row is surjective, the middle vertical arrow
is injective with image being $E^{i-1} \cdot \varphi^*\mathfrak{M}^i_n$,
and the map $\varphi_{i-1}$ is an isomorphism.
We know $\varphi^*\mathfrak{M}^i_n \simeq (k[u]/u^p)^{\oplus r}$, hence $E^{i-1} \cdot \varphi^*\mathfrak{M}^i_n$
is also abstractly isomorphic to $k^{\oplus r}$. Let $\ell (\cdot)$ denote the $k$-length. 
The above diagram gives a chain of inequality of lengths
\[
r  \leq \ell(\rH^i_{\qsyn}(\mathcal{X}, \Fil^{i-1}_{\rN}/p^n)[u^{\infty}])
\leq r = \ell(\mathfrak{M}^i_n),
\]
where the first inequality follows from previous sentence.
So the above inequalities are both equalities, and the claim follows easily.
The composition, which we have shown to be surjective, of
\[
\mathrm{Frob}_k^*\mathfrak{M}^i_n \xrightarrow{- \otimes_{\s} (E^{i-1})} \varphi^*\mathfrak{M}^i_n \otimes_{\s} (E^{i-1})
\to \rH^i_{\qsyn}(\mathcal{X}, \Fil^{i-1}_{\rN}/p^n)[u^{\infty}] \xrightarrow{\varphi_{i-1}} \mathfrak{M}^i_n
\]
is the linearization of the semi-linear Frobenius on $\mathfrak{M}^i_n$.
This shows that the semi-linear Frobenius on $\mathfrak{M}^i_n$ is surjective, hence bijective by length/dimension
considerations.
\end{proof}

Below let us remark on results in literature concerning $u^\infty$-torsion in Breuil--Kisin prismatic cohomology.

\begin{remark}
\leavevmode
\begin{enumerate}
\item Under the assumption $e \cdot i < p-1$, Min \cite[Theorem 0.1]{Min20} showed that the $i$-th prismatic cohomology
has no $u$-torsion and ``looks like'' the \'{e}tale cohomology of the geometric generic fibre.
His strategy is to exploit the fact that Frobenius map in degree $i$ has height $i$.
Note that his method also shows that in the same range, the $i$-th (derived) mod $p^n$ prismatic cohomology
also has no $u$-torsion.
But as far as we can tell, the method stops outside the above range.
\item Philosophically speaking, the $u^{\infty}$-torsion in $i$-th (derived) mod $p^n$ prismatic cohomology
is surjected on by $(i-1)$-st cohomology of the sheaf ${\Prism}_n/u^N$
for some large $N$, hence it should secretly have height $(i-1)$.
Our \Cref{prop control u-torsion} may be taken as a manifestation of this philosophy.
Later on we show this philosophy is literally true for $u^\infty$-torsion in
the integral prismatic cohomology, see \Cref{induced Vers Corollary 2}.
\item In our previous work, we showed a close relation between $u$-torsion in prismatic cohomology
and structure of Breuil's crystalline cohomology~\cite[Theorem 7.22]{LL20}. 
Using this relation, together with Caruso's result \cite[Theorem 4.1.24]{CarusoInvent},
we obtained the same conclusion as in \Cref{cor control u-torsion}.(1) 
and an improvement 
of Caruso's result~\cite[Theorem 4.1.24 and Theorem 4.2.1]{CarusoInvent}, see \cite[Corollary 7.25]{LL20}.
Note that our bound on the cohomological index is $1$ higher than Caruso's result.
\item Our control of $u$-torsion in this paper bypasses Caruso's result. 
Hence we obtain a proof of Caruso's result and its improvement simultaneously,
c.f.~\cite[Theorem 7.22 and Corollary 7.25]{LL20}.
\item Later on, we shall see that our bound is in some sense sharp by exhibiting an example having $(u,p)$-torsion with $e = p-1$ and $i = 2$.
See \Cref{subsection example}.
\end{enumerate}
\end{remark}

Let us give an application by showing the module structure of prismatic cohomology in low range looks like
a $\mathbb{Z}_p$-module.

\begin{corollary}
\label{finite free in low degree and ramification}
Let $i$ be an integer satisfying $e \cdot (i-1) < p-1$. 
Then there exists a \emph{(}non-canonical\emph{)} isomorphism of $\s$-modules
\[
\rH^i_{\Prism}(\mathcal{X}/\s) \simeq \rH^i_{\et}(\mathcal{X}_C, \mathbb{Z}_p) \otimes_{\mathbb{Z}_p} \s.
\]
\end{corollary}

\begin{proof}
Since we always have inclusions
$\rH^i_{\Prism}(\mathcal{X}/\s)/p^n \subset \rH^i_{\qsyn}(\mathcal{X}, \Prism_n)$.
In the specified range, we know the latter has no $u$-torsion by \Cref{cor control u-torsion}.(1).
Applying \Cref{prop Zp module structure} shows that there exists an isomorphism of $\s$-modules
\[
\rH^i_{\Prism}(\mathcal{X}/\s) \simeq N_i \otimes_{\mathbb{Z}_p} \s
\]
for some $\mathbb{Z}_p$-module $N_i$. To obtain $N_i \simeq \rH^i_{\et}(\mathcal{X}_C, \mathbb{Z}_p)$,
we simply use the \'{e}tale comparison of Bhatt--Scholze, see \cite[Theorem 1.8.(iv)]{BMS1} and \cite[Theorem 1.8.(4)]{BS19}.
Here we are using the fact that the isomorphism class of a finitely generated $\mathbb{Z}_p$-module
is determined by its base change to $W(C^{\flat})$.
\end{proof}

One should compare with Min's result \cite[Theorem 5.11]{Min20}.
Our bound on the cohomological degree $i$ here is $1$ better than Min's.
Below we remind readers a useful result in \cite{BMS1} assuring nice behavior of prismatic cohomology
when crystalline cohomology has no torsion,
which is a condition often summoned in literature.

\begin{remark}
\label{finite free follows from crystalline torsion free}
If $\rH^i_{\cris}(\mathcal{X}_0/W)$ is torsion free, then $\rH^i_{\Prism}(\mathcal{X}/\s)$ is free.
This follows from \cite[Corollary 4.17]{BMS1} and crystalline comparison.
We sketch a proof below, see also \cite[Lemma 4.3.28.(1)]{KP21}.
\end{remark}

\begin{proof}
Let us denote $\m \coloneqq \rH^i_{\Prism}(\mathcal{X}/\s)$,
we shall use the two short exact sequences appeared in the proof of \Cref{prop Zp module structure}.
Crystalline comparison implies $\m/u \m \hookrightarrow \rH^i_{\cris}(\mathcal{X}_0/W)$.
Let us derived modulo the sequence
\[
0 \to \m_{\tor} \to \m  \to \m_{\tf} \to 0
\]
by $u$. 
Since $\m_{\tf}$ is torsion free, we have $\m_{\tor}/u \hookrightarrow \m/u$.
The target is $p$-torsion free by assumption whereas $\m_{\tor}$ consists of $p$-power torsion,
therefore $\m_{\tor} = 0$ and $\m_{\tf} = \m$.
Now we again derived modulo the sequence
\[
0 \to \m  \to \m _{\free} \to \m_0 \to 0
\]
by $u$ to get
$\m_0[u] \hookrightarrow \m /u$, same argument as above shows $\m_0[u] = 0$ whereas $\m_0$ is supported
at the maximal ideal of $\s$.
Therefore we again conclude $\m_0 = 0$ and $\m = \m_{\free}$.
\end{proof}

Let us conclude this subsection by asking some questions.

\begin{question}
\label{general question}
Recall $\m^i_n \coloneqq \rH^i_{\qsyn}(\mathcal{X}, \Prism_n)[u^\infty]$.
\begin{enumerate}
\item Let $\beta$ be the smallest exponent such that $p^\beta \in \mathrm{Ann}(\m^i_n)$,
let $\gamma$ be the exponent such that ${\rm Ann}(\m^i_n) + (u) = (u, p^{\gamma})$.
Is there a bound on $\beta$ and $\gamma$ in terms of $e$ and $i$?
\item In light of the example in \Cref{subsection example}, is $\beta$ and/or $\gamma$ bounded above by a polynomial in
$\log_p$ of a polynomial in $e  \text{ and } i$,
maybe simply bounded above by
$\log_p(\frac{e \cdot (i-1)}{p-1}) + 1$ when $p$ is odd?\footnote{After contemplating with
the image of Whitehead's $J$-homomorphism, we suspect the above bound should be up by $1$ when $p=2$ and $e \cdot (i-1) \geq 2$.}
\end{enumerate}
\end{question}

\subsection{Comparing Frobenius and Verschiebung}

Given a smooth proper formal scheme $\mathcal{X}$ over $\mathcal{O}_K$, for each degree $i$, we have a  natural inclusion
$\rH^i_{\Prism}(\mathcal{X}/\s)^{(1)}/u \hookrightarrow \rH^i_{\cris}(\mathcal{X}_0/W)$
coming from the crystalline comparison of prismatic cohomology theory.
Here the supscript $(-)^{(1)}$ denotes the Frobenius twist, 
so 
\[
\rH^i_{\Prism}(\mathcal{X}/\s)^{(1)} \coloneqq \varphi^*_{\s}\rH^i_{\Prism}(\mathcal{X}/\s) \cong \rH^i_{\qsyn}(\mathcal{X}, \Prism^{(1)}).
\]
The map is compatible with Frobenius and Verschiebung, hence induces Frobenius and Verschiebung maps
on the cokernel $\rH^{i+1}_{\Prism}(\mathcal{X}/\s)^{(1)}[u]$.
How to understand these maps? That is the question we shall answer in this subsection.

Given any algebra $R$ which is quasi-syntomic over $\mathcal{O}_K$, we may take its mod $\pi$ reduction $R_0$
which is quasi-syntomic over $k$.
This way we obtain a natural map of sites $i \colon k_{\qsyn} \to (\mathcal{O}_K)_{\qsyn}$.
Note that the functor
$R_0 \mapsto \Cris(R_0/W)$ is a quasi-syntomic sheaf on $k_{\qsyn}$.
Here by abuse of notation we use $\Cris(R_0/W)$ to denote the left Kan extended crystalline cohomology.
The sheaf $i_*\Cris$ takes an algebra $R$ in $(\mathcal{O}_K)_{\qsyn}$ to $i_*\Cris(R) \coloneqq \Cris(R_0/W)$.
The base change property and the crystalline comparison of prismatic cohomology \cite[Theorem 1.8.(1)\&(5)]{BS19} gives us the following exact triangles
of sheaves on $(\mathcal{O}_K)_{\qsyn}$:
$$\Prism \xrightarrow{\cdot u} \Prism \to i_*\Cris^{(-1)}$$
and
$$\Prism^{(1)} \xrightarrow{\cdot u} \Prism^{(1)} \to i_*\Cris$$
where $i_*\Cris^{(-1)}(R) = \Cris(R_0/W) \otimes_{W, \varphi^{-1}} W$ is the Frobenius inverse twist.

\begin{proposition}
\label{induced Frob}
\leavevmode
\begin{enumerate}
\item The linear Frobenius maps $\rH^i_{\Prism}(\mathcal{X}/\s)^{(1)} \to \rH^i_{\Prism}(\mathcal{X}/\s)$
and $\rH^i_{\cris}(\mathcal{X}_0/W) \to \rH^i_{\cris}(\mathcal{X}_0/W)^{(-1)}$
induces a linear map $\rH^{i+1}_{\Prism}(\mathcal{X}/\s)^{(1)}[u] \to \rH^{i+1}_{\Prism}(\mathcal{X}/\s)[u]$
which agrees with the linear Frobenius $\rH^{i+1}_{\Prism}(\mathcal{X}/\s)^{(1)} \to \rH^{i+1}_{\Prism}(\mathcal{X}/\s)$
restricted to $u$-torsion.
\item The semi-linear Frobenius maps 
$\rH^i_{\Prism}(\mathcal{X}/\s)^{(1)} \to \rH^i_{\Prism}(\mathcal{X}/\s)^{(1)}$
and
$\rH^i_{\cris}(\mathcal{X}_0/W) \to \rH^i_{\cris}(\mathcal{X}_0/W)$
induces a semi-linear map 
$\rH^{i+1}_{\Prism}(\mathcal{X}/\s)^{(1)}[u] \to \rH^{i+1}_{\Prism}(\mathcal{X}/\s)^{(1)}[u]$.
This map is $u^{p-1}$ times the semi-linear Frobenius on $\rH^{i+1}_{\Prism}(\mathcal{X}/\s)^{(1)}$
restricted to $u$-torsion.
\end{enumerate}
\end{proposition}

Note that semi-linearity means $u$-torsion are only sent to $u^p$-torsion under the semi-linear Frobenius,
after multiplying $u^{p-1}$ we land in $u$-torsion again.

\begin{proof}
Below we use $\rm{lin-Frob}$ (resp.~$\rm{sl-Frob}$) to denote the linearized Frobenius
(resp.~semi-linear Frobenius) on $\Prism^{(1)}$.

(1): this follows from the following commutative diagram
\[
\xymatrix{
\Prism^{(1)} \ar[r]^{\cdot u} \ar[d]_{{\rm lin-Frob}} & \Prism^{(1)} \ar[r] \ar[d]_{\rm lin-Frob} & i_*\Cris \ar[d]^{i_*(\rm lin-Frob)} \\
\Prism \ar[r]^{\cdot u} & \Prism \ar[r] & i_*\Cris^{(-1)}.
}
\]

(2):  this follows from the following analogous commutative diagram
\[
\xymatrix{
\Prism^{(1)} \ar[r]^{\cdot u} \ar[d]_{u^{p-1} \cdot ({\rm sl-Frob})} & \Prism^{(1)} \ar[r] \ar[d]_{\rm sl-Frob} & i_*\Cris \ar[d]^{i_*({\rm sl-Frob})} \\
\Prism^{(1)} \ar[r]^{\cdot u} & \Prism^{(1)} \ar[r] & i_*\Cris.
}
\] 
\end{proof}

\begin{remark}
Comparing the above two formulas, the appearance of extra $u^{p-1}$ factor has a natural explanation.
Let $M$ be an $\s$-module, then by flatness of $\varphi_{\s}$ we know
$\big(M \otimes_{\s, \varphi_\s} \s\big) [u] \cong \big(M[u] \otimes_{\s, \varphi_\s} \s\big) [u]$.
We may expand the right hand side as $\big(M[u] \otimes_{W, \varphi_W} W\big) \otimes_W (\s/u^p [u])$.
Under this identification, one checks that there is a semi-linear bijection:
$M[u] \xrightarrow{\simeq} \big(M[u] \otimes_{W, \varphi_W} W\big) \otimes_W (\s/u^p [u])$
given by $m \mapsto (m \otimes 1) \otimes u^{p-1}$.
Applying this to $M = \rH^{i+1}_{\Prism}(\mathcal{X}/\s)$ gives
the relation between (1) and (2) above.
\end{remark}

Next we turn to understand the map on $\rH^{i+1}_{\Prism}(\mathcal{X}/\s)^{(1)}[u]$ induced from
Verschiebung maps.
We need the following fact on Nygaard filtration.
\begin{lemma}
\label{induced Vers Lemma 1}
The divided Frobenius $\varphi_{i-1} \colon \Fil^{i-1}_{\mathrm{N}} \to \Prism$ induces an isomorphism
\[
\varphi_{i-1} \colon \rH^{i}_{\qsyn}(\mathcal{X}, \Fil^{i-1}_{\mathrm{N}})_{\mathrm{tors}} \xrightarrow{\cong} \rH^{i}_{\Prism}(\mathcal{X})_{\mathrm{tors}}.
\]
\end{lemma}

\begin{proof}
Note that we have a commutative diagram of quasisyntomic sheaves:
\[
\xymatrix{
\Fil^{i-1}_{\mathrm{N}} \otimes_{\s} (E) \ar[rr]^{\rm incl} \ar[rd]_{\varphi_{i-1}} & & \Fil^{i}_{\mathrm{N}} \ar[ld]^{\varphi_{i}} \\
& \Prism &
}
\]
By \cite[Lemma 7.8.(3)]{LL20} we know the $i$-th divided Frobenius map in degree $i$ is an isomorphism
for any bounded prism.
Therefore we only need to show the map
$\rH^{i}_{\qsyn}(\mathcal{X}, \Fil^{i-1}_{\mathrm{N}}) \otimes_{\s} (E) \to \rH^{i}_{\qsyn}(\mathcal{X}, \Fil^{i}_{\mathrm{N}})$
induces an isomorphism on the torsion submodule.

We claim these modules have the property that their torsion submodule coincides with $p^\infty$-torsion submodule.
To see this, just use the fact that both $\varphi_{i-1}$ and $\varphi_i$ are injective in degree $i$, thanks to \cite[Lemma 7.8.(3)]{LL20}.
The torsion submodule in prismatic cohomology is well-known to coincide with $p^\infty$-torsion submodule.

Therefore we are reduced to showing the above map induces an isomorphism on $p^\infty$-torsion submodule.
To that end, we use the exact sequence of quasisyntomic sheaves: $\Fil^{i-1}_{\mathrm{N}} \otimes_{\s} (E) \to \Fil^{i}_{\mathrm{N}} \to \Fil^i_{\rH}$.
Lastly just note that $\rH^i(\mathcal{X},\Fil^i_{\rH}) \cong \rH^0(\mathcal{X}, \Omega^i_{\mathcal{X}/\mathcal{O}_K})$ 
is $p$-torsion free.
\end{proof}

\begin{corollary}
\label{induced Vers Corollary 1}
If $e \cdot (i-1) = p-1$, then the map
${\rm incl} \colon \rH^{i}_{\qsyn}(\mathcal{X}, \Fil^{i-1}_{\mathrm{N}})[u] \to \rH^{i}_{\Prism}(\mathcal{X}/\s)^{(1)}[u]$ 
is an isomorphism.
\end{corollary}

\begin{proof}
We stare at the following diagram and contemplate taking $\rH^i$:
\[
\xymatrix{
& \Prism^{(1)} \otimes (E)^{\otimes i-1} \ar[ld]_{\varphi \otimes id} \ar[d]^{\rm incl} \ar[rd]^{\rm incl} & \\
\Prism \otimes (E)^{\otimes i-1} & \Fil^{i-1}_{\mathrm{N}} \ar[l]_-{\varphi} \ar[r]^{\rm incl} & \Prism^{(1)}.
}
\]
By \Cref{cor control u-torsion} (3), applied to $n=\infty$, combined with \Cref{induced Vers Lemma 1} we know that the map
\[
{\rm incl} \colon \big(\rH^{i}_{\Prism}(\mathcal{X}/\s)^{(1)}[u^{\infty}]\big)/u \otimes (E)^{\otimes i-1}
\to \rH^{i}_{\qsyn}(\mathcal{X}, \Fil^{i-1}_{\mathrm{N}})[u]
\]
is an isomorphism.
Using \Cref{cor control u-torsion} (3) again we know the map
\[
{\rm incl} \colon \big(\rH^{i}_{\Prism}(\mathcal{X}/\s)^{(1)}[u^{\infty}]\big)/u \otimes (E)^{\otimes i-1}
\to \rH^{i}_{\Prism}(\mathcal{X}/\s)^{(1)}[u]
\]
is also an isomorphism. Therefore we get the desired result.
\end{proof}

The relevance of Nygaard filtration when discussing the Verschiebung map follows from \cite[Corollary 7.9]{LL20}.
We recall its statement below:
\begin{lemma}
\label{induced Vers Lemma 2}
Let $(A,I)$ be a bounded prism, and let $\mathcal{X}$ be a smooth formal scheme over $\Spf(A/I)$.
The $i$-th Verschiebung map \emph{(}see \cite[Corollary 15.5]{BS19}\emph{)}
\[
V_i \colon \tau^{\leq i}\Prism_{\mathcal{X}/A} \otimes_A I^{\otimes i} \to \tau^{\leq i}\Prism_{\mathcal{X}/A}^{(1)}
\]
can be functorially identified with ${\rm incl} \circ \varphi^{-1}$:
\[
\tau^{\leq i}\Prism_{\mathcal{X}/A} \otimes_A I^{\otimes i} \xleftarrow[\cong]{\varphi} \tau^{\leq i}\Fil^i_{\mathrm{N}}(\mathcal{X}/A)
\xrightarrow{\text{\rm incl}} \tau^{\leq i}\Prism_{\mathcal{X}/A}^{(1)}.
\]
\end{lemma}

\begin{proof}[Proof sketch:]
This follows from the following commutative diagram:
\[
\xymatrix{
\tau^{\leq i}\Fil^i_{\mathrm{N}}(\mathcal{X}/A) \ar[r]^-{\varphi}_-{\cong} \ar[d]_-{\rm incl}
& \tau^{\leq i}\Prism_{\mathcal{X}/A} \otimes_A I^{\otimes i} \ar[d]
\ar@{-->}[dl]_-{V_i} \\
\tau^{\leq i}\Prism_{\mathcal{X}/A}^{(1)} \ar[r]^{\varphi} & \tau^{\leq i}\Prism_{\mathcal{X}/A}.
}
\]
Here the top arrow is an isomorphism due to \cite[Lemma 7.8.(3)]{LL20},
the diagonal map is defined affine locally and follows from the description
$\Prism_{\mathcal{Y}/A}^{(1)} \cong L\eta_{I}\Prism_{\mathcal{Y}/A} \xrightarrow[\cong]{\varphi} \Prism_{\mathcal{Y}/A}$
for any smooth affine $\mathcal{Y}$ over $\Spf(A/I)$
(see \cite[Theorem 15.3]{BS19}).
\end{proof}

Consequently we see that the torsion and $u^\infty$-torsion in $i$-th prismatic cohomology is canonically
a (generalized) Kisin module of height $(i-1)$.

\begin{corollary}
\label{induced Vers Corollary 2}
The restriction of the Verschiebung map $V_i \colon \rH^{i}_{\Prism}(\mathcal{X}) \to \rH^{i}_{\qsyn}(\mathcal{X}, \Prism^{(1)})$
to either torsion submodule or $u^\infty$-torsion submodule of the source is canonically divisible by $E$.
The division is given by 
\[
\rH^{i}_{\Prism}(\mathcal{X})_{\mathrm{tors}} \xleftarrow[\cong]{\varphi_{i-1}} \rH^{i}_{\qsyn}(\mathcal{X}, \Fil^{i-1}_{\mathrm{N}})_{\mathrm{tors}} \to 
\rH^{i}_{\qsyn}(\mathcal{X}, \Prism^{(1)})_{\mathrm{tors}},
\]
which, together with the usual prismatic Frobenius, makes the torsion submodule and $u^\infty$-torsion submodule
in $\rH^{i}_{\Prism}(\mathcal{X})$ a (generalized) Kisin module of height $(i-1)$.
\end{corollary}

\begin{proof}
This follows from combining \Cref{induced Vers Lemma 1} and \Cref{induced Vers Lemma 2}.
\end{proof}

We can finally understand the induced ``Verschiebung'' map:
\begin{corollary}
\label{induced Vers}
The $i$-th linear Verschiebung maps
$\rH^i_{\Prism}(\mathcal{X}/\s) \to \rH^i_{\Prism}(\mathcal{X}/\s)^{(1)}$
and 
$\rH^i_{\cris}(\mathcal{X}_0/W)^{(-1)} \to \rH^i_{\cris}(\mathcal{X}_0/W)$
induces a linear map $V_i \colon \rH^{i+1}_{\Prism}(\mathcal{X}/\s)[u] \to \rH^{i+1}_{\Prism}(\mathcal{X}/\s)^{(1)}[u]$,
which fits into the following diagram:
\[
\xymatrix{
\rH^{i+1}_{\qsyn}(\mathcal{X}, \Fil^i_{\mathrm{N}})[u] \ar[rr]^{\varphi_i}_{\cong} \ar[rd]_{\rm incl} & & \rH^{i+1}_{\Prism}(\mathcal{X}/\s)[u] \ar[ld]_{V_i} \\
& \rH^{i+1}_{\Prism}(\mathcal{X}/\s)^{(1)}[u]. &
}
\]
In particular the induced map $V_i$ is identified with $\text{\rm incl} \circ \varphi_i^{-1}$.
\end{corollary}

In other words, the induced $V_i$ is the restriction of ``$V_{i+1}$ divided by $E$'' (from 
\Cref{induced Vers Corollary 2}) to the $u$-torsion submodule.

\begin{proof}
This follows from combining \Cref{induced Vers Lemma 1} and \Cref{induced Vers Lemma 2}.
\end{proof}

\begin{corollary}
\label{induced Vers Corollary 3}
If $e \cdot (i-1) = p-1$, then the induced Verschiebung
$V_{i-1} \colon \rH^{i}_{\Prism}(\mathcal{X}/\s)[u] \to \rH^{i}_{\Prism}(\mathcal{X}/\s)^{(1)}[u]$
is an isomorphism.
\end{corollary}

\begin{proof}
This follows from \Cref{induced Vers}, \Cref{induced Vers Lemma 1} and \Cref{induced Vers Corollary 1}.
\end{proof}

\subsection*{Summary}
\addtocontents{toc}{\protect\setcounter{tocdepth}{1}}
Let us summarize our knowledge on the structure of prismatic cohomology,
with the auxiliary $(i-1)$-st Nygaard filtration in mind.
Fix the cohomological degree $i$ and $n \in \mathbb{N} \cup \{\infty\}$.
The relevant diagram is:
\[
\xymatrix@C=50pt{
\Prism^{(1)} \ar[r]^{\otimes (E^{i-1})} \ar[rd]_{\varphi} & 
\Fil^{i-1}_{\rN} \ar@{-->}@/_2pc/[l]_{\rm incl} \ar[r]^{\otimes (E)} \ar[d]_{\varphi_{i-1}} &
\Fil^i_{\rN} \ar@{-->}@/_2pc/[l]_{\rm incl} \ar[ld]^{\varphi_{i}} \\
& \Prism. &
}
\]
If we drop dotted arrows, then the diagram commutes. On the other hand, the two circles
on top has the property that compose the two arrows either way gives multiplication by $E^{i-1}$
and $E$ separately.

The above diagram induces a diagram:
\[
\xymatrix{
\varphi_{\s}^*\rH^i_{\qsyn}(\mathcal{X}, \Prism_n) \ar[r]^-{f} & 
\rH^i_{\qsyn}(\mathcal{X}, \Fil^{i-1}_{\rN}/p^n) \ar@/_1pc/[l] \ar[r]^-{g} &
\rH^i_{\qsyn}(\mathcal{X}, \Prism_n) \ar@/_1pc/[l].
}
\]
Here are some knowledge of the above diagram.
\begin{enumerate}
\item The two arrows in the first circle composes either way gives multiplication by $E^{i-1}$.
\item The two arrows in the second circle composes either way gives multiplication by $E$.
\item Compose the rightward arrows gives the prismatic Frobenius.
\item Compose the leftward arrows gives the prismatic Verschiebung $V_i$,
see \cite[Corollary 15.5]{BS19}, \cite[Corollary 7.9]{LL20} and \Cref{induced Vers Lemma 2}.
\item The map $g$ is injective, see \cite[Lemma 7.8.(3)]{LL20}.
\item When $n = \infty$, then $g$ induces an isomorphism of torsion submodules,
see \Cref{induced Vers Lemma 1}.
Hence as far as torsion or $u^\infty$-torsion in $\rH^i_{\Prism}(\mathcal{X}/\s)$
is concerned, we may focus on the first circle and see that these Frobenius modules
are canonically (generalized) Kisin modules of height $(i-1)$, see \Cref{induced Vers Corollary 2}.
\end{enumerate}

\subsection{Induced Nygaard filtration}
\addtocontents{toc}{\protect\setcounter{tocdepth}{2}}
Lastly let us discuss the induced Nygaard filtration on $u^\infty$-torsion in the boundary degree prismatic cohomology.

\begin{lemma}
Assume $e \cdot (i-1) = p-1$ and let $n \in \mathbb{N} \cup \{\infty\}$.
For any $j \in \mathbb{N}$, consider the induced map on $\rH^i_{\qsyn}(\mathcal{X}, -/p^n)$ of the maps of quasi-syntomic sheaves
$\Fil^j_{\rN} \to \Prism^{(1)}$,
the following two submodules of $\rH^i_{\qsyn}(\mathcal{X},{\Prism_n^{(1)}})[u^\infty]$
\begin{itemize}
\item $\mathrm{Im}\big(\rH^i_{\qsyn}(\mathcal{X},\Fil^j_{\rN}/p^n) \to \rH^i_{\qsyn}(\mathcal{X},\Prism^{(1)}_n) \big) \cap \rH^i_{\qsyn}(\mathcal{X},\Prism^{(1)}_n)[u^\infty]$; and
\item $\mathrm{Im}\big(\rH^i_{\qsyn}(\mathcal{X},\Fil^j_{\rN}/p^n)[u^\infty] \to \rH^i_{\qsyn}(\mathcal{X},\Prism^{(1)}_n)[u^\infty] \big)$
\end{itemize}
agree.
\end{lemma}

\begin{proof}
We look at the following diagram of $\s$-modules, with exact rows:
\[
\xymatrix{
0  \ar[r] &
\rH^i_{\qsyn}(\mathcal{X},\Fil^j_{\rN}/p^n)[u^\infty] \ar[r] \ar[d] & \rH^i_{\qsyn}(\mathcal{X},\Fil^j_{\rN}/p^n) \ar[r] \ar[d]^{f} & Q_1 \ar[d]^{g} \ar[r]
& 0 \\
0 \ar[r] &
\rH^i_{\qsyn}(\mathcal{X},\Prism^{(1)}_n)[u^\infty] \ar[r] & \rH^i_{\qsyn}(\mathcal{X},\Prism^{(1)}_n) \ar[r] & Q_2 \ar[r] & 0
}
\]
By \cite[Proposition 7.12]{LL20}, the $\Ker(f)$ has finite length, hence must be contained in $\rH^i_{\qsyn}(\mathcal{X},\Fil^j_{\rN}/p^n)[u^\infty]$.
The snake lemma implies that $\Ker(g)$ embeds inside a quotient of
$\rH^i_{\qsyn}(\mathcal{X},\Prism^{(1)}_n)[u^\infty]$.
Since $\Ker(g)$, being a submodule of $Q_1$, is $u$-torsion free,
we see it must be zero, which is exactly what we need to show.
\end{proof}

If no confusion would arise, when $e \cdot (i-1) = p-1$,
we shall refer to the submodule in the above
lemma as the \emph{(induced) $j$-th Nygaard filtration} on
$\rH^i_{\qsyn}(\mathcal{X},\Prism^{(1)}_n)[u^\infty]$.
The following proposition reveals what this filtration is.

\begin{proposition}
\label{prop induced Nygaard filtration}
Assume $e \cdot (i-1) = p-1$ and let $n \in \mathbb{N} \cup \{\infty\}$.
\begin{enumerate}
\item The Nygaard filtrations on
$\rH^i_{\qsyn}(\mathcal{X},\Prism^{(1)}_n)[u^\infty]$
as above is the $E(u) \equiv u^e$-adic filtration.
\item The map $\rH^i_{\qsyn}(\mathcal{X},\Fil^{i-1}_{\rN}/p^n) \to \rH^i_{\qsyn}(\mathcal{X},\Prism^{(1)}_n)$ is injective.
\item For any $j \geq 0$, 
the map $\rH^i_{\qsyn}(\mathcal{X},\Fil^{i+j}_{\rN}/p^n) \to \rH^i_{\qsyn}(\mathcal{X},\Prism^{(1)}_n)$
has kernel given by $u^\infty$-torsion of the source.
\end{enumerate}
\end{proposition}

\begin{remark}
\leavevmode
\begin{enumerate}
\item We remind readers that, under the hypothesis of this lemma,
\Cref{cor control u-torsion} (3) gives a canonical isomorphism
of $\s$-modules
\[
\rH^i_{\qsyn}(\mathcal{X},\Prism^{(1)}_n)[u^\infty]
\cong \rH^i_{\qsyn}(\mathcal{X},\Prism_n)[u^\infty] \otimes_{\s,\varphi_{\s}} \s
\cong \rH^i_{\qsyn}(\mathcal{X},\Prism_n)[u,p] \otimes_k \s/(p,u^p).
\]
Therefore the $E(u)$-adic filtration is the same as $u^e$-filtration.
\item Also note that $u^{e \cdot (i+j)} = u^{p-1 + e \cdot (j+1)} \in (u^p)$
if $j \geq 0$, hence (1) implies (3).
\item To put \Cref{prop induced Nygaard filtration} (3) in context,
let us point out \cite[Corollary 7.9]{LL20} which says that the divided Frobenius
$\varphi_{i+j} \colon \rH^i_{\qsyn}(\mathcal{X},\Fil^{i+j}_{\rN}/p^n)
\to \rH^i_{\qsyn}(\mathcal{X},\Prism_n)$
is an isomorphism for all $j \geq 0$.
\end{enumerate}
\end{remark}

\begin{proof}[Proof of \Cref{prop induced Nygaard filtration}]
Throughout this proof, all filtrations referred to are filtrations
on $\rH^i_{\qsyn}(\mathcal{X},\Prism^{(1)}_n)[u^\infty]$.

Since we have containment of quasi-syntomic sheaves: 
$E^j \cdot \Prism^{(1)} \subset \Fil^j_{\rN} \subset \Prism^{(1)}$, one easily sees
the Nygaard filtration contains the $u^e$-adic filtration.
All we need to show is the converse containment.

Let us first show (1) holds for the $(i-1)$-st Nygaard filtration and (2).
As discussed above, since $u^{e \cdot (i-1)} = u^{p-1}$, we see that
the $(i-1)$-st Nygaard filtration has length at least 
that of $\rH^i_{\qsyn}(\mathcal{X},\Prism_n)[u^\infty]$.\footnote{Note that here we are not twisting the prismatic cohomology by Frobenius.}
To finish, it suffices to show that the $u^\infty$-torsion in 
$\rH^i_{\qsyn}(\mathcal{X},\Fil^{i-1}_{\rN}/p^n)$ has length at most that.
This follows from the fact that the divided Frobenius, which is $\s$-linear,
\[
\varphi_{i-1} \colon \rH^i_{\qsyn}(\mathcal{X},\Fil^{i-1}_{\rN}/p^n)
\to \rH^i_{\qsyn}(\mathcal{X},\Prism_n),
\]
is injective, see \cite[Lemma 7.8.(3)]{LL20}.

Next we show (1) holds for $j$-th filtration whenever $0 \leq j \leq i-1$.
Now we look at another containment of quasi-syntomic sheaves:
$E^{i-1-j} \cdot \Fil^j_{\rN} \subset \Fil^{i-1}_{\rN} \subset \Fil^j_{\rN}$.
Therefore we see the $j$-th filtration can differ with the $(i-1)$-st filtration
by at most $u^{e \cdot (i-1-j)}$,
this gives the desired converse containment by what we proved in the previous paragraph.

Finally we show (1) holds for $(i+j)$-th filtration for any $j \geq 0$,
note that this implies (3) as remarked right after the statement of this proposition.
We want to show the map 
$\rH^i_{\qsyn}(\mathcal{X},\Fil^{i+j}_{\rN}/p^n)[u^\infty] \to \rH^i_{\qsyn}(\mathcal{X},\Prism^{(1)}_n)[u^\infty]$
is the zero map when $j \geq 0$. Since this map factors through the $j=0$ case,
it suffices to prove the $j=0$ case.
To that end, we shall utilize \cite[Corollary 7.9]{LL20}:
according to loc.~cit.~we need to show the prismatic Verschiebung annihilates
$\rH^i_{\qsyn}(\mathcal{X},\Prism_n)[u^\infty]$.
Now we contemplate the following sequence of arrows:
\[
\rH^i_{\qsyn}(\mathcal{X},\Prism^{(1)}_n)[u^\infty]
\xrightarrow{\varphi} \rH^i_{\qsyn}(\mathcal{X},\Prism_n)[u^\infty]
\xrightarrow{\psi} \rH^i_{\qsyn}(\mathcal{X},\Prism^{(1)}_n)[u^\infty].
\]
Here $\psi = V_i$ is the $i$-th Verschiebung as in \cite[Corollary 15.5]{BS19}.
The composition of these two arrows is multiplication by $E^i = u^{p-1+e} = 0$,
as the module is abstractly several copies of $\s/(p,u^p)$.
Finally we finish the proof by recalling \Cref{cor control u-torsion} (3)
that the map $\varphi$ above is surjective.
\end{proof}

As a consequence, in the boundary degree,
we can use torsion in cohomology of $\mathcal{O}_{\mathcal{X}}$
to bound $u^{\infty}$-torsion.

\begin{corollary}
\label{torsion in O bounds prismatic torsion}
Assume $e \cdot (i-1) = p-1$ and let $n \in \mathbb{N} \cup \{\infty\}$.
The natural map $\Prism^{(1)} \to \gr^0_{\rN} \Prism^{(1)} \cong \mathcal{O}$
gives rise to a canonical injection:
\[
\rH^i_{\qsyn}(\mathcal{X},\Prism_n)[u^\infty] \otimes_{k} \big(\mathcal{O}_K/p\big)
\hookrightarrow \rH^i(\mathcal{X}, \mathcal{O}_{\mathcal{X}}/p^n).
\]
\end{corollary}

\begin{proof}
The exact sequence $\Fil^1_{\rN} \to \Prism^{(1)} \to \mathcal{O}_X$ tells us that the kernel
of the map
\[
\rH^i_{\qsyn}(\mathcal{X},\Prism^{(1)}_n)[u^\infty] \cong
\rH^i_{\qsyn}(\mathcal{X},\Prism_n)[u^\infty] \otimes_{k} \bar{k}[u]/(u^p) \to \rH^i(\mathcal{O}_{\mathcal{X}}/p^n)
\]
is given by the induced first Nygaard filtration on the source, which we know is exactly $u^e$ times the source,
thanks to \Cref{prop induced Nygaard filtration} (1).
Notice that, as an $\mathcal{O}_K$-algebra, we have $\bar{k}[u]/(u^e) \cong \mathcal{O}_K/p$.
Therefore we get the desired injection
\[
\rH^i_{\qsyn}(\mathcal{X},\Prism_n)[u^\infty] \otimes_{k} \bar{k}[u]/(u^e) \cong
\rH^i_{\qsyn}(\mathcal{X},\Prism_n)[u^\infty] \otimes_{k} \big(\mathcal{O}_K/p\big)
\hookrightarrow \rH^i(\mathcal{O}_{\mathcal{X}}/p^n).
\]
\end{proof}

\section{Geometric applications}
\label{geometric applications}
\subsection{The discrepancy of Albanese varieties}
\label{Albanese subsection}
In this subsection,
we give a geometric interpretation of $u$-torsion in the second Breuil--Kisin
prismatic cohomology.
Our main application in this subsection has partly been obtained by Raynaud in \cite{Ray79},
our method is of course quite different.
Without loss of generality, we assume our smooth proper (formal) scheme
$\mathcal{X}$ has an $\mathcal{O}_K$-point.
This can be arranged after an unramified extension of $\mathcal{O}_K$.

The generic fibre of $\mathcal{X}$ is a smooth proper rigid space $X$ over $\mathrm{Sp}(K)$
admitting a $K$-point.
Specialize the main result of \cite{HL00} to our case where $X$ has a smooth proper formal model,
we know the $\mathrm{Pic}^0(X)$ is an abeloid variety which has good reduction,
namely it is the rigid generic fibre of a formal abelian scheme over $\mathcal{O}_K$.
In the algebraic situation, the existence of abelian scheme integral model
follows from Serre--Tate's generalization \cite{ST68} of the N\'{e}ron--Ogg--Shafarevich's criterion.
For the general theory of N\'{e}ron model of abeloid variety, we refer
readers to \cite{Lut95}.
Now we can form the Albanese of $X$, which is a universal map
\[
g_K \colon X \to A
\]
from $X$ to abeloid varieties, see \cite[Section 4]{HL20}.
Since in this case $A$ is the dual of $\mathrm{Pic}^0(X)$, we know it also has good reduction:
namely the N\'{e}ron model of $A$ is a formal abelian scheme $\mathcal{A}$ over $\mathcal{O}_K$.
Lastly since $\mathcal{X}$ is smooth over $\mathcal{O}_K$, the N\'{e}ron mapping property
implies that the map $X \to A$ extends uniquely to
\[
g \colon \mathcal{X} \to \mathcal{A}
\]
over $\mathcal{O}_K$.
Take the special fibre of the above map, we get 
\[
g_0 \colon \mathcal{X}_0 \to \mathcal{A}_0.
\]
Now the Albanese theory tells us the above map factors:
\[
\xymatrix{
\mathcal{X}_0 \ar[rr]^{g_0} \ar[rd]^{h} & & \mathcal{A}_0 \\
& \mathrm{Alb}(\mathcal{X}_0) \ar[ru]^{f} &
}
\]
where $\mathrm{Alb}(\mathcal{X}_0)$ is the Albanese of $\mathcal{X}_0$.
Therefore, out of a pointed smooth proper formal scheme $\mathcal{X}$ over $\mathcal{O}_K$, we can cook
up a map $f \colon \mathrm{Alb}(\mathcal{X}_0) \to \mathcal{A}_0$ of abelian varieties over $k$.
What can we say about this map?

\begin{proposition}
The map $f \colon \mathrm{Alb}(\mathcal{X}_0) \to \mathcal{A}_0$ above is an isogeny
of $p$-power degree.
\end{proposition}

\begin{proof}
It suffices to show that $f$ induces an isomorphism of the first $\ell$-adic \'{e}tale cohomology
for all primes $\ell \not= p$.
From now on we fix such an $\ell$.
Usual Albanese theory tells us that the Albanese maps
$h \colon \mathcal{X}_0 \to \mathrm{Alb}(\mathcal{X}_0)$
and $g_K \colon X \to A$
induces an isomorphism of the first $\ell$-adic \'{e}tale cohomology.
To finish the proof, we just use the smooth and proper base change theorems in
\'{e}tale cohomology theory to see that the map $g_0$,
being reduction of the ``smooth proper model'' $g$ of $g_K$,
also induces an isomorphism of the first $\ell$-adic \'{e}tale cohomology.
Since $h^* \circ f^* = g_0^*$ and both $h^*$ and $g_0^*$ induces an isomorphism
of the first $\ell$-adic \'{e}tale cohomology, we conclude that $f^*$ also does.
\end{proof}

Let us denote the finite $p$-power order group scheme $\ker(f)$ by $G$.
The Dieudonn\'{e} module of $G$ is related to $\mathcal{X}$ in the following way.

\begin{theorem}
\label{defect of Albanese}
We have an isomorphism of $W$-modules
\[
\mathbb{D}(G) \cong \rH^2_{\Prism}(\mathcal{X}/\s)^{(1)}[u].
\]
Under this identification, the semi-linear Frobenius $F$ on the left hand side
and the semi-linear Frobenius $\varphi$ on the left hand side
are related via $F = u^{p-1} \cdot \varphi$,
and the linear Verschiebung on the left hand side can be understood as
\[
\rH^2_{\Prism}(\mathcal{X}/\s)[u]
\xleftarrow[\cong]{\varphi_1} \rH^2_{\qsyn}\Fil^1_{\mathrm{N}}(\mathcal{X}/\s)[u]
\xrightarrow{\text{incl}} \rH^2_{\Prism}(\mathcal{X}/\s)^{(1)}[u].
\]
\end{theorem}

\begin{proof}
The Dieudonn\'{e} module of $G$ in our situation is given by
\[
\mathbb{D}(G) \cong 
\mathrm{Coker}\left(f^* \colon \rH^1_{\cris}(\mathcal{A}_0/W) \to \rH^1_{\cris}(\mathrm{Alb}(\mathcal{X}_0)/W)\right),
\]
so we need to understand the above map $f^*$.

We want to relate everything to $\mathcal{X}$.
First by \cite[Remarque 3.11.2]{Ill79} we know the map 
\[
h^* \colon \rH^1_{\cris}(\mathrm{Alb}(\mathcal{X}_0)/W)
\to \rH^1_{\cris}(\mathcal{X}_0/W)
\]
is an isomorphism.
Therefore by composing with $h^*$ we have
\[
\mathbb{D}(G) \cong 
\mathrm{Coker}\left(g_0^* \colon \rH^1_{\cris}(\mathcal{A}_0/W) \to \rH^1_{\cris}(\mathcal{X}_0/W)\right).
\]
Next we use the crystalline comparison of prismatic cohomology \cite[Theorem 1.8.(1)]{BS19}, and get the following diagram 
\[
\xymatrix{
\rH^1_{\Prism}(\mathcal{A}/\s)^{(1)} \ar@{->>}[r] \ar[d]^{g^*}_-{\cong} & 
\rH^1_{\Prism}(\mathcal{A}/\s)^{(1)}/u \ar[r]^-{\cong} \ar[d]^{g^*}_-{\cong} & 
\rH^1_{\cris}(\mathcal{A}_0/W) \ar[d]^{g_0^*} \\
\rH^1_{\Prism}(\mathcal{X}/\s)^{(1)} \ar@{->>}[r] & 
\rH^1_{\Prism}(\mathcal{X}/\s)^{(1)}/u \ar@{^{(}->}[r] &
\rH^1_{\cris}(\mathcal{X}_0/W).
}
\]
We postpone the proof of the left (and therefore the middle) vertical arrow being an isomorphism of $\varphi$-modules over $\s$
to the next Proposition.
The right horizontal arrows are injective because of the standard sequence
$0 \to \rH^i_{\qsyn}(\Prism^{(1)})/u \to \rH^i_{\qsyn}(\Prism^{(1)}/u) \to \rH^{i+1}_{\qsyn}(\Prism^{(1)})[u] \to 0$.
The top right horizontal arrow is an isomorphism, as the (Breuil--Kisin) prismatic cohomology
of abelian schemes are finite free, which in turn follows from the torsion-freeness
of the crystalline cohomology of abelian varieties and \Cref{finite free follows from crystalline torsion free}.
The above diagram and sequence tells us that $g_0^*$ is injective with cokernel given by
$\rH^2_{\Prism^{(1)}}(\mathcal{X}/\s)[u]$.
The description of (the semi-linear) Frobenius follows from \Cref{induced Frob} (2),
and the description of the linear Verschiebung follows from \Cref{induced Vers}.
\end{proof}

The following Proposition was summoned in the above proof.

\begin{proposition}
\leavevmode
\begin{enumerate}
\item The underlying $\s$-module of $\rH^1_{\Prism}(\mathcal{X}/\s)$ is finite free; and
\item The map $g^* \colon \rH^1_{\Prism}(\mathcal{A}/\s) \to \rH^1_{\Prism}(\mathcal{X}/\s)$ is an isomorphism of Kisin modules.
Therefore the Frobenius-twisted version
$g^* \colon \rH^1_{\Prism}(\mathcal{A}/\s)^{(1)} \to \rH^1_{\Prism}(\mathcal{X}/\s)^{(1)}$
is also an isomorphism.
\end{enumerate}
\end{proposition}

\begin{proof}
(1): this follows from \Cref{finite free in low degree and ramification}.
Alternatively we can prove this using
\Cref{finite free follows from crystalline torsion free}, \cite[Remarque 3.11.2]{Ill79}
and the fact that crystalline cohomology of abelian varieties are torsion-free.

(2): Since \'{e}tale realization of finite free Kisin modules is fully faithful,
see \cite[Proposition 2.1.12]{KisinFcrystal} and also \cite[Theorem 7.2]{BS21},
we are reduced to checking the \'{e}tale realization of $g^*$ is an isomorphism.
Since the map $\s \to A_{\inf}$ sending $u$ to $[\underline{\pi}]$ is $p$-completely faithfully flat,
it remains so after $p$-completely inverting $u$ and $[\underline{\pi}]$ respectively.
Therefore we are further reduced to proving it for the $W(C^{\flat})$-\'{e}tale realizations.
Now the \'{e}tale comparison \cite[Theorem 1.8.(4)]{BS19} translates the above
to the statement that $g_K$ induces an isomorphism of first $p$-adic \'{e}tale cohomology,
which follows from the usual Kummer sequence together with the fact that
the Picard variety of $X$ is an abeloid.\footnote{Note that in general Albanese of smooth proper
rigid spaces (granting its existence) always induces an injective but not necessarily surjective
map of first \'{e}tale cohomology, no matter $\ell$-adic or $p$-adic,
see \cite[Proposition 4.4, Example 5.2 and Example 5.8]{HL20}.
The surjectivity is equivalent to the Picard variety being an abeloid (assuming $p$ is invertible
in the ground non-archimedean field).}
\end{proof}


We get two consequences from \Cref{defect of Albanese}.

\begin{corollary}
The finite group scheme $G$ is connected.
\end{corollary}

\begin{proof}
Since the induced Frobenius on $\mathbb{D}(G)$, when identified with 
$\rH^2_{\Prism}(\mathcal{X}/\s)[u] \subset \rH^2_{\Prism}(\mathcal{X}/\s)[u^{\infty}]$, is divisible by $u^{p-1}$,
powers of Frobenius will gain more and more $u$-divisibility.
We see the Frobenius is nilpotent as 
$\rH^2_{\Prism}(\mathcal{X}/\s)[u] \subset \rH^2_{\Prism}(\mathcal{X}/\s)[u^{\infty}]$
and there is a power of $u$ which kills the latter.
Now \Cref{defect of Albanese} implies the Frobenius on $\mathbb{D}(G)$ is nilpotent,
therefore $G$ is connected.
\end{proof}

\begin{remark}
The above fact can actually be seen directly.
Let us quotient out $\mathrm{Alb}(\mathcal{X}_0)$ by the neutral component subgroup scheme of $G$,
denoted by $\mathcal{A}_0'$.
Then we get a factorization $\mathcal{X}_0 \to \mathcal{A}_0' \xrightarrow{f'_0} \mathcal{A}_0$ of $g_0$.
Now $f'_0$ is finite \'{e}tale by construction.
Hence deformation theory implies the above sequence lift to
$\mathcal{X} \to \mathcal{A}' \xrightarrow{f'} \mathcal{A}$ with $\mathcal{A}'$
being a formal abelian scheme finite \'{e}tale above $\mathcal{A}$.
Now the composition of the above map is the universal map from $\mathcal{X}$ to formal abelian schemes
as pointed formal schemes,\footnote{We use the N\'{e}ron mapping
property and the fact that the generic fibre map being the Albanese map.} we conclude that the map $f'$ has to be an isomorphism,
hence the neutral component subgroup scheme of $G$ is $G$ itself.
\end{remark}

Combining \Cref{defect of Albanese}, \Cref{cor control u-torsion} (with $i = 2$) and \Cref{induced Vers Corollary 3},
we immediately yield the following result.
\begin{corollary}
\label{defect controlled by ramification index}
\leavevmode
\begin{enumerate}
\item If $e < p-1$ then the map $f \colon \mathrm{Alb}(\mathcal{X}_0) \to \mathrm{Alb}(X)_0$ is an isomorphism.
\item If $e < 2(p-1)$ then $\ker(f)$ is $p$-torsion.
\item If $e = p-1$ then $\ker(f)$ is $p$-torsion and of multiplicative type, hence must be a form of several copies of $\mu_p$.
Moreover there is a canonical injection of $\mathcal{O}_K$-modules
$\mathbb{D}(\ker(f)) \otimes_k \big(\mathcal{O}_K/p\big) \hookrightarrow \rH^2(\mathcal{X}, \mathcal{O}_{\mathcal{X}})$.
\end{enumerate}
\end{corollary}

\begin{proof}
(1) and (2) follows from \Cref{defect of Albanese} and \Cref{cor control u-torsion} (with $i = 2$) (1) and (2) respectively.
As for the multiplicativity claim in (3):
recall that a finite flat group scheme over $k$ is of multiplicative type if and only
if its Dieudonn\'{e} module has bijective Verschiebung,
hence the claim follows from \Cref{defect of Albanese} and \Cref{induced Vers Corollary 3}.
The last sentence follows from \Cref{torsion in O bounds prismatic torsion}.
\end{proof}

When $e = 1$ and  $p=2$ the above says that although the $f$ need not be an isomorphism, the kernel is always a
$2$-torsion,
such an interesting example can be found in \cite[Subsection 2.1]{BMS1},
and one can check directly that the example there does satisfy our prediction here.
In fact the $f$ for their example can be identified with the relative Frobenius
of an ordinary elliptic curve (which is the reduction of the $E$ in their notation)
over $\mathbb{F}_2$.
For a generalization of this example to the case when $p \not= 2$, we refer readers to our \Cref{subsection example},
and specifically our \Cref{rmk stacky example} and \Cref{rmk example} (3).

\begin{remark}
If $\mathrm{Pic}^0(\mathcal{X}_0)$ is reduced, then the relative (formal) Picard scheme of $\mathcal{X}/\mathcal{O}_K$
is a formal abelian scheme which is the N\'{e}ron model of the Picard variety of $X/K$.
Base change property of relative Picard functor now guarantees that the $f$ we have been studying is an isomorphism
in this case. Therefore combining with \Cref{defect of Albanese} we see that
$\mathcal{X}_0$ having reduced Picard scheme implies the second prismatic cohomology of $\mathcal{X}$ has no $u$-torsion.
\end{remark}

The dual question of what we discussed here was studied by Raynaud \cite{Ray79}, below
we recall some of the main results in loc.~cit.~and compare with ours.

\begin{remark}
\label{translation to Picard remark}
Using determinant construction \cite{KM76}, 
the universal line bundle on $\mathcal{X}_K \times_K \mathrm{Pic}^0(\mathcal{X}_K)$
extends to a line bundle on $\mathcal{X} \times_{\mathcal{O}_K} \mathcal{P}$,
where $\mathcal{P}$ is the formal N\'{e}ron model of $\mathrm{Pic}^0(\mathcal{X}_K)$
which is itself a formal abelian scheme over $\mathcal{O}_K$.
Here we used the regularity of $\mathcal{X} \times_{\mathcal{O}_K} \mathcal{P}$
so that any coherent sheaf on it can be presented as a perfect complex,
in order to perform the determinant construction.
Moreover if we rigidify using the given point $x \in \mathcal{X}(\mathcal{O}_K)$,
then the extension as a rigidified line bundle is unique.
Taking the special fibre, we get an induced map $\mathcal{P}_0 \to \mathrm{Pic}^0(\mathcal{X}_0)$
which necessarily factors through the reduced subvariety of the target
$f^{\vee} \colon \mathcal{P}_0 \to \mathrm{Pic}^0(\mathcal{X}_0)_{\mathrm{red}}$.
By construction, the map $f^{\vee}$ is dual to the map $f$ we considered before.

Raynaud has studied the question of whether $f^{\vee}$ is an isomorphism in \cite{Ray79}.
His main result says:
\begin{enumerate}
\item{\cite[Th\`{e}or\'{e}me 4.1.3.(2)]{Ray79}} When $e < p-1$, then $f^{\vee}$ is an isomorphism.
\item{\cite[Th\`{e}or\'{e}me 4.1.3.(3)]{Ray79}} When $e=p-1$, then $\ker(f^{\vee})$ is $p$-torsion and unramified.
\end{enumerate}
We see that his results are the same as \Cref{defect controlled by ramification index} (1) and first half of (3),
our slight improvement is \Cref{defect controlled by ramification index} (2) and second half of (3): 
We prove the map $f^{\vee}$ has $p$-torsion kernel in a larger range of ramifications,
and when $e = p-1$ the second cohomology of structure sheaf needs to have ``actual''
$p$-torsion in order for $\ker(f)$ to be nonzero.
On the other hand, Raynaud's result allows $\mathcal{X}$ to be singular:
for instance he just needs $\mathcal{X}_0$ to be normal.
Our method crucially relies on prismatic theory, which seems to only work well with 
local complete intersection singularities.
Whether our \Cref{defect controlled by ramification index} can be extended to
the generality considered by Raynaud remains unclear and interesting to us.
\end{remark}

\begin{remark}
One of the key ingredient letting Raynaud to prove the aforementioned results in \cite{Ray79}
is an earlier result of his \cite{Ray74} concerning prolongations of finite flat commutative
group schemes.
In the end of this paper, \Cref{Raynaud subsection}, we shall see a way to go backward:
applying these structural results on Picard/Albanese varieties to a marvelous
construction due to Bhatt--Morrow--Scholze \cite{BMS1}, one deduces
Raynaud's prolongation theorem.
\end{remark}

\subsection{The $p$-adic specialization maps}
\label{cospecialization subsection}

Another reason why one might care about $u^\infty$-torsion is because it appears naturally in understanding
the specialization map of $p$-adic \'{e}tale cohomology or, phrased differently, the $p$-adic vanishing cycle.

Let us introduce some notations. Fix a complete algebraically closed non-archimedean extension $C$ of $K$,
with ring of integers $\mathcal{O}_C$. Denote the perfect prism associated with $\mathcal{O}_C$,
which is known to be oriented, by $(A_{\inf}, (\xi))$.
Given $p$-adic formal scheme $\mathcal{X}$ over $\Spf(\mathcal{O}_K)$,
we denote its base change to $\mathcal{O}_C$ (resp.~$C$) as $\mathcal{X}_{\mathcal{O}_C}$ (resp.~$\mathcal{X}_C$).
Denote the central fibre of $\mathcal{X}_{\mathcal{O}_C}$ by $\mathcal{X}_{\bar{k}}$.
We keep assuming $\mathcal{X}$ to be smooth and proper over $\Spf(\mathcal{O}_K)$.

Recall the proper base change theorem gives, for any prime $\ell$, a specialization map~\cite[\href{https://stacks.math.columbia.edu/tag/0GJ2}{Tag 0GJ2}]{stacks-project}
\[
\mathrm{Sp} \colon \mathrm{R\Gamma}_{\et}(\mathcal{X}_{\bar{k}}, \mathbb{Z}_\ell) \to \mathrm{R\Gamma}_{\et}(\mathcal{X}_C, \mathbb{Z}_\ell).
\]
The cone of specialization map above is called the vanishing cycle (of $\mathbb{Z}_\ell$).
The smooth base change theorem says that the above map is an isomorphism for any $\ell \not= p$~\cite[\href{https://stacks.math.columbia.edu/tag/0GKD}{Tag 0GKD}]{stacks-project},
in other words $\ell$-adic vanishing cycle vanishes in our setting.
On the other hand, one may ask what happens when $\ell = p$. Fix a cohomological degree $i$ and 
$n \in \mathbb{N} \cup \{\infty\}$, let us look at
\[
\mathrm{Sp}_n^i \colon \rH^i_{\et}(\mathcal{X}_{\bar{k}}, \mathbb{Z}/p^n) \to \rH^i_{\et}(\mathcal{X}_C, \mathbb{Z}/p^n),
\]
when $n=\infty$ the above means $\mathbb{Z}_p$ coefficient and we will simply write $\mathrm{Sp}^i$.
It is well-known that $\mathrm{Sp}^i$ is almost never surjective unless for trivial reasons
such as the target being $0$.
We shall contemplate with $\ker(\mathrm{Sp}^i)$ in this subsection.

In \cite[Section 9]{BS19} one finds a prismatic interpretation of the $p$-adic specialization map:
\begin{theorem}[{\cite[Theorem 9.1 and Remark 9.3]{BS19}}]
\label{etale comparison}
There are canonical identifications:
\[
\mathrm{R\Gamma}_{\et}(\mathcal{X}_C, \mathbb{Z}/p^n) \cong 
\big((\mathrm{R\Gamma}_{\Prism}(\mathcal{X}_{\mathcal{O}_C}/A_{\inf})[1/\xi])^{\widehat{}}/p^n\big)^{\varphi = 1},
\]
and
\[
\mathrm{R\Gamma}_{\et}(\mathcal{X}_{\bar{k}}, \mathbb{Z}/p^n) \cong 
\big(\mathrm{R\Gamma}_{\Prism}(\mathcal{X}_{\mathcal{O}_C}/A_{\inf})/p^n\big)^{\varphi = 1},
\]
fitting in the following diagram, which is commutative up to coherent homotopy:
\[
\xymatrix{
\mathrm{R\Gamma}_{\et}(\mathcal{X}_{\bar{k}}, \mathbb{Z}/p^n) \ar[r]^-{\cong} \ar[d]_-{\mathrm{Sp}} &
\big(\mathrm{R\Gamma}_{\Prism}(\mathcal{X}_{\mathcal{O}_C}/A_{\inf})/p^n\big)^{\varphi = 1} \ar[d]^-{\mathrm{incl}} \\
\mathrm{R\Gamma}_{\et}(\mathcal{X}_C, \mathbb{Z}/p^n) \ar[r]^-{\cong} & 
\big((\mathrm{R\Gamma}_{\Prism}(\mathcal{X}_{\mathcal{O}_C}/A_{\inf})[1/\xi])^{\widehat{}}/p^n\big)^{\varphi = 1}.
}
\]
\end{theorem}

Here $(\mathrm{R\Gamma}_{\Prism}(\mathcal{X}_{\mathcal{O}_C}/A_{\inf})[1/\xi])^{\widehat{}}$ denotes the $p$-completion
of the localization, which is only relevant in the statement when $n=\infty$.
This theorem is true without smooth or proper assumption on $\mathcal{X}$:
one may safely replace $\mathcal{X}_{\mathcal{O}_C}$ over
$\Spf(\mathcal{O}_C)$ with any $p$-adic formal scheme $\mathcal{Y}$ over a perfectoid base ring
as in loc.~cit.

\begin{proof}[Sketch of proof following that of loc.~cit.]
The first identification is \cite[Theorem 9.1]{BS19}, the second identification
is \cite[Remark 9.3]{BS19} with details left to readers, so let us fill in some details.

We follow the proof of \cite[Theorem 9.1]{BS19}.
First we see that both of $\mathrm{R\Gamma}_{\et}((-)_{\bar{k}}, \mathbb{Z}/p^n)$
and $\big(\Prism_{-/A_{\inf}}/p^n\big)^{\varphi = 1}$
are arc-sheaves (see \cite{BM21} for more details on this notion) on $\mathrm{fSch}_{/\Spf(\mathcal{O}_C)}$.
The former is \cite[Theorem 5.4]{BM21}, the latter follows from the same argument as in loc.~cit.:
using \cite[Lemma 9.2]{BS19} one has an identification
$\big(\Prism_{-/A_{\inf}}/p^n\big)^{\varphi = 1} \cong \big(\Prism_{-/A_{\inf},\mathrm{perf}}/p^n\big)^{\varphi = 1}$,
then one again uses \cite[Corollary 8.10]{BS19} to see the latter is an arc-sheaf.

Since everything involved is an arc-sheaf  and is arc-locally
supported in cohomological degree $0$, the relevant maps (of arc-sheaves)
live in mapping spaces with contractible components.
Altogether we get the following diagram which commutes up to coherent homotopy:
\[
\xymatrix{
\mathrm{R\Gamma}_{\et}((-)_{\bar{k}}, \mathbb{Z}/p^n) \ar[r] \ar[d]_-{\mathrm{Sp}} &
\big(\Prism_{-/A_{\inf}}/p^n\big)^{\varphi = 1} \ar[d]^-{\mathrm{incl}} \\
\mathrm{R\Gamma}_{\et}((-)_C, \mathbb{Z}/p^n) \ar[r]^-{\cong} & 
\big((\Prism_{-/A_{\inf}}[1/\xi])^{\widehat{}}/p^n\big)^{\varphi = 1}.
}
\]
Lastly we need to show the top horizontal arrow is an isomorphism.
We may localize in the arc-topology, reducing to the case of $\Spf$ of a perfectoid ring $S$,
which follows from applying Artin--Schreier--Witt and the fact that perfection  does not change the \'{e}tale site
(of a characteristic $p$ scheme).
\end{proof}

In order to pass from the derived statement above to concrete cohomology groups, we need the following:
\begin{lemma}
\label{Frob-id lemma}
For any $i$ and $n$, the $\mathbb{Z}_p$-linear operator $\varphi - 1$ is surjective on
both $\rH^i\big((\mathrm{R\Gamma}_{\Prism}(\mathcal{X}_{\mathcal{O}_C}/A_{\inf})[1/\xi])^{\widehat{}}/p^n\big)$
and $\rH^i\big(\mathrm{R\Gamma}_{\Prism}(\mathcal{X}_{\mathcal{O}_C}/A_{\inf})/p^n\big)$.
\end{lemma}

\begin{proof}
We observe that, since $\varphi([a]) = [a]^p$ for any $a \in \mathfrak{m}_C^{\flat}$, we know
$\varphi$ acts topologically nilpotently on 
\[
[\mathfrak{m}_C^{\flat}] \cdot \rH^i\big(\mathrm{R\Gamma}_{\Prism}(\mathcal{X}_{\mathcal{O}_C}/A_{\inf})/p^n\big).
\]
Therefore to check surjectivity of $\varphi - 1$ on $\rH^i\big(\mathrm{R\Gamma}_{\Prism}(\mathcal{X}_{\mathcal{O}_C}/A_{\inf})/p^n\big)$
we may quotient out $W(\mathfrak{m}_C^{\flat}) \cdot \rH^i\big(\mathrm{R\Gamma}_{\Prism}(\mathcal{X}_{\mathcal{O}_C}/A_{\inf})/p^n\big)$.
Since $\mathcal{X}$ is smooth and proper over $\mathcal{O}_K$, we know the relevant groups are
finitely generated modules over $W(C^{\flat})$ and $W(\bar{k})$.
Both of $C^{\flat}$ and $\bar{k}$ are algebraically closed field of characteristic $p$, hence we are reduced to
\cite[Expos\'{e} III, Lemma 3.3]{CL98}.
\end{proof}

Using the same proof, we may identify $p$-adic \'{e}tale cohomology of $\mathcal{X}_{\bar{k}}$ 
as Frobenius fixed points in various prismatic cohomology of $\mathcal{X}$, after suitably base changing to $W(\bar{k})$.

\begin{porism}
\label{porism etale in Frob twisted}
Consider the $\s$-algebra $W(\bar{k})[\![u]\!]$.
We have an identification of $G_k$-modules:
\[
\rH^i_{\et}(\mathcal{X}_{\bar{k}}, \mathbb{Z}/p^n) \cong 
\Big(\rH^i_{\qsyn}(\mathcal{X}, \Prism_n) \otimes_{\s} W(\bar{k})[\![u]\!] \Big)^{\varphi = 1}
\cong
\Big(\rH^i_{\qsyn}(\mathcal{X}, \Prism^{(1)}_n) \otimes_{\s} W(\bar{k})[\![u]\!] \Big)^{\varphi = 1}.
\]
\end{porism}

\begin{proof}
As showed in the proof of \Cref{Frob-id lemma}, we may compute Frobenius fixed points after
quotient out $W(\mathfrak{m}_C^{\flat})$ (for the $A_{\inf}$-module) or $u$ for the Frobenius module appearing
in this porism.
Now the first identification is reduced to \Cref{etale comparison} and an equality of $\s$-algebras:
$A_{\inf}/W(\mathfrak{m}_C^{\flat}) \cong W(\bar{k}) \cong W(\bar{k})[\![u]\!]/(u)$.
The second identification is reduced to the fact that given a Frobenius module $M$ on $W(\bar{k})$,
then the natural map $M \to M \otimes_{W(\bar{k}), \varphi} W(\bar{k})$
given by $m \mapsto m \otimes 1$ induces an isomorphism of Frobenius fixed points.
\end{proof}


\begin{remark}
\label{etale to O map}
Assume that the residue field $\bar{k}$ of $\mathcal{O}_K$ is separably closed.
The above \Cref{porism etale in Frob twisted} induces a map 
\[
\rH^i_{\et}(\mathcal{X}_{\bar{k}}, \mathbb{Z}/p^n) \cong
\Big(\rH^i_{\qsyn}(\mathcal{X}, \Prism^{(1)}_n)\Big)^{\varphi = 1}
\hookrightarrow \rH^i_{\qsyn}(\mathcal{X}, \Prism^{(1)}_n)
\to \rH^i(\mathcal{O}_{\mathcal{X}}/p^n).
\]
This map can be seen at the level of \'{e}tale-sheaves on $\mathrm{fSch}_{/\Spf(\mathcal{O}_K)}$:
$\mathbb{Z}_p/p^n \to \Prism^{(1)}_n \to \mathcal{O}_{\mathcal{X}}/p^n$.
Therefore we get a canonical map
\[
\rH^i_{\et}(\mathcal{X}_{\bar{k}}, \mathbb{Z}/p^n) \otimes_{\mathbb{Z}_p} W
\to \rH^i(\mathcal{O}_{\mathcal{X}}/p^n).
\]
In general, we just base change along $W(k) \to W(\bar{k})$
and get a $G_k$-equivariant map
\[
\rH^i_{\et}(\mathcal{X}_{\bar{k}}, \mathbb{Z}/p^n) \otimes_{\mathbb{Z}_p} W(\bar{k})
\to \rH^i(\mathcal{O}_{\mathcal{X}}/p^n) \otimes_{W} W(\bar{k}).
\]
\end{remark}

Later in \Cref{control ker Cosp} (3) we shall see a peculiar result concerning this map in the boundary degree.
Now we come back to the relation between kernel of specialization map and $u^\infty$-torsion
in prismatic cohomology.

\begin{theorem}
\label{ker Cosp and u-torsion}
Let $\mathcal{X}$ be a smooth proper formal scheme over $\Spf(\mathcal{O}_K)$.
Recall $\m^i_n \coloneqq \rH^i_{\qsyn}(\mathcal{X}, \Prism_n)[u^\infty]$.
There is a canonical isomorphism of $G_k$-modules
\[
\ker(\mathrm{Sp}_n^i) \cong (\m^i_n \otimes_{\s} A_{\inf}\big)^{\varphi = 1}
\cong \big(\m^i_n/u \otimes_{W(k)} W(\bar{k})\big)^{\varphi = 1}
\]
for any $n \in \mathbb{N} \cup \{\infty\}$.
\end{theorem}

\begin{proof}
Combining \Cref{etale comparison} and \Cref{Frob-id lemma}, we get the following diagram with exact rows:
\[
\xymatrix{
0 \ar[r] & 
\rH^i_{\et}(\mathcal{X}_{\bar{k}}, \mathbb{Z}/p^n) \ar[d]^-{\mathrm{Sp}_n^i} \ar[r] & 
\rH^i\big(\mathrm{R\Gamma}_{\Prism}(\mathcal{X}_{\mathcal{O}_C}/A_{\inf})/p^n\big) \ar[d]^-{\mathrm{incl}} \ar[r]^-{\varphi - 1} & 
\rH^i\big(\mathrm{R\Gamma}_{\Prism}(\mathcal{X}_{\mathcal{O}_C}/A_{\inf})/p^n\big) \ar[d]^-{\mathrm{incl}} \ar[r] & 
0 \\
0 \ar[r] & 
\rH^i_{\et}(\mathcal{X}_C, \mathbb{Z}/p^n) \ar[r] & 
\rH^i\big((\mathrm{R\Gamma}_{\Prism}(\mathcal{X}_{\mathcal{O}_C}/A_{\inf})[1/\xi])^{\widehat{}}/p^n\big) \ar[r]^-{\varphi - 1} &
\rH^i\big((\mathrm{R\Gamma}_{\Prism}(\mathcal{X}_{\mathcal{O}_C}/A_{\inf})[1/\xi])^{\widehat{}}/p^n\big) \ar[r] & 
0.
}
\]
We shall apply the snake lemma to the above. 
First, we claim
\[
\rH^i\big(\mathrm{R\Gamma}_{\Prism}(\mathcal{X}_{\mathcal{O}_C}/A_{\inf})/p^n\big)[\xi^{\infty}] \cong
\ker(\rH^i\big(\mathrm{R\Gamma}_{\Prism}(\mathcal{X}_{\mathcal{O}_C}/A_{\inf})/p^n\big) \xrightarrow{\mathrm{incl}}
\rH^i\big((\mathrm{R\Gamma}_{\Prism}(\mathcal{X}_{\mathcal{O}_C}/A_{\inf})[1/\xi])^{\widehat{}}/p^n\big).
\]
When $n \in \mathbb{N}$ the map is localization with respect to $\xi$, hence tautological.
We need to see this when $n = \infty$, namely we need to show injectivity of
$\rH^i\big(\mathrm{R\Gamma}_{\Prism}(\mathcal{X}_{\mathcal{O}_C}/A_{\inf})[1/\xi]\big)
\to \rH^i(\mathrm{R\Gamma}_{\Prism}(\mathcal{X}_{\mathcal{O}_C}/A_{\inf})[1/\xi])^{\widehat{}}$.
Here the latter completion is the classical $p$-adic completion: 
our assumption implies all cohomology groups $\rH^i_{\Prism}(\mathcal{X}_{\mathcal{O}_C}/A_{\inf})$
have bounded $p$-torsion, hence derived $p$-completion agrees with derived $p$-completion.
Since $\rH^i_{\Prism}(\mathcal{X}_{\mathcal{O}_C}/A_{\inf})$ are finitely presented over $A_{\inf}$,
its localization with respect to $\xi$ has separated $p$-adic topology, hence the $p$-adic completion
map is injective.

Next, applying the base change property of prismatic cohomology to the $p$-completely faithfully flat
map $\s \to A_{\inf}$ and \cite[Proposition 4.3]{BMS1}, we get an identification of Frobenius modules:
\[
\m^i_n \otimes_{\s} A_{\inf} \cong
\rH^i\big(\mathrm{R\Gamma}_{\Prism}(\mathcal{X}_{\mathcal{O}_C}/A_{\inf})/p^n\big)[\xi^{\infty}].
\]
Now we get the first identification.
To finish, just observe that $\varphi([a]) = [a]^p$, for any $a \in \mathfrak{m}_C^{\flat}$,
which acts nilpotently on $\rH^i\big(\mathrm{R\Gamma}_{\Prism}(\mathcal{X}_{\mathcal{O}_C}/A_{\inf})/p^n\big)[\xi^{\infty}]$.
Hence the map $\varphi - 1$ is necessarily an isomorphism (of $\mathbb{Z}_p$-modules) on
$[\mathfrak{m}^{\flat}] \cdot \rH^i\big(\mathrm{R\Gamma}_{\Prism}(\mathcal{X}_{\mathcal{O}_C}/A_{\inf})/p^n\big)[\xi^{\infty}]$.
Therefore we may quotient this part, as far as Frobenius fixed points are concerned, which leads
to the second identification.
\end{proof}

\begin{corollary}
\label{control ker Cosp}
Let $\mathcal{X}$ be a smooth proper formal scheme over $\Spf(\mathcal{O}_K)$ with ramification index $e$,
let $i \in \mathbb{N}$ and $n \in \mathbb{N} \cup \{\infty\}$.
We have the following understanding of the kernel of the specialization map
$\mathrm{Sp}_n^i \colon \rH^i_{\et}(\mathcal{X}_{\bar{k}}, \mathbb{Z}/p^n) \to \rH^i_{\et}(\mathcal{X}_C, \mathbb{Z}/p^n)$.
\begin{enumerate}
\item If $e \cdot (i-1) < p-1$, then $\mathrm{Sp}_n^i$ is injective.
\item If $e \cdot (i-1) < 2(p-1)$, then $\ker(\mathrm{Sp}_n^i)$ is annihilated by $p^{i-1}$.
\item If $e \cdot (i-1) = p-1$, then $\ker(\mathrm{Sp}_n^i)$ is $p$-torsion,
and corresponds to the \'{e}tale-$\varphi$ module $\m^i_n$ over $k$.
Moreover the natural $G_k$-equivariant map in \Cref{etale to O map}
\[
\rH^i_{\et}(\mathcal{X}_{\bar{k}}, \mathbb{Z}/p^n) \otimes_{\mathbb{Z}_p} W(\bar{k})
\to \rH^i(\mathcal{O}_{\mathcal{X}}/p^n) \otimes_{W} W(\bar{k})
\]
induces a $G_k$-equivariant injection:
\[
\ker(\mathrm{Sp}^i_n) \otimes_{\mathbb{F}_p} \big(\mathcal{O}_K \otimes_W W(\bar{k})\big)/p
\hookrightarrow \rH^i(\mathcal{O}_{\mathcal{X}}/p^n) \otimes_{W} W(\bar{k}).
\]
\end{enumerate}
\end{corollary}

\begin{proof}
All but the last statement immediately follow from \Cref{cor control u-torsion} and \Cref{ker Cosp and u-torsion}.
The last statement is a Galois-theoretic analog of \Cref{torsion in O bounds prismatic torsion}.
To prove this, we may base change $\mathcal{X}$ from $\mathcal{O}_K$ to $\mathcal{O}_K \otimes_W W(\bar{k})$
and it suffices to prove the statement there.
Hence it suffices to assume that $\mathcal{O}_K$ has algebraically closed residue field $\bar{k}$.

Let us analyze the sequence of maps of $\s$-modules
\[
\rH^i_{\et}(\mathcal{X}_{\bar{k}}, \mathbb{Z}/p^n) \cong
\Big(\rH^i_{\qsyn}(\mathcal{X}, \Prism^{(1)}_n)\Big)^{\varphi = 1}
\hookrightarrow \rH^i_{\qsyn}(\mathcal{X}, \Prism^{(1)}_n)
\to \rH^i(\mathcal{O}_{\mathcal{X}}/p^n).
\]
By \Cref{control ker Cosp} (3), we see the first map induces an isomorphism:
\[
\ker(\mathrm{Sp}^i_n) \otimes_{\mathbb{F}_p} \bar{k}[u]/(u^p) \cong \rH^i_{\qsyn}(\mathcal{X}, \Prism^{(1)}_n)[u^\infty].
\]
The exact sequence $\Fil^1_{\rN} \to \Prism^{(1)} \to \mathcal{O}_{\mathcal{X}}$ tells us that the kernel
of the map
\[
\ker(\mathrm{Sp}^i_n) \otimes_{\mathbb{F}_p} \bar{k}[u]/(u^p) \to \rH^i(\mathcal{O}_{\mathcal{X}}/p^n)
\]
is given by the induced first Nygaard filtration on the source, which we know is exactly $u^e$ times the source,
thanks to \Cref{prop induced Nygaard filtration} (1).
Notice that, as an $\mathcal{O}_K$-algebra, we have $\bar{k}[u]/(u^e) \cong \mathcal{O}_K/p$.
Therefore we get the desired injection
\[
\ker(\mathrm{Sp}^i_n) \otimes_{\mathbb{F}_p} \bar{k}[u]/(u^e) \hookrightarrow \rH^i(\mathcal{O}_{\mathcal{X}}/p^n).
\]
\end{proof}

We refer readers to \Cref{subsection example},
especially \Cref{rmk stacky example} and \Cref{rmk example} (4),
for a related interesting example.

\subsection{Revisiting the integral Hodge--de Rham spectral sequence}
In this subsection, we revisit the question discussed in \cite{Li20}:
what mild condition on $\mathcal{X}$ guarantees that the Hodge numbers of the generic fibre $X$
can be read off from the special fibre $\mathcal{X}_0$?

Let us introduce a notation, which is the threshold of cohomological degree
for which we can say something about the integral Hodge--de Rham spectral sequence,
based on knowledge of the integral Hodge--Tate spectral sequence.

\begin{notation}
Let $T$ be the largest integer such that $e \cdot (T-1) \leq p-1$.
\end{notation}

The main result in this subsection is:
\begin{theorem}[{Improvement of \cite[Theorem 1.1]{Li20}}]
\label{improving Li20}
Let $\mathcal{X}$ be a smooth proper $p$-adic formal scheme over $\Spf(\mathcal{O}_K)$.
\begin{enumerate}
\item Assume there is a lift of $\mathcal{X}$ to $\s/(E^2)$,
then for all $i,j$ satisfying $i + j < T$, we have equalities
\[
h^{i,j}(X) = \mathfrak{h}^{i,j}(\mathcal{X}_0)
\]
where the latter denotes virtual Hodge numbers of $\mathcal{X}_0$, defined as in \cite[Definition 3.1]{Li20}.
\item Assume furthermore that $e \cdot (\dim{\mathcal{X}_0}-1) \leq p-1$.
Then the special fibre $\mathcal{X}_0$ knows the Hodge numbers of the rigid generic fibre $X$.
\end{enumerate}
\end{theorem}

For instance, in the unramified case $e = 1$, condition (1) is automatic and condition (2)
says we allow $\mathcal{X}$ to be at most dimension $p$.
From the proof,
we shall see that the Hodge numbers of $X$ can be computed using the virtual Hodge numbers of $\mathcal{X}_0$
(see \cite[Subsection 3.2]{Li20}) together with Euler characteristics of $\Omega^i_{\mathcal{X}_0}$'s
in an algorithmic way.

We largely follow the proof of \cite[Theorem 1.1]{Li20}. Just like there, we need to first
analyze the integral Hodge--de Rham spectral sequence, hence the title of this subsection.

\begin{theorem}
\label{integral HdRSS}
Let $\mathcal{X}$ be a smooth proper $p$-adic formal scheme over $\Spf(\mathcal{O}_K)$ 
liftable to $\s/(E^2)$.
Let $n \in \mathbb{N} \cup \{\infty\}$.
\begin{enumerate}
\item The Hodge--de Rham spectral sequence for $\mathcal{X}_n$ has no nonzero differentials
with source of total degree $< T$.
\item If $e > 1$, then $\mathfrak{M}^T_n \coloneqq \rH^T_{\qsyn}(\mathcal{X}, \Prism_n)[u^\infty] = 0$.
In particular, the prismatic cohomology $\rH^m_{\Prism}(\mathcal{X}/\s) \simeq M_m \otimes_{\mathbb{Z}_p} \s$
is of the shape of a $\mathbb{Z}_p$-module $M_m$ for all $m \leq T$.
\item If $e=1$, The induced Hodge filtrations $\rH^i(\mathcal{X},\Fil^j_\rH) \subset \rH^i_{\dR}(\mathcal{X})$
are split for any $i \leq p$ and any $j$.
\item If $e>1$, The induced Hodge filtrations $\rH^i(\mathcal{X},\Fil^j_\rH) \subset \rH^i_{\dR}(\mathcal{X})$
are split for any $i < T$ and any $j$.
\end{enumerate}
\end{theorem}

Here $\mathcal{X}_n$ denotes the mod $p^n$ fibre.
We do not know if the split statement in (3) above holds at the mod $p^n$ level.
Mimicking the terminology in \cite{Li20}, we may say the Hodge--de Rham sequence for $\mathcal{X}_n$
is split degenerate up to degree $T$.
We need some preparations.

\begin{lemma}
\label{length lemma}
\leavevmode
\begin{enumerate}
\item If $e=1$, we have
$\ell\big( \mathrm{Tor}_1^{\s}(k, \mathcal{O}_K) \big) = \ell\big( \mathrm{Tor}_1^{\s}(k, \varphi_{\s,*}\mathcal{O}_K) \big)$.
\item If $e > 1$, we have
$\ell\big( \mathrm{Tor}_1^{\s}(k, \mathcal{O}_K) \big) < \ell\big( \mathrm{Tor}_1^{\s}(k, \varphi_{\s,*}\mathcal{O}_K) \big)$.
\item Let $M$ be a finitely generated $p^{\infty}$-torsion $\s$-module without $u$-torsion, then
\[
\ell\big( M \otimes_{\s} \mathcal{O}_K \big) = \ell\big( M \otimes_{\s, \varphi_{\s}} \mathcal{O}_K \big).
\]
\end{enumerate}
\end{lemma}

Here $\ell(-)$ denotes length of the $\mathcal{O}_K$-module.

\begin{proof}
For (1) and (2): Simply note that $\mathrm{Tor}_1^{\s}(k, \mathcal{O}_K)$ is the $u^e$-torsion in $k = \s/(p,u)$,
whereas the module $\mathrm{Tor}_1^{\s}(k, (\varphi_{\s})_* \mathcal{O}_K)$ is the $u^e$-torsion in $k \otimes_{\s, \varphi_{\s}} \s = \s/(p,u^p)$.

For (3): It is easy to see that the condition guarantees a finite filtration on $M$ with graded pieces given by $\s/p \cong k[\![u]\!]$.
Indeed we just make an induction on the exponent of powers of $p$ that annihilates $M$ and contemplate with the sequence
\[
0 \to M[p] \to M \to M/M[p] \to 0.
\]
Hence the equality of lengths follows from the equality of $\s/p \otimes_{\s, \varphi_{\s}} \s \simeq \s/p$.
\end{proof}

\begin{lemma}
\label{split filtration lemma}
Let $F \subset M$ be an inclusion of finitely generated $W(k)$-modules.
If the induced maps $F/p^n \to M/p^n$ are injective for any $n \in \mathbb{N}$,
then $F$ is a direct summand in $M$.
\end{lemma}

\begin{proof}
Denote $M/F$ by $C$, the condition implies that 
$M[p^n] \twoheadrightarrow C[p^n]$ for all $n$.
Write the torsion submodule $C_{\mathrm{tor}}$
as direct sums of cyclic torsion $W(k)$-modules,
and use the condition, we see that each cyclic summand admits a section
back to $M$.
This way we see that the extension class restricted to $0$
in $\mathrm{Ext}^1_{W(k)}(C_{\mathrm{tor}}, F)$,
hence it must come from a class in 
$\mathrm{Ext}^1_{W(k)}(C/C_{\mathrm{tor}}, F)$.
But now $C/C_{\mathrm{tor}}$ is finitely generated torsion free $W(k)$-module,
which is well-known to be free $W(k)$-module,
hence the extension group is $0$.
\end{proof}

\begin{proof}[Proof of \Cref{integral HdRSS}]
Let us show (1) and (2).
The case of $n = \infty$ follows from the finite $n$ case:
for (1) this is by left exactness of taking inverse limit,
for (2) this follows from \Cref{prop Zp module structure}.
Now we assume $n \in \mathbb{N}$, the degeneration statement is equivalent to equality of lengths
\[
\ell(\rH^m_{\dR}(\mathcal{X}_n)) = \sum_{i+j = m} \ell(\rH^{i,j}(\mathcal{X}_n)),
\]
for any $m < T$.
Note that by the mere existence of the Hodge--de Rham spectral sequence, we have the inequality
\[
\ell(\rH^m_{\dR}(\mathcal{X}_n)) \leq \sum_{i+j = m} \ell(\rH^{i,j}(\mathcal{X}_n))
\]
for free for any $m$. Below we shall try to show the converse inequality for $m<T$.

To that end, by the same argument as in the first paragraph of \cite[Proof of Theorem 1.1]{Li20},
the liftability condition implies that the Hodge--Tate spectral sequence degenerates up to degree $p-1$,
see \cite[Remark 4.13 and Proposition 4.14]{BS19}, \cite[Proposition 3.2.1]{ALB21},
and \cite[Corollary 4.23]{LL20}.
In particular, since $T-1 \leq p-1$ we have a splitting of $\mathcal{O}_K$-modules:
$\rH^m_{\mathrm{HT}}(\mathcal{X}_n) \simeq \bigoplus_{i + j = m} \rH^{i,j}(\mathcal{X}_n)$
for any $m < T$.
Here the Hodge--Tate cohomology of $\mathcal{X}_n$ is defined to be the quasi-syntomic
cohomology of the mod $p^n$ of the Hodge--Tate sheaf $\overline{\mathcal{O}}_{\Prism}$.
What remains to be shown is an inequality of length: 
\[
\ell(\rH^m_{\mathrm{HT}}(\mathcal{X}_n)) \leq \ell(\rH^m_{\dR}(\mathcal{X}_n)).
\]
By the Hodge--Tate and de Rham comparisons of prismatic cohomology \cite[Theorem 4.10 and Corollary 15.4]{BS19}, 
we have equalities:
\[
\ell(\rH^m_{\mathrm{HT}}(\mathcal{X}_n)) = \ell\big( \rH^m_{\qsyn}(\mathcal{X}, \Prism_n) \otimes_{\s} \mathcal{O}_K \big)
+ \ell\big( \mathrm{Tor}_1^{\s}(\mathfrak{M}^{m+1}_n, \mathcal{O}_K) \big)
\]
and
\[
\ell(\rH^m_{\mathrm{dR}}(\mathcal{X}_n)) = \ell\big( \rH^m_{\qsyn}(\mathcal{X}, \Prism_n) \otimes_{\s, \varphi_{\s}} \mathcal{O}_K \big)
+ \ell\big( \mathrm{Tor}_1^{\s}(\mathfrak{M}^{m+1}_n, (\varphi_{\s})_* \mathcal{O}_K) \big).
\]
Now the desired inequality between length of Hodge--Tate and de Rham cohomology follows from
the definition of $T$, the inequality $m < T$, the \Cref{cor control u-torsion}, and the \Cref{length lemma}.
This finishes the proof of (1).

Next we turn to (2), note that by \Cref{cor control u-torsion} (3), 
if $\mathfrak{M}^{T}_n$ were nonzero, it would necessarily be a direct sum of $k$
as an $\s$-module.
Then \Cref{length lemma} (2) shows that when $e > 1$, the strict inequality
\[
\ell(\rH^{T-1}_{\mathrm{HT}}(\mathcal{X}_n)) < \ell(\rH^{T-1}_{\dR}(\mathcal{X}_n))
\]
holds, which violates the fact that the left hand side is the same as sum of
length of Hodge cohomology groups whereas the right hand side is at most that sum.
Hence we arrive at a contradiction.
The vanishing of $\mathfrak{M}^{m}_n$ when $m < T$ already follows from \Cref{cor control u-torsion} (1).
The statement concerning structure of prismatic cohomology now follows from \Cref{prop Zp module structure}.

Now we turn to (3): $e=1$, hence $T = p$. 
In this case, the statement (1) we proved above implies that
for any $i \leq p$ and any $j$, the map
$\rH^i(\mathcal{X}_n, \Fil^j_{\rH}) \to \rH^i_{\dR}(\mathcal{X}_n)$
is injective.
Hence the submodule $\rH^i(\mathcal{X}, \Fil^j_{\rH}) \subset \rH^i_{\dR}(\mathcal{X})$
has the property that it induces an injection modulo any $p^n$.
The desired splitness follows from \Cref{split filtration lemma}.

Lastly we show (4): when $e > 1$.
We follow the argument of \cite[Corollary 3.9]{Li20}.
Using the vanishing statement established in (2),
it follows that we have abstract isomorphism $\rH^m_{\mathrm{HT}}(\mathcal{X}) \simeq \rH^m_{\dR}(\mathcal{X})$
whenever $m < T$.
Hence the argument of loc.~cit.~shows that in the range $m < T$, splitting of the Hodge--Tate filtration
on $\rH^m_{\mathrm{HT}}(\mathcal{X})$
is equivalent to the splitting of the Hodge filtration on $\rH^m_{\dR}(\mathcal{X})$.
We can then finish our proof, as liftability to $\s/(E^2)$ gives the desired splitting of the Hodge--Tate
filtration in the range $m < T \leq p$.
\end{proof}


\begin{remark}
Comparing our \Cref{integral HdRSS}(1) with what Fontaine--Messing obtained \cite[II.2.7.(i)]{FontaineMessing}
(assuming the existence of a lifting over $W$), we seemingly get a stronger statement:
namely loc.~cit~only claims degeneration statement when the differential has target of degree $< p$
whereas ours allow the differential has source of degree $< p$ (so the target can have degree $p$).
However this is due to Fontaine and Messing not trying to squeeze their method to the most optimal,
which is understandable given how many indices they needed to take care of.
Indeed, their \cite[II.2.6.(ii)]{FontaineMessing} implies the map in next degree (following their notation)
$\oplus_{r=1}^t \rH^{m+1}(J_n^{[r]}) \to \oplus_{r=0}^t \rH^{m+1}(J_n^{[r]})$ is injective,
which can be used to strengthen their \cite[II.2.7.(i)]{FontaineMessing},
hence also gaining the extra degeneration statement we obtained here.
\end{remark}

%

Now it is time to prove the main theorem in this subsection.

\begin{proof}[Proof of {\Cref{improving Li20}}]
Fix an $m < T$, and a $j \in \mathbb{N}$.
We consider the map of two $\mathcal{O}_K$-complexes
\[
\mathrm{R\Gamma}(\mathcal{X},\Fil^j_{\rH}) \to \mathrm{R\Gamma_{dR}}(\mathcal{X}/\mathcal{O}_K).
\]
Our \Cref{integral HdRSS} (1), (3) and (4) implies that this map in degree $m$
satisfies the assumption of \cite[Lemma 2.16]{Li20} (with our $m$ being the $n$ in loc.~cit.).
We finish the proof of (1) by combining the conclusion of \cite[Lemma 2.16]{Li20} 
with the definition of Hodge numbers of $X$ and virtual Hodge numbers of $\mathcal{X}_0$.

The fact that (1) implies (2) is rather a brain teaser. In the Hodge diamond of $X$,
all numbers below the middle row, which is the row with total degree given by $\dim(X)$ ($\leq T$ by assumption),
are given by the corresponding virtual Hodge number of $\mathcal{X}_0$.
Hodge symmetry implies that $\mathcal{X}_0$ also knows all numbers above the middle row.
Now for the middle row, simply use the fact that Euler characteristic is locally constant
for any flat family of coherent sheaves.
\end{proof}

For the rest of this subsection, let us specialize to the case of $e = 1$. 
Using knowledge on the Hodge--de Rham spectral sequence,
we have a similar degeneration of the ``Nygaard--Prism'' spectral sequence
up to cohomological degree $p$.

\begin{theorem}
\label{Nygaard--Prism degeneration}
Assume $e=1$ \emph{(}so $\mathcal{O}_K = W$\emph{)}, and let $n \in \mathbb{N} \cup \{\infty\}$.
The map
\[
\rH^i_{\qsyn}(\mathcal{X}, \Fil^j_{\rN}/p^n) \to \rH^i_{\qsyn}(\mathcal{X}, \Prism^{(1)}_n)
\]
is injective, when $i < p$ or $i = p, j \leq p-1$.
\end{theorem}

Recall that when $i=p$ and $j \geq p$, kernels of these maps have been studied in \Cref{prop induced Nygaard filtration} (3).

\begin{proof}
We shall induct on $j$, the case of $j=0$ being trivial.
We need to stare at the following diagram:
\[
\xymatrix{
\rH^i_{\qsyn}(\mathcal{X}, \Fil^j_{\rN}/p^n) \otimes_{\s} (E) \ar[r] \ar@{^{(}->}[d] & 
\rH^i_{\qsyn}(\mathcal{X}, \Fil^{j+1}_{\rN}/p^n) \ar[r] \ar[d] & 
\rH^i(\mathcal{X}, \Omega^{\geq j+1}_{\mathcal{X}/W}/p^n) \ar@{^{(}->}[d]\\
\rH^i_{\qsyn}(\mathcal{X}, \Prism^{(1)}_n) \otimes_{\s} (E) \ar[r] & 
\rH^i_{\qsyn}(\mathcal{X}, \Prism^{(1)}_n) \ar[r] & 
\rH^i_{\dR}(\mathcal{X}_n/W_n).
}
\]
The rows are exact as they are part of long exact sequences, coming from exact sequences of sheaves on $\mathcal{X}_\qsyn$.
The right vertical arrow is injective for all $i \leq p$ thanks to \Cref{integral HdRSS} (1), note that $T=p$ as $e=1$.
The left vertical arrow is injective by induction hypothesis.

Let us first show the statement for $i < p$. Take an element in the kernel of the middle vertical arrow,
by diagram chasing we see that the element comes from an element $\alpha$ in $\rH^i_{\qsyn}(\mathcal{X}, \Fil^j_{\rN}/p^n) \otimes_{\s} (E)$.
Now it suffices to show the image of $\alpha$ in $\rH^i_{\qsyn}(\mathcal{X}, \Prism^{(1)}_n) \otimes_{\s} (E)$ is zero.
Lastly we note that the further image in $\rH^i_{\qsyn}(\mathcal{X}, \Prism^{(1)}_n)$ is zero,
therefore it suffices to know $\rH^i_{\qsyn}(\mathcal{X}, \Prism^{(1)}_n)$ has no $E$-torsion, or equivalently it has no $u$-torsion,
thanks to \Cref{cor control u-torsion} (1).

Finally let us show the statement when $i = p$, and let $j+1 \leq p-1$.
Argue as the previous paragraph, we are reduced to showing: Given an element $\beta_j \in \rH^p_{\qsyn}(\mathcal{X}, \Fil^j_{\rN}/p^n)$
whose image $\gamma$ in $\rH^p_{\qsyn}(\mathcal{X}, \Prism^{(1)}_n)$ is an $E$-torsion, then the image of $\beta_j \otimes E$
in $\rH^p_{\qsyn}(\mathcal{X}, \Fil^{j+1}_{\rN}/p^n)$ is already zero.
To that end, we need the help of another diagram:
\[
\xymatrix{
\rH^p_{\qsyn}(\mathcal{X}, \Fil^{p-1}_{\rN}/p^n) \otimes_{\s} (E) \ar[r] \ar[d] & 
\rH^p_{\qsyn}(\mathcal{X}, \Fil^{p-1}_{\rN}/p^n) \ar[d]
\\
\rH^p_{\qsyn}(\mathcal{X}, \Fil^j_{\rN}/p^n) \otimes_{\s} (E) \ar[r] \ar@{^{(}->}[d] & 
\rH^p_{\qsyn}(\mathcal{X}, \Fil^{j+1}_{\rN}/p^n)\ar[d]\\
\rH^p_{\qsyn}(\mathcal{X}, \Prism^{(1)}_n) \otimes_{\s} (E) \ar[r] & 
\rH^p_{\qsyn}(\mathcal{X}, \Prism^{(1)}_n). \\
}
\]
Although it will not be used, we point out that two vertical arrows in the top square are both injective
because of \Cref{prop induced Nygaard filtration} (2).
Since $\gamma$ is an $E$-torsion, we know it is $(u,p)$-torsion, see \Cref{cor control u-torsion} (3).
Therefore we see $\gamma$ is the image of a $(u,p)$-torsion $\beta_{p-1}$ in $\rH^p_{\qsyn}(\mathcal{X}, \Fil^{p-1}_{\rN}/p^n)$
thanks to \Cref{prop induced Nygaard filtration} (1).
By induction hypothesis, we see the image of $\beta_{p-1}$ in $\rH^p_{\qsyn}(\mathcal{X}, \Fil^j_{\rN}/p^n)$ is precisely $\beta_j$.
Now we are done as $E \cdot \beta_{p-1} = 0$ in $\rH^p_{\qsyn}(\mathcal{X}, \Fil^{p-1}_{\rN}/p^n)$.
\end{proof}

\section{Crystalline cohomology in boundary degree}
\label{arithmetic applications}
\begin{notation}
Throughout this section let us fix $n \in \mathbb{Z} \cup \{\infty\}$,
and fix $e,i$ such that $e \cdot i = p-1$.
Let $S$ be the PD envelope of $\s \twoheadrightarrow \mathcal{O}_K$,
let $c_1 = \varphi(E)/p \in S^{\times}$.
Denote $\s_n \coloneqq \s/p^n$ and $S_n \coloneqq S/p^n$.
Let $\mathcal{X}$ be a smooth proper formal scheme over $\Spf(\mathcal{O}_K)$.
Let $\mathfrak{M} \coloneqq \rH^i_{\qsyn}(\mathcal{X}, \Prism^{(1)}_n)$,
let $\mathcal{M} \coloneqq \rH^i_{\cris}(\mathcal{X}_n/S_n, \mathcal{O}_{\cris}) \cong \rH^i_{\qsyn}(\mathcal{X}, \dR_{-/\s}/p^n)$,
and finally let $V \coloneqq \rH^{i+1}_{\qsyn}(\mathcal{X}, \Prism_n)[u^\infty]$.
We use $\mathrm{Frob}_k$ to denote the Frobenius on $k$.
\end{notation}

Recall that, by \Cref{cor control u-torsion}, the module $\mathfrak{M}$ is $u$-torsion free and
the Frobenius $\s$-module $V$ is an \'{e}tale $\varphi$-module over $k$. 
Also recall \cite[Theorem 5.2]{BS19} and \cite[Theorem 3.5 and Lemma 7.16]{LL20} that,
we have a short exact sequence of Frobenius $S$-modules:
\[
\label{the sequence}
\tag{\epsdice{2}}
0 \to \mathfrak{M} \otimes_{\s_n} S_n \to \mathcal{M} \to 
\mathrm{Tor}_1^{\s_n}(V, \varphi_* S_n) \cong \mathrm{Tor}_1^{\s_1}(V, \varphi_* S_1) \eqqcolon \overline{M} \to 0,
\]
where the last equality follows from the fact that $S_1 = \s_1 \otimes^{\mathbb{L}}_{\s_n} S_n$.
Here by assumption on $\mathcal{X}$, we know $\mathfrak{M}$
is finitely generated over $\s$ and we can replace completed tensor with tensor to ease
notation a little bit.

Let us give a functorial description of $\overline{M}$.
\begin{lemma}
\label{identify Tor1}
Let $N$ be an $\s_1$-module, then we have identifications of $\s_1$-modules:
\begin{enumerate}
\item $\mathrm{Tor}_1^{\s_1}(V, N) \cong V \otimes_k \big(N[u]\big)$; and
\item $\mathrm{Tor}_1^{\s_1}(V, \varphi_* N) \cong \mathrm{Frob}^*_k(V) \otimes_k \big(N[u^p]\big)$.
\end{enumerate}
Here the $\s$-module structures on right hand sides are via the second factor.
\end{lemma}

In particular we have $\overline{M} \cong \mathrm{Frob}^*_k(V) \otimes_k S_1[u^p]$.

\begin{proof}
Let us prove (2) here as the proof of (1) follows a similar argument.
Note that
\[
V \otimes^{\LL}_{\s_1, \varphi} N = V \otimes^{\LL}_{k, id} k \otimes^{\LL}_{\s_1, \varphi} N
= V \otimes^{\LL}_{k, id} k \otimes^{\LL}_{\s_1, \varphi} \s_1 \otimes^{\LL}_{\s_1} N.
\]
Then one simply computes
\[
k \otimes^{\LL}_{\s_1, \varphi} \s_1 \cong \s_1/u^p,
\]
with $k$ module structure via Frobenius on $k$.
Therefore the above derived tensor becomes
\[
\mathrm{Frob}_k^* V \otimes_k \mathrm{Tor}_1^{\s_1}(\s_1/u^p, N) \cong \mathrm{Frob}^*_k V \otimes_k \big(N[u^p]\big)
\]
\end{proof}

In the following we shall describe the induced filtrations, divided Frobenii and connections
on all terms of the sequence \ref{the sequence}.

\subsection{Understand filtrations}
Recall \cite[Theorem 4.1]{LL20} (and references thereof) we have filtered isomorphisms:
\[
\mathrm{R\Gamma}(\mathcal{X},\Fil^{\bullet}_{\rH}\dR^{\wedge}_{-/\s}) \xrightarrow{\cong}  \mathrm{R\Gamma}_{\cris}(\mathcal{X}/S, \mathcal{I}_{\cris}^{\bullet}).
\]
By the above identification, we need to understand the Hodge filtration on the derived de Rham cohomology of $\mathcal{X}/\s$.

\begin{lemma}
\label{Fil i injective}
We have the following.
\begin{enumerate}
\item The map $\rH^i_{\qsyn}(\mathcal{X}, \Fil^i_{\rN}\Prism^{(1)}_n) \to \mathfrak{M}$ is injective.
\item The map $\rH^i(\mathcal{X}, \Fil^i_{\rH}\dR^{\wedge}_{-/\s}/p^n) \to \mathcal{M}$ is injective.
\end{enumerate}
\end{lemma}

This fact has appeared in the proof of \cite[Theorem 7.22]{LL20},
for the convenience of readers let us reproduce its proof below.
The key point is that the inequality $e \cdot (i-1) < p-1$ implies the $i$-th prismatic cohomology is $u$-torsion free,
which in turn guarantee injectivity.

\begin{proof}
By \cite[Corollary 4.23]{LL20} and diagram chasing, we know the kernel of
\[
\rH^i_{\qsyn}(\mathcal{X}, \Fil^i_{\rN}\Prism^{(1)}_n) \to \mathfrak{M}
\]
surjects onto the kernel of 
\[
\rH^i(\mathcal{X}, \Fil^i_{\rH}\dR^{\wedge}_{-/\s}/p^n) \to \mathcal{M}.
\]
Hence it suffices to prove (1).

By \cite[Lemma 7.8]{LL20} we know the $i$-th divided Frobenius 
$\varphi_i \colon \rH^i_{\qsyn}(\mathcal{X}, \Fil^i_{\rN}\Prism^{(1)}_n) \to 
\rH^i_{\qsyn}(\mathcal{X}, \Prism_n)$ is an isomorphism.
Combining with \Cref{cor control u-torsion} (1) we see that the cohomology
of Nygaard filtration has no finite length sub-$\s$-module.
Finally \cite[Proposition 7.12]{LL20} says the kernel of the map in (1) must be a finite length sub-$\s$-module,
hence zero.
\end{proof}

\begin{notation}
\label{filtration notation}
We denote the image of above injections by $\Fil^i \mathfrak{M}$ and $\Fil^i \mathcal{M}$ respectively.
\end{notation}

The submodule $\Fil^i \mathcal{M} \subset \mathcal{M}$ induces filtrations
on the first and the third term in the sequence \ref{the sequence}.
For instance 
\[
\Fil^i \big(\mathfrak{M} \otimes_{\s_n} S_n\big) \coloneqq \big(\mathfrak{M} \otimes_{\s_n} S_n\big) \cap \Fil^i \mathcal{M}
\]
where the intersection happens inside $\mathcal{M}$, and 
\[
\Fil^i \overline{M} \coloneqq \mathrm{Im}(\Fil^i \mathcal{M} \to \overline{M}).
\]
Let us investigate these filtrations.

Let $\mathcal{I}^{[i]} \subset S$ be the $i$-th PD filtration ideal,
which is $p$-completely generated by $\geq i$-th divided powers of $E(u)$ in $S$.
Note that the quotient $S/\mathcal{I}^{[i]}$ is $p$-torsion free, hence the ideal 
$\mathcal{I}^{[i]}_n \coloneqq \mathcal{I}^{[i]}/p^n \subset S_n$
can be regarded as the $i$-th PD filtration ideal on $S_n$.

Recall \cite[\S 4]{LL20} that we have a commutative diagram of sheaves on $\big(\mathcal{O}_K\big)_{\qsyn}$:
\[
\xymatrix{
(E(u)^j) \otimes_{\s} \Prism^{(1)} \ar[r] \ar[d] &
\Fil^j_{\rN} \Prism^{(1)} \ar[r] \ar[d] &
\Prism^{(1)} \ar[d] \\
\mathcal{I}^{[j]} \otimes_{S} \dR^{\wedge}_{-/\s} \ar[r] &
\Fil^j_{\rH} \dR^{\wedge}_{-/\s}  \ar[r] &
\dR^{\wedge}_{-/\s}. \\
}
\]
\begin{lemma}
\label{identify quotient of filtrations}
The diagram above induces the following commutative diagram of sheaves on $\big(\mathcal{O}_K\big)_{\qsyn}$:
\[
\xymatrix{
0 \ar[r] & 
\frac{\Fil^j_{\rN} \Prism^{(1)}}{(E(u)^j) \otimes_{\s} \Prism^{(1)}} \ar[r] \ar[d] &
\frac{\Prism^{(1)}}{(E(u)^j) \otimes_{\s} \Prism^{(1)}} \ar[r] \ar[d] &
\frac{\Prism^{(1)}}{\Fil^j_{\rN} \Prism^{(1)}} \ar[r] \ar[d] & 0 \\
0 \ar[r] & 
\frac{\Fil^j_{\rH} \dR^{\wedge}_{-/\s}}{\mathcal{I}^{[j]} \hat{\otimes}_{S} \dR^{\wedge}_{-/\s}} \ar[r] &
\frac{\dR^{\wedge}_{-/\s}}{\mathcal{I}^{[j]} \hat{\otimes}_{S} \dR^{\wedge}_{-/\s}} \ar[r] &
\frac{\dR^{\wedge}_{-/\s}}{\Fil^j_{\rH} \dR^{\wedge}_{-/\s}} \ar[r] & 0
}
\]
which has short exact rows, and vertical arrows are isomorphisms if $j \leq p$ and remains so after derived mod $p^n$.
\end{lemma}

\begin{proof}
The derived mod $p^n$ statement follows from the fact that derived mod $p^n$ is exact.
It suffices to show two of the three vertical arrows are isomorphisms.

Using $\dR^{\wedge}_{-/\s} \cong S \hat{\otimes}_{\s} \Prism^{(1)}$, 
see \cite[Theorem 5.2]{BS19} and \cite[Theorem 3.5]{LL20},
the middle vertical arrow is identified with 
\[
\Prism^{(1)} \hat{\otimes}_{\s} \left (\frac{\s}{(E(u)^j)} \longrightarrow \frac{S}{\mathcal{I}^{[j]}}\right ),
\]
hence it suffices to note that the ring map $\frac{\s}{(E(u)^j)} \to \frac{S}{\mathcal{I}^{[j]}}$ is an isomorphism.

The right vertical arrow is an isomorphism (thanks to \cite[Corollary 4.23]{LL20}).
\end{proof}

\begin{proposition}
\label{residual filtration}
The map $\Fil^i \mathcal{M} \to \overline{M}$ is surjective.
Hence $\Fil^i \overline{M} = \overline{M}$.
\end{proposition}

\begin{proof}
We stare at the following map between long exact sequences:
\[  
\xymatrix{ 
\rH ^i_{\qsyn}(\mathcal{X}, \Fil^i_{\rm N}\Prism^{(1)}/p^n) \ar[r] \ar[d] & 
\mathfrak{M} \ar[d] ^\iota \ar[r] &
\rH^{i}_{\qsyn} (\mathcal{X}, \big(\Prism^{(1)}/ \Fil^i_{\rm N} \Prism^{(1)} \big)/p^n ) \ar[d]^{\simeq} \\
\rH ^i_{\qsyn} (\mathcal{X}, \Fil^i_{\rm H} \dR^\wedge _{-/ \s}/p^n) \ar[r] & 
\mathcal{M} \ar[r] &
\rH^{i}_{\qsyn} (\mathcal{X}, \big(\dR^\wedge_{-/\s}/ \Fil^i_{\rm H} \dR ^\wedge _{-/ \s} \big)/p^n ) \\
}
\]
Chasing diagram, we see that it suffices to show the top right horizontal arrow is a surjection.
Indeed, granting the surjectivity assertion, we get that the summation map
\[
\Fil^i \mathcal{M} \oplus \mathfrak{M} \to \mathcal{M}
\]
is a surjection. Projection further to $\overline{M}$ kills the second factor above, hence we get the desired surjectivity.

Lastly we prolong the top long exact sequence:
\[
\rH^{i}_{\qsyn} (\mathcal{X}, \big(\Prism^{(1)}/ \Fil^i_{\rm N} \Prism^{(1)} \big)/p^n )
\to \rH ^{i+1} _{\qsyn}(X_n, \Fil^i_{\rm N}\Prism^{(1)}/p^n) \xrightarrow{\iota} 
\rH ^{i+1}_{\qsyn}(\mathcal{X}, \Prism^{(1)}/p^n).
\]
We are reduced to showing $\iota$ is injective, which is exactly \Cref{prop induced Nygaard filtration} (2):
note that $e \cdot \big( (i+1) - 1\big) = e \cdot i = p-1$.
\end{proof}


Using what is proved in the above proposition, 
we can also understand $\Fil^i \big(\mathfrak{M} \otimes_{\s_n} S_n\big)$. 
The diagram before \Cref{identify quotient of filtrations} implies 
that we have a natural map
$\mathfrak{M} \otimes_{\s_n} \mathcal{I}^{[i]}_n \to 
\Fil^i \big(\mathfrak{M} \otimes_{\s_n} S_n\big)$.
Since the map $\Prism^{(1)} \to \dR^{\wedge}_{-/\s}$ of quasi-syntomic sheaves
is filtered, we also have a natural map 
$\Fil^i \mathfrak{M} \to \Fil^i \big(\mathfrak{M} \otimes_{\s_n} S_n\big)$.
We remind readers \Cref{filtration notation} that the source denotes $\rH^i$ of $i$-th mod $p^n$ Nygaard filtration.

\begin{proposition}
\label{Fil on tensor product}
The summation map $\Fil^i \mathfrak{M} \oplus 
\big(\mathfrak{M} \otimes_{\s_n} \mathcal{I}^{[i]}_n\big) \to \Fil^i \big(\mathfrak{M} \otimes_{\s_n} S_n\big)$
is surjective.
\end{proposition}

\begin{proof}
Note that 
\[
\frac{\mathfrak{M} \otimes_{\s_n} S_n}{\mathfrak{M} \otimes_{\s_n} \mathcal{I}^{[i]}_n}
= \mathfrak{M} \otimes_{\s_n} \frac{S_n}{\mathcal{I}^{[i]}_n}
= \mathfrak{M}/(E^i).
\]
In the last equality, we use the fact that $i < p$ implies
$S_n/\mathcal{I}^{[i]}_n = \s_n/(E^i)$.
Therefore any element $x$ in $\mathfrak{M} \otimes_{\s_n} S_n$
can be written as $x = y + z$ with $y \in \mathfrak{M}$ and $z$ is in the image 
of $\mathfrak{M} \otimes_{\s_n} \mathcal{I}^{[i]}_n$.
Hence we have 
\[
\Fil^i \big(\mathfrak{M} \otimes_{\s_n} S_n\big) = 
\Big(\Fil^i \big(\mathfrak{M} \otimes_{\s_n} S_n\big) \cap \mathfrak{M}\Big)
+ \mathrm{Im}(\mathfrak{M} \otimes_{\s_n} \mathcal{I}^{[i]}_n).
\]
It suffices to show
\[
\Fil^i \big(\mathfrak{M} \otimes_{\s_n} S_n\big) \cap \mathfrak{M}
= \Fil^i \mathfrak{M} \coloneqq \rH ^i_{\qsyn}(\mathcal{X}, \Fil^i_{\rm N}\Prism^{(1)}/p^n),
\]
which exactly follows from chasing the diagram in the proof of \Cref{residual filtration}.
\end{proof}

\begin{corollary}
\label{degree p-1 is Breuil}
Let $e=1$ and $i=p-1$.
Then the triple $(\mathcal{M}, \Fil^{i}\mathcal{M}, \varphi_i)$
is an object in $\Mod_{S, \tor}^{\varphi, p-1}$.
\end{corollary}

\begin{proof}
Note that the map $\Fil^i\mathcal{M} \to \mathcal{M}$ is injective by \Cref{Fil i injective}.
We need to show admissibility, i.e.~the image $\varphi_i$ generates $\mathcal{M}$.
To that end, we shall explain why both images of 
$\varphi_i \colon \Fil^i\big(\mathfrak{M} \otimes_{\s_n} S_n\big) \to \Fil^i\big(\mathfrak{M} \otimes_{\s_n} S_n\big)$
and
$\overline{\varphi_i} \colon \Fil^{i} \overline{M} = \overline{M} \to \overline{M}$
generates the target.
For the latter, it follows from the $e=1$ case of 
\Cref{compute residual divided Frobenius}.
For the former, just note that the Nygaard filtration,
a.k.a.~$\Fil^i\mathfrak{M}$ (see \Cref{filtration notation})
already has its image of $\varphi_i$ generating the module,
thanks to \cite[Lemma 7.8.(3)]{LL20}.
\end{proof}
\subsection{Compute divided Frobenius}
Next we discuss the divided Frobenius on $\Fil^i$ of terms in the sequence \ref{the sequence}.
We will use $\varphi_i$ to denote the divided Frobenius on both Nygaard and Hodge filtrations,
hopefully readers can tell them apart by looking at the source of the arrow to see which divided 
Frobenius we are using.

Recall \cite[Remark 4.24]{LL20} that when $j \leq p-1$, the semi-linear Frobenius $\varphi$ on $\dR^{\wedge}_{-/\s}$
becomes uniquely divisible by $p^j$ when restricted to the sub-quasi-syntomic sheaf 
$\Fil^i_{\rH}\dR^{\wedge}_{-/\s}$ (c.f.~\cite[p.~10]{Bre98}), which we denote by $\varphi_j$.
The divided Frobenius on Nygaard and Hodge filtrations are related by:
\[
\xymatrix{
\Fil^j_{\rN}\Prism^{(1)} \ar[r]^-{\varphi_j} \ar[d]^-{\iota} & 
\Prism \ar[d]^-{1 \otimes c_1^{j}} \\
\Fil^j_{\rH}\dR^{\wedge}_{-/\s} \ar[r]^-{\varphi_j} &
\dR^{\wedge}_{-/\s} \cong \Prism \hat{\otimes}_{\s, \varphi} S,
}
\]
as one computes: $\frac{\varphi}{\varphi(E)^j} \cdot (\frac{\varphi(E)}{p})^{j} = \frac{\varphi}{p^j}$.
Restricting further to $\mathcal{I}^{[j]} \hat{\otimes}_{\varphi, \s} \Prism \subset \Fil^j_{\rH}\dR^{\wedge}_{-/\s}$,
the divided Frobenius is related to the (semi-linear) prismatic Frobenius via:
\[
\xymatrix{
\mathcal{I}^{[j]} \hat{\otimes}_{\varphi, \s} \Prism \ar[r]^-{\varphi_j \otimes \varphi} \ar[d]^-{\iota} & 
S \hat{\otimes}_{\varphi, \s} \Prism \ar[d]^-{\cong} \\
\Fil^j_{\rH}\dR^{\wedge}_{-/\s} \ar[r]^-{\varphi_j} &
\dR^{\wedge}_{-/\s},
}
\]
where the $\varphi_j$ and $\varphi$ on the top arrow are respectively the divided Frobenius on $\mathcal{I}^{[j]} \subset S$
and the semi-linear Frobenius on $\Prism$.
Since we assumed $e \cdot i = p-1$, in particular $i \leq p-1$.
From the discussion, we immediately get the following.
\begin{lemma}
\label{divided Frobenius descends}
Restricting the divided Frobenius $\varphi_i \colon \Fil^i \mathcal{M} \to \mathcal{M}$
to $\Fil^i\big(\mathfrak{M} \otimes_{\s_n} S_n\big)$, the image lands in the submodule $\mathfrak{M} \otimes_{\s_n} S_n$.
\end{lemma}

\begin{proof}
By the above discussion, we have a commutative diagram:
\[
\xymatrix{
\Fil^i \mathfrak{M} \oplus \big(\mathcal{I}^{[i]}_n \otimes_{\s_n} \mathfrak{M}\big) \ar[r] \ar[d] & 
\mathfrak{M} \otimes_{\s_n} S_n \ar[d]^-{\iota} \\
\Fil^i \mathcal{M} \ar[r]^-{\varphi_i} &
\mathcal{M},
}
\]
where the top arrow is given by $\big(\varphi_i \otimes c_1^i\big) \oplus \big(\varphi_i \otimes \varphi\big)$.
Our claim follows from \Cref{Fil on tensor product} which says the image of the left vertical arrow is precisely $\Fil^i\big(\mathfrak{M} \otimes_{\s_n} S_n\big)$.
\end{proof}

Consequently the divided Frobenius $\varphi_i \colon \Fil^i \mathcal{M} \to \mathcal{M}$ descends 
to a semi-linear map 
$\Fil^i \overline{M} = \overline{M} \to \overline{M}$ (see \Cref{residual filtration}), which we refer to as the \emph{residual divided Frobenius}.
Our next task is to relate this residual divided Frobenius with the Frobenius on $V$.

To that end, we factorize the divided Frobenius on $i$-th Hodge filtration as:
\[
\label{factorize divided Frobenius}
\tag{\epsdice{3}}
\Fil^{i}_{\rH} \dR^{\wedge}_{-/\s} \xrightarrow{\alpha} 
\Prism \widehat{\otimes}_{\s} \mathcal{I}^{[i]} \xrightarrow{\mathrm{id} \otimes \varphi_{i}}
\Prism \widehat{\otimes}_{\s} \varphi_* S.
\]
Here $\alpha$ is $S$-linear and is defined at the level of sheaves in $\big(\mathcal{O}_K\big)_{\qsyn}$:
Recall \cite[Remark 4.24]{LL20} that on the basis of large quasi-syntomic algebras, we know
\[
\Fil^{i}_{\rH} \dR^{\wedge}_{-/\s} = \sum_{0 \leq j \leq i} \mathcal{I}^{[i-j]} \hat{\otimes}_{\s} \Fil^j_{\rN}{\Prism}^{(1)}.
\]
Therefore the linear Frobenius
\[
\Prism \widehat{\otimes}_{\s} \varphi_* S \cong \dR^{\wedge}_{-/\s} \xrightarrow{\beta} \Prism \widehat{\otimes}_{\s} S
\]
restricted to the $i$-th Hodge filtration lands in $\Prism \widehat{\otimes}_{\s} \mathcal{I}^{[i]}$,
and compose further with the $i$-th divided Frobenius on the second factor gives the semi-linear divided Frobenius.

\begin{lemma}
The map $\Fil^{i}_{\rH} \dR^{\wedge}_{-/\s} \xrightarrow{\alpha} 
\Prism \widehat{\otimes}_{\s} \mathcal{I}^{[i]}$ induces a commutative diagram:
\[
\xymatrix{
\Fil^i\big(\mathfrak{M} \otimes_{\s_n} S_n\big) \ar[r] \ar[d] & \rH^i_{\qsyn}(\mathcal{X}, \Prism_n) \otimes_{\s_n} \mathcal{I}^{[i]}_n \ar[d] \\
\Fil^i \mathcal{M} \ar[r]^-{\alpha} & \rH^i_{\qsyn}(\mathcal{X}, \Prism_n \hat{\otimes}_{\s_n} \mathcal{I}^{[i]}_n).
}
\]
\end{lemma}

The content of this lemma is that when we first derived mod $\alpha$ by $p^n$, then take $\rH^i_{\qsyn}(\mathcal{X},-)$,
and finally restrict it to the submodule $\Fil^i\big(\mathfrak{M} \otimes_{\s_n} S_n\big)$, it lands in the submodule
$\rH^i_{\qsyn}(\mathcal{X}, \Prism_n) \otimes_{\s_n} \mathcal{I}^{[i]}_n$ of the target.
This is proved exactly the same way as \Cref{divided Frobenius descends} so let us omit it.
From the above lemma, we know the map $\alpha$ descends to a map
\[
\Fil^i \overline{M} = \overline{M} = \mathrm{Frob}^*_k(V) \otimes_k S_1[u^p]
\xrightarrow{\overline{\alpha}} \mathrm{Tor}_1^{\s_n}(V, \mathcal{I}^{[i]}_n) = V \otimes_{k} \mathcal{I}^{[i]}_1[u].
\]
Here we use $\mathcal{I}^{[i]}_n \otimes^{\mathbb{L}}_{\s_n} \s_1 = \mathcal{I}^{[i]}_1$ and \Cref{identify Tor1} (1)
to obtain the identification of target.

\begin{proposition}
\label{computing bar alpha}
Let $F \colon V \to V$ denote the semi-linear prismatic Frobenius on $V$,
which induces linearized Frobenius
$\tilde{F} \colon \mathrm{Frob}^*_k(V) \to V$.
Then the map
\[
\mathrm{Frob}^*_k(V) \otimes_k S_1[u^p] \xrightarrow{\overline{\alpha}} V \otimes_{k} \mathcal{I}^{[i]}_1[u]
\]
is given by $\tilde{F} \otimes u^{p-1}$.
\end{proposition}

Note that given a $u^p$-torsion in $S_1$, multiplying with $u^{p-1}$ gives us a $u$-torsion in $S_1$,
implicitly in the statement we have used the fact that the inclusion $\mathcal{I}^{[i]}_1[u] \subset S_1[u]$
is a bijection because $i \leq ep-1$.

\begin{proof}
We stare at the following commutative diagram of sheaves on $\big(\mathcal{O}_K\big)_{\qsyn}$:
\[
\xymatrix{
\Fil^{i}_{\mathrm{H}} \dR^{\wedge}_{-/\s} \ar[d] \ar[r]^-{\alpha} & \Prism \widehat{\otimes}_{\s} \mathcal{I}^{[i]} \ar[d]
\\
\Prism \widehat{\otimes}_{\s} \varphi_* S \cong \dR^{\wedge}_{-/\s} \ar[r]^-{\beta} & \Prism \widehat{\otimes}_{\s} S
}
\]
which induces the following commutative diagram:
\[
\xymatrix{
\Fil^{i} \overline{M} \ar[d]_{\cong} \ar[r]^-{\overline{\alpha}} & V \otimes_{k} \mathcal{I}^{[i]}_1[u] \ar[d]^{\cong} \\
\overline{M} \ar[r]^-{\beta} & V \otimes_{k} S_1[u].
}
\]
The left vertical arrow is an isomorphism. As explained right after the statement, the right vertical arrow is also an isomorphism.
Therefore we are reduced to computing the effect on $\rH^{-1}$ of the map
\[
(V \otimes^{\mathbb{L}}_{\s_1, \varphi} \s_1) \otimes^{\mathbb{L}}_{\s_1} S_1 \to V \otimes^{\mathbb{L}}_{\s_1} S_1
\]
induced by the linearized Frobenius 
$V \otimes^{\mathbb{L}}_{\s_1, \varphi} \s_1 \cong \mathrm{Frob}_k^*(V) \otimes_k \s_1/u^p \xrightarrow{\tilde{F} \otimes \mathrm{proj}} V \otimes_k \s_1/u$.
We can choose the following explicit resolution of the above map of $\s_1$-modules:
\[
\xymatrix{
\mathrm{Frob}_k^*(V) \otimes_k \s_1 \ar[r]^-{\mathrm{id} \otimes u^p} \ar[d]_-{\tilde{F} \otimes u^{p-1}} &
\mathrm{Frob}_k^*(V) \otimes_k \s_1 \ar[r]^-{\mathrm{id} \otimes \mathrm{proj}} \ar[d]_-{\tilde{F} \otimes \mathrm{id}} &
\mathrm{Frob}_k^*(V) \otimes_k \s_1/u^p \ar[d]^-{\tilde{F} \otimes \mathrm{proj}} \\
V \otimes_k \s_1 \ar[r]^-{\mathrm{id} \otimes u} &
V \otimes_k \s_1 \ar[r]^-{\mathrm{id} \otimes \mathrm{proj}} &
 V \otimes_k \s_1/u.
}
\]
Tensor the above with $S_1$ over $\s_1$ and look at the induced map on $\rH^{-1}$ yields the conclusion.
\end{proof}

The effect of the second arrow in \ref{factorize divided Frobenius} is very easy to understand: 
we only need to understand the divided Frobenius
$\varphi_i \colon \mathcal{I}^{[i]}_1[u] \xrightarrow{\varphi_{i}} S_1[u^p]$.
Note that we assumed $e \cdot i = p-1$, hence $e=1$ means $i = p-1$.
\begin{lemma}
\label{compute divided Frobenius on ideal}
The $S_1$-module $\mathcal{I}^{[i]}_1[u] = S_1[u]$ is generated by $u^{ep-1}$,
and we have
\[
\varphi_i(u^{ep-1}) = 
\begin{cases}
c_1^{p-1} \in S_1 = S_1[u^p], \text{ when } e=1 \\
0, \text{ when } e>1.
\end{cases}
\]
\end{lemma}

\begin{proof}
The description of $\mathcal{I}^{[i]}_1[u]$ is well-known.
It follows from the explicit description of $\mathcal{I}^{[i]}_1 \subset S_1$,
given in the proof of \Cref{Fil on tensor product}.

Let us choose a lift of $u^{ep-1} \equiv E(u)^{p-1} \cdot u^{e-1}$ to $\mathcal{I}^{[i]}$ and compute
\[
\varphi_i(E(u)^{p-1} \cdot u^{e-1}) = c_1^{p-1} \cdot p^{p-1-i} \cdot u^{ep-p}.
\]
After reducing mod $p$, the right hand side is $0$ if $0 < p-1-i$ which is equivalent to $e > 1$,
and when $e=1$, the right hand side is $c_1^{p-1}$.
\end{proof}

Putting everything together, we arrive at the following:
\begin{proposition}
\label{compute residual divided Frobenius}
The divided Frobenius $\Fil^{i} \mathcal{M} \to \mathcal{M}$
descends to a residual divided Frobenius
\[
\overline{\varphi_i} \colon \Fil^{i} \overline{M} = \overline{M} \to \overline{M}.
\]
After identifying $\overline{M} \cong \mathrm{Frob}_k^*(V) \otimes_k S_1[u^p]$,
we have 
\[
\overline{\varphi_i} =
\begin{cases}
F \otimes c_1^{p-1} \cdot \varphi_{S_1}, \text{ when } e=1 \\
0, \text{ when } e>1.
\end{cases}
\]
%
\end{proposition}
Here we abuse notation a little bit by writing the induced Frobenius on $\mathrm{Frob}_k^*(V)$
still as $F$.
\begin{proof}
The first sentence is \Cref{divided Frobenius descends}.
As for the computation of the residual divided Frobenius, we look at the sequence \Cref{factorize divided Frobenius},
which gives rise to
\[
\mathrm{Frob}_k^*(V) \otimes_k S_1[u^p] \xrightarrow{\overline{\alpha}} V \otimes_k \mathcal{I}^{[i]}_1[u]
\xrightarrow{\mathrm{id} \otimes \varphi_i} V \otimes_{k, \varphi} S_1[u^p].
\]
Combining \Cref{computing bar alpha} and \Cref{compute divided Frobenius on ideal} yields the result.
\end{proof}


\subsection{The connection}
In \cite[Subsection 5.1]{LL20} we explained how one gets a natural connection on the derived de Rham complex
relative to $\s$.
Consequently we see that there is a connection $\nabla \colon \mathcal{M} \to \mathcal{M}$
satisfying $\nabla(f \cdot m) = f' \cdot m + f \cdot \nabla(m)$
for any $f \in S$ and $m \in \mathcal{M}$.
In this section, we shall see that in a strong sense there is a unique such connection.
As a corollary, the connection $\nabla$ preserves the sequence~\ref{the sequence}.
Moreover the compatibility between $\nabla$ and divided Frobenius \cite[Subsection 5.2]{LL20}  will determine
the residual connection on $\overline{M}$.

\begin{notation}
Let $S[\epsilon] \coloneqq S[x]/(x^2)$ and let $S \xrightarrow{\iota_1} S[\epsilon]$
and $S \xrightarrow{\iota_2} S[\epsilon]$ be two ring homomorphisms defined as
$\iota_1(f) = f \otimes 1$ and $\iota_2(f) = f \otimes 1 + f' \otimes \epsilon$.
\end{notation}

\begin{proposition}
\label{identifying connection}
There is a unique $\mathbb{E}_{\infty}$-$S[\epsilon]$-algebra isomorphism
$\dR_{R/\s}^{\wedge} \otimes_{S, \iota_1} S[\epsilon] \to \dR_{R/\s}^{\wedge} \otimes_{S, \iota_2} S[\epsilon]$
which reduces to identity modulo $\epsilon$ and is functorial in formally smooth $\mathcal{O}_K$-algebra $R$.
\end{proposition}

\begin{proof}
One observes the formula $\nabla \mapsto \big(g(m \otimes 1) = m \otimes 1 + \nabla(m) \otimes \epsilon\big)$
gives a bijection between functorial connections on $\dR_{R/\s}^{\wedge}$ and said functorial isomorphisms.
Therefore the existence follows from \cite[Subsection 5.1]{LL20}.

To show uniqueness, we follow the same argument as in the proof of \cite[Theorem 3.13]{LL20}.
First by left Kan extension and quasi-syntomic descent, it suffices to check the uniqueness
when viewing both sides as quasi-syntomic sheaves of $S[\epsilon]$-algebras.
Secondly, by the same argument in loc.c~cit.,
one sees that restricting to the category of quasi-syntomic $\mathcal{O}_K$-algebras of the form
$\mathcal{O}_K\langle X_j^{1/p^\infty}; j \in J \rangle$ for some set $J$
determines such morphisms of $S[\epsilon]$-algebras.
Finally, when $\tilde{R} = \mathcal{O}_K\langle X_j^{1/p^\infty}; j \in J \rangle$, both of the source
and the target are given by $S[\epsilon]\langle X_j^{1/p^\infty}; j \in J \rangle$,
now we need to show $g(X)$ has to be $X$.

To that end, let us assume $g(X^{1/p^n}) = X^{1/p^n} + Y_n \otimes \epsilon$,
then we compute $g(X) = g(X^{1/p^n})^{p^n} = (X^{1/p^n} + Y_n \otimes \epsilon)^{p^n}
\equiv X$ modulo $p^n$. Therefore we conclude $g(X) - X$ is divided by arbitrary powers of $p$,
hence must be $0$ by $p$-adic separatedness of $S[\epsilon]\langle X_j^{1/p^\infty}; j \in J \rangle$.
\end{proof}

\begin{remark}
For any qcqs smooth formal scheme $\mathcal{Y}$ over $\Spf(\mathcal{O}_K)$, the crystal nature
of $\mathrm{R\Gamma}_{\cris}(\mathcal{Y}/S)$ gives a connection on $\mathrm{R\Gamma}_{\cris}(\mathcal{Y}/S)$,
see \cite[p.2~and Lemma 2.8]{BdJ11}.
Note that although in loc.~cit.~the authors were talking about crystals in quasi-coherent modules,
their argument works in our setting of crystals in perfect complexes as $\Omega^{1,pd}_{S/W}$ is finite free over $S$,
so there is no subtlety when derived tensoring it.
Consequently, one gets a connection on $\mathrm{R\Gamma}_{\cris}(\mathcal{Y}/S)$,
and when identifying $\mathrm{R\Gamma}_{\cris}(\mathcal{Y}/S) \cong \dR^{\wedge}_{\mathcal{Y}/\s}$,
our \Cref{identifying connection} shows the ``crystalline'' connection agrees with our ``derived de Rham'' connection.
\end{remark}

Below we explain yet another way to get the connection, via prismatic crystal nature
of prismatic cohomology.
Recall \cite[Construction 7.13]{BS21} that there is a cosimplicial prism
$\big(\s^{(\bullet)}, J^{(\bullet)}\big) \to \mathcal{O}_K \cong \s^{(\bullet)}/J^{(\bullet)}$.
Let $S^{(\bullet)} \to \mathcal{O}_K$ be the similarly defined cosimplicial ring
obtained by taking divided power envelopes of $\s^{\hat{\otimes}_{W} n} \twoheadrightarrow \mathcal{O}_K$
where $[n] \in \Delta$.
Note that there is a map of these cosimplicial rings induced by the Frobenius
$\varphi_{\s}^{\otimes \bullet} \colon \s^{\otimes \bullet} \to \s^{\otimes \bullet}$,
let us explicate this for $\bullet = 0,1$ as we will need it later:
\[
\xymatrix{
\s \cong W[\![u]\!] \ar[r]^-{\iota_1} \ar[d]_{u \mapsto u^p} & W[\![u,v]\!]\{\frac{u-v}{E(u)}\}^{\wedge}
\cong \s^{(1)} \cong W[\![u,v]\!]\{\frac{u-v}{E(v)}\}^{\wedge} \ar[d]_{u \mapsto u^p}^{v \mapsto v^p} &
W[\![v]\!] \cong \s \ar[l]_-{\iota_2} \ar[d]^{v \mapsto v^p} \\
S \cong W[\![u]\!]\langle\!\langle E(u) \rangle\!\rangle \ar[r]^-{\iota_1} &
W[\![u,v]\!]\langle\!\langle E(u), u-v \rangle\!\rangle
\cong S^{(1)} \cong W[\![u,v]\!]\langle\!\langle E(v), u-v \rangle\!\rangle &
W[\![v]\!]\langle\!\langle E(v) \rangle\!\rangle \cong S \ar[l]_-{\iota_2} \\
}
\]
where $\langle\!\langle - \rangle\!\rangle$ denotes $p$-completely adjoining divided powers of the designated
elements.
To see the middle arrow is well-defined we use the fact that $\varphi(E(u))$ and $\varphi(E(v))$ in $S^{(1)}$ is $p$ times a unit,
and adjoining $\varphi(u-v)/p$ as a $\delta$-ring is the same as adjoining divided powers of $u-v$,
see \cite[Corollary 2.39]{BS19}.

Now for any $p$-adically smooth $\mathcal{O}_K$-algebra $R$, we have a functorial isomorphism
of $\mathbb{E}_{\infty}$-$\s^{(1)}$-algebras:
\[
\Prism_{R/\s} \hat{\otimes}_{\s, \iota_1} \s^{(1)} \cong \s^{(1)} \hat{\otimes}_{\iota_2 ,\s} \Prism_{R/\s}
\]
by base change of prismatic cohomology.
Base change the above along the aforesaid map $\s^{(1)} \to S^{(1)}$ 
(and use either \cite[Theorem 5.2]{BS19} or \cite[Theorem 3.5]{LL20})
identifies the left (resp.~right) hand side with
\[
\Prism_{R/\s} \hat{\otimes}_{\s, \iota_1} \s^{(1)} \hat{\otimes}_{\s^{(1)}, \varphi} S^{(1)}
\cong \Prism_{R/\s} \hat{\otimes}_{\s, \varphi} S \hat{\otimes}_{S, \iota_1} S^{(1)}
\cong \dR^{\wedge}_{R/\s} \hat{\otimes}_{S, \iota_1} S^{(1)}
\]
respectively $S^{(1)} \hat{\otimes}_{\iota_2 , S} \dR^{\wedge}_{R/\s}$.
This gives rise another description of the ``crystalline'' connection:
\begin{proposition}
\label{identify crystals}
The following diagram commutes functorially in the $p$-adically smooth $\mathcal{O}_K$-algebra $R$:
\[
\xymatrix{
\Prism_{R/\s} \hat{\otimes}_{\s, \iota_1} \s^{(1)} \ar[r]^-{\cong} \ar[d]_-{\hat{\otimes}_{\s^{(1)}} S^{(1)}} &
\s^{(1)} \hat{\otimes}_{\iota_2 ,\s} \Prism_{R/\s} \ar[d]^-{\hat{\otimes}_{\s^{(1)}} S^{(1)}} \\
\dR^{\wedge}_{R/\s} \hat{\otimes}_{S, \iota_1} S^{(1)} \ar[r]^-{\cong} &
S^{(1)} \hat{\otimes}_{\iota_2 , S} \dR^{\wedge}_{R/\s}.
}
\]
\end{proposition}

\begin{proof}
Base changing the top arrow along $\s^{(1)} \to S^{(1)}$ gives a potentially
different functorial isomorphism in the bottom.
Therefore it suffices to show that there is no non-trivial automorphism
of the quasi-syntomic sheaf of $S^{(1)}$-algebras
$R \mapsto \dR^{\wedge}_{R/\s} \hat{\otimes}_{S, \iota_1} S^{(1)}$.
The same argument as in \cite[Theorem 3.13]{LL20} does the job.
\end{proof}

As a consequence, we know the sequence \ref{the sequence} is stable under the connection.
In fact more generally we have the following.

\begin{corollary}
\label{descend connection}
For any $j \in \mathbb{N}$ and any $n \in \mathbb{N} \cup \{\infty\}$,
the connection on $\rH^j_{\qsyn}(\mathcal{X}, \dR^{\wedge}_{-/\s}/p^n)$
preserves the submodule $\rH^j_{\qsyn}(\mathcal{X}, \Prism^{(1)}/p^n) \otimes_{\s_n} S_n$.
\end{corollary}

\begin{proof}
Under the dictionary between connections and crystals \cite[Lemma 2.8]{BdJ11},
we need to show
the isomorphism (note that both of $\iota_i \colon S \to S^{(1)}$ are $p$-completely flat)
\[
\rH^j_{\qsyn}(\mathcal{X}, \dR^{\wedge}_{-/\s}/p^n) \otimes_{S_n, \iota_1} S^{(1)}_n
\cong  S^{(1)}_n \otimes_{\iota_2, S_n} \rH^j_{\qsyn}(\mathcal{X}, \dR^{\wedge}_{-/\s}/p^n)
\]
preserves the submodule $\rH^j_{\qsyn}(\mathcal{X}, \Prism^{(1)}/p^n) \otimes_{\s_n} S_n$.
This immediately follows from the following commutative diagram
\[
\xymatrix{
\rH^j_{\qsyn}(\mathcal{X}, \Prism/p^n) \hat{\otimes}_{\s_n, \iota_1} \s^{(1)}_n \ar[r]^-{\cong} 
\ar[d]_-{- \hat{\otimes}_{\s^{(1)}_n} S^{(1)}_n} &
\s^{(1)} \hat{\otimes}_{\iota_2 ,\s_n} \rH^j_{\qsyn}(\mathcal{X}, \Prism/p^n) \ar[d]^-{S^{(1)}_n \hat{\otimes}_{\s^{(1)}_n} -} \\
\rH^j_{\qsyn}(\mathcal{X}, \dR^{\wedge}_{-/\s}/p^n) \hat{\otimes}_{S_n, \iota_1} S^{(1)}_n \ar[r]^-{\cong} &
S^{(1)}_n \hat{\otimes}_{\iota_2 , S_n} \rH^j_{\qsyn}(\mathcal{X}, \dR^{\wedge}_{-/\s}/p^n).
}
\]
induced by \Cref{identify crystals}.
\end{proof}

Therefore we see that there is a \emph{residual connection} $\overline{\nabla} \colon \overline{M} \to \overline{M}$.
Recall \cite[Subsection 5.2]{LL20} that the connection $\nabla$ and divided Frobenius $\varphi_i$
are related by the following commutative diagram:
\[
\xymatrix{
\Fil^i \mathcal{M} \ar[r]^-{\varphi_i} \ar[d]_{E(u) \cdot \nabla} &
\mathcal{M} \ar[d]^{c_1 \cdot \nabla} \\
\Fil^i \mathcal{M} \ar[r]^-{u^{p-1} \varphi_i} & \mathcal{M}.
}
\]
Since all maps descend down to $\Fil^i \overline{M} = \overline{M}$, we have the following:
\begin{proposition}
\label{residual connection}
There is a commutative diagram:
\[
\xymatrix{
\overline{M} \ar[r]^-{\varphi_i} \ar[d]_{E(u) \cdot \overline{\nabla}} &
\overline{M} \ar[d]^{c_1 \cdot \overline{\nabla}} \\
\overline{M} \ar[r]^-{u^{p-1} \overline{\varphi_i}} & \overline{M}.
}
\]
Consequently when $e=1$, after identifying 
$\overline{M} \cong \mathrm{Frob}_k^*(V) \otimes_k S_1[u^p]$,
we have $\overline{\nabla}(v \otimes 1) = v \otimes \mathrm{d}\log(c_1)$.
\end{proposition}

Here $\mathrm{d}\log(c_1) = \frac{c_1'}{c_1} = \frac{u^{p-1}}{c_1}$.

\begin{proof}
The existence of such a commutative diagram follows from the preceding discussion
and the fact that both of $\varphi_i$ and $\nabla$ descends to $\overline{M}$
by \Cref{compute residual divided Frobenius} and \Cref{descend connection} respectively.

Start with $v \otimes 1$ at the top left corner and compare the end results of 
the two routes, we arrive at an identity:
\[
\overline{\nabla}(F(v) \otimes 1) \cdot c_1^p + F(v) \otimes (p-1) c_1^{p-1} \cdot c_1' = 0,
\]
where we used the description of $\varphi_i$ in \Cref{compute residual divided Frobenius}.
Now we use the fact that $\overline{M}$ is $p$-torsion
and the fact that $F$ is a bijection to yield the desired conclusion.
\end{proof}

\begin{corollary}\label{cor-rep-Mbar} 
Let $e=1$ and $h = p-1$, then the quadruple $(\overline M, \Fil^{p-1}\overline M , \varphi _{p-1}, \nabla)$ is a Breuil module and 
there is a canonical isomorphism $T_S (\overline{M}) \xrightarrow{\cong} (V \otimes_{W(k)} W(\bar k))^{\varphi=1}$ of representation of $G_K$. 
In particular the resulting Galois representation $T_S (\overline{M})$ is the unramified $\mathbb{F}_p$-representation associated with 
the \'{e}tale $\varphi$-module $V$.
\end{corollary}
\begin{proof} The first part of statement follows from \Cref{identify Tor1},
Proposition \ref{compute residual divided Frobenius}, and Proposition \ref{residual connection}. 
To compute $T _S(\overline M)$, let $I_+ A_{\cris}\subset A_{\cris}$ be the ideal so that $I_+ A_{\cris}$ contains $W(\mathfrak m_{\O_\C ^\flat}) $ and $A_{\cris}/ I_+ A_{\cris} = W(\bar k)$. It is clear that $\varphi ^n (a) \to 0$ for any $a\in I_+ A_{\cris}$ and $I_+ A_{\cris}\cap S = I_+$. By \eqref{eqn-G-action-by-nabla},  $\overline M \otimes_S I_+A_{\cris} $ is stable under the $G_K$-action. 
So we have a canonical map of $G_K$-representations
\[ T_S(\overline M) =(\Fil^h \overline M \otimes_S A_{\cris})^{\varphi_h =1}=  (\overline M \otimes_{S} A _{\cris})^{\varphi_h =1}\to  (\overline M \otimes_A A_{\cris}/ I_+ A_{\cris})^{\varphi_h =1} = (\overline M / I_+ \otimes _{k} \bar k) ^{\varphi_h =1}. \] 
By Proposition \ref{compute residual divided Frobenius}, if we identity $\overline M = {\rm Frob }^* V \otimes _k S_1 [u ^p]$ then 
$\forall x\otimes 1 \in \overline M / I_+$, $\varphi (x\otimes 1) =  F(x)\otimes a_0 ^{p-1}$ with $a_0 = E(0)/ p\in W(k)^\times$.
So $\varphi_h : \overline M/I_+ \to \overline M/I_+$ is bijective.   Using that $\lim\limits_{n \to \infty}\varphi ^n (a) =  0, \forall a\in I_+ A_{\cris}$, we conclude that the above map is an isomorphism $T_S(\overline M) \xrightarrow{\cong} (\overline M / I_+ \otimes _{k} \bar k) ^{\varphi_h =1}$ of $G_K$-representations. Finally, we have to check that $\overline M /I_+ \simeq {\rm Frob }^* V$ as $\varphi$-modules.  Indeed ${\rm Frob }^* V \to  {\rm Frob }^* V \otimes_k S_1[u ^p]/ I_+S= \overline M / I_+$ via $x \mapsto  a_0 (x \otimes 1) $ is the required isomorphism of $\varphi$-modules. 
\end{proof}

\subsection{Fontaine--Laffaille and Breuil modules}
In this subsection we assume $e=1$. 
For simplicity we pick the uniformizer $p$, but all results in this subsection hold true with any other uniformizer.
We shall compare the two approaches of understanding \'{e}tale cohomology, as a Galois representation, from linear algebraic data
on certain crystalline cohomology, which are due to Fontaine--Messing--Kato, and Breuil--Caruso.

First we need a reminder on the filtered comparison between derived de Rham cohomology and crystalline cohomology,
see \cite[Theorem 4.1]{LL20} and references thereof.
\begin{remark}
Let $\mathcal{X}$ be a smooth $p$-adic formal scheme over $\Spf(W)$.
We have filtered isomorphisms:
\[
\mathrm{R\Gamma}(\mathcal{X},\Fil^{\bullet}_{\rH}\dR^{\wedge}_{-/W}) \xrightarrow{\cong}  \mathrm{R\Gamma}_{\cris}(\mathcal{X}/W, \mathcal{I}_{\cris}^{\bullet}),
\]
and
\[
\mathrm{R\Gamma}(\mathcal{X},\Fil^{\bullet}_{\rH}\dR^{\wedge}_{-/\s}) \xrightarrow{\cong}  \mathrm{R\Gamma}_{\cris}(\mathcal{X}/S, \mathcal{I}_{\cris}^{\bullet}).
\]
\end{remark}
In classical references by Fontaine--Messing, Kato and Breuil--Caruso, 
they were considering the right hand side objects of the above isomorphisms.
However we will be thinking about the derived de Rham side, as it is compatible with various techniques
developed by Bhatt--Morrow--Scholze and Bhatt--Scholze.

For the remaining of
this subsection we let $\mathcal{X}$ be a quasi-compact quasi-separated
$p$-adic formal scheme over $\Spf(W)$.
At the derived level, we have the following comparisons:
\begin{proposition}
\label{derived FL to B}
Consider the diagram:
\[
\xymatrix{
\mathcal{X} \ar[rd] \ar[d] \ar[rrd] & & \\
\Spf(W) \ar[r]^{u \mapsto p} & \Spf(\s) \ar[r] & \Spf(W).
}
\]
For any $n \in \mathbb{Z} \cup \{\infty\}$, we have
\begin{enumerate}
\item The canonical maps of $p$-complete cotangent complexes
$\mathbb{L}^{\wedge}_{\mathcal{X}/W} \to \mathbb{L}^{\wedge}_{\mathcal{X}/\s}$ (from right triangle)
and $\mathbb{L}^{\wedge}_{W/\s} \to \mathbb{L}^{\wedge}_{\mathcal{X}/\s}$
(from left triangle)
induces an isomorphism 
\[
\mathbb{L}^{\wedge}_{\mathcal{X}/W}/p^n \oplus \big(\mathbb{L}^{\wedge}_{W/\s} \hat{\otimes}_{W} \mathcal{O}_{\mathcal{X}}\big)/p^n
\xrightarrow{\cong} \mathbb{L}^{\wedge}_{\mathcal{X}/\s}/p^n,
\]
functorial in $\mathcal{X}/W$.
\item The canonical filtered maps of $p$-complete de Rham complexes
$\dR^{\wedge}_{\mathcal{X}/W} \to \dR^{\wedge}_{\mathcal{X}/\s}$ (from right triangle)
and $\dR^{\wedge}_{W/\s} \to \dR^{\wedge}_{\mathcal{X}/\s}$ (from left triangle)
induces a filtered isomorphism 
\[
\big(\dR^{\wedge}_{\mathcal{X}/W} \hat{\otimes}_{W} \dR^{\wedge}_{W/\s}\big)/p^n
\xrightarrow{\cong} \dR^{\wedge}_{\mathcal{X}/\s}/p^n,
\]
functorial in $\mathcal{X}/W$.
\item Moreover the identification in (2) is compatible with divided Frobenii $\varphi_j$ on $j$-th filtration
of both sides for any $j \leq p-1$.
\end{enumerate}
\end{proposition}

In case readers are worried that we do not put any smoothness assumption on $\mathcal{X}$,
just notice that both sides of these equalities are left Kan extended from smooth $\mathcal{X}$'s,
therefore it suffices to prove these statements for smooth affine $\mathcal{X}$'s.
That said, we will prove the statement without the smoothness assumption as the proof
just works in this generality.

\begin{proof}
The finitary $n$ cases follow from the case of $n = \infty$.
Henceforth, we assume $n = \infty$.

(1): This follows from exact triangle of cotangent complexes associated with a triangle of morphisms.

(2): Let $\mathcal{X}_{\s} \coloneqq \mathcal{X} \times_{\Spf(W)} \Spf(\s)$ be the base change.
Then we have $\mathcal{X} \cong \mathcal{X}_{\s} \times_{\Spf(\s)} {\Spf(W)}$.
These objects fit in a commutative diagram:
\[
\xymatrix{
\mathcal{X} \ar[d] \ar[r] & \mathcal{X}_{\s} \ar[d] \ar[r]  & \mathcal{X} \ar[d]\\
\Spf(W) \ar[r]^{u \mapsto p} & \Spf(\s) \ar[r] & \Spf(W).
}
\]
Using K\"{u}nneth formula for derived de Rham complex we obtain a filtered isomorphism:
\[
\dR^{\wedge}_{\mathcal{X}_{\s}/{\s}} \hat{\otimes}_{\s} \dR^{\wedge}_{W/{\s}} \xrightarrow{\cong}  \dR^{\wedge}_{\mathcal{X}/{\s}}.
\]
The base change formula for derived de Rham complex gives us a filtered isomorphism:
\[
\dR^{\wedge}_{\mathcal{X}/W}  \hat{\otimes}_{W} {\s} \xrightarrow{\cong} \dR^{\wedge}_{\mathcal{X}_{\s}/{\s}}.
\]
In both filtered isomorphisms above we put derived Hodge filtration on the derived de Rham complex,
and trivial filtration on the coefficient ring $W$ and $\s$.
Combining these two filtered isomorphisms gives our desired filtered isomorphism.

(3): This just follows from the fact that the two maps
in (2) is compatible with divided Frobenii.
\end{proof}

\begin{remark}
\label{Breuil ring as dR}
Since $\s \xrightarrow{u \mapsto p} W$ is a complete intersection, the $p$-adic derived de Rham complex
$\dR^{\wedge}_{W/\s} \cong S$ is given by Breuil's ring $S$ with the Hodge filtration given by divided powers
of $(u-p)$ and the usual Frobenius $u \mapsto u^p$.
Similarly the mod $p^n$ derived de Rham complex is $S/p^n$ with the induced filtration:
This is because the rings $S$ and $S/\mathcal{I}^{[j]}$ are all $p$-torsion free for any $j \in \mathbb{N}$.
\end{remark}

To obtain consequences at the level of cohomology groups, we need the following abstract lemma.
\begin{lemma}
\label{HTT fact}
Let $\mathcal{C}$ be a stable $\infty$-category.
Let $F \colon \mathbb{N}^{\mathrm{op}} \times \mathbb{N}^{\mathrm{op}} \to \mathcal{C}$
be a map of simplicial sets.
Then for any $0 < m \leq n$, we have a pushout diagram 
\[
\xymatrix{
F(n+1-m, m) \ar[r] \ar[d] & F(n+1-m, m-1) \ar[d] \\
\mathrm{colim}_{i+j \geq n, j \geq m} F(i,j) \ar[r] & \mathrm{colim}_{i+j \geq n, j \geq m-1} F(i,j).
}
\]
In particular, the two inclusions $F(i,j) \to F(i-1, j)$ and $F(i,j) \to F(i, j-1)$ gives rise to
a pushout diagram:
\[
\xymatrix{
\bigoplus_{i+j = n+1, i>0, j>0} F(i,j) \ar@<.4ex>[r] \ar@<-.4ex>[r] & \bigoplus_{i+j = n} F(i,j) \ar[r] & \mathrm{colim}_{i+j \geq n} F(i,j).
}
\]
\end{lemma}

\begin{proof}
The second statement follows from repeatedly applying the first statement and observing that
\[
\mathrm{colim}_{i+j \geq n, j \geq n} F(i,j) = \mathrm{colim}_{(i,j) \geq (0,n)} F(i,j) = F(0,n)
\]
as $(0,n)$ is the final object in $\{(i,j) \geq (0,n)\} \subset \mathbb{N}^{\mathrm{op}} \times \mathbb{N}^{\mathrm{op}}$.
The first statement follows from \cite[Proposition 4.4.2.2]{HTT}: just apply the statement to
\[
\{i+j \geq n, j \geq m-1\} = \{i+j \geq n, j \geq m\} \bigsqcup_{\{(i,j) \geq (n+1-m, m)\}} \{(i,j) \geq (n+1-m, m-1)\}
\]
yields the desired pushout diagram.
\end{proof}

Combining the previous two general statements yield the following. 
\begin{corollary}
\label{filtration auxiliary corollary}
Let $\mathcal{I}^{[\bullet]} \subset S$ be the filtration given by divided powers of $(u-p)$.
The natural maps, for any $q + j \geq m$,
\[
\mathrm{R\Gamma}(\mathcal{X},\Fil^q_{\rH}\dR^{\wedge}_{-/\s}) \hat{\otimes}_W \mathcal{I}^{[j]} \to 
\mathrm{R\Gamma}(\mathcal{X},\Fil^m_{\rH}\dR^{\wedge}_{-/\s})
\]
give rise to an exact triangle, 
\[
\bigoplus_{q+j = \ell +1, i>0, j>0} 
\mathrm{R\Gamma}(\mathcal{X},\Fil^q_{\rH}\dR^{\wedge}_{-/\s})/p^n \hat{\otimes}_{W_n} \big(\mathcal{I}^{[j]}/p^n\big)
\to
\]
\[
\to
\bigoplus_{q+j = \ell} \mathrm{R\Gamma}(\mathcal{X},\Fil^q_{\rH}\dR^{\wedge}_{-/\s})/p^n \hat{\otimes}_{W_n} \big(\mathcal{I}^{[j]}/p^n\big) \to
\mathrm{R\Gamma}(\mathcal{X},\Fil^n_{\rH}\dR^{\wedge}_{-/\s})/p^n
\]
for any $\ell \in \mathbb{Z}$ and any $n \in \mathbb{Z} \cup \{\infty\}$.
\end{corollary}

\begin{proof}
The comparison of filtration \Cref{derived FL to B} (2) shows the right hand side is given by the
$\ell$-th Day convolution filtration on $\mathrm{R\Gamma}(\mathcal{X}, \dR^{\wedge}_{-/\s})/p^n \hat{\otimes}_{W_n} \dR^{\wedge}_{W/\s}/p^n$.
Here the filtered ring $\dR^{\wedge}_{W/\s}/p^n$ is given by $(S/p^n, \mathcal{I}^{[\bullet]})/p^n$,
see \Cref{Breuil ring as dR}.
Last we apply \Cref{HTT fact} to conclude the proof.
\end{proof}

\begin{theorem}[{c.f.~\cite[p.~559 Remarques.(2)]{Bre98}}]
\label{MF and Breuil module}
For any $j, \ell \in \mathbb{Z}$ and any $n \in \mathbb{Z} \cup \{\infty\}$, use 
\[
\mathrm{Im}\big(\rH^j(\mathcal{X},\Fil^\ell_{\rH}\dR^{\wedge}_{-/\s}/p^n) \to 
\rH^j(\mathcal{X}, \dR^{\wedge}_{-/\s})/p^n \big) \eqqcolon \Fil^\ell\rH^j(\mathcal{X}, \dR^{\wedge}_{-/\s}/p^n)
\]
to filter $\rH^j(\mathcal{X}, \dR^{\wedge}_{-/\s}/p^n)$,
and similarly filter $\rH^j(\mathcal{X}, \dR^{\wedge}_{-/W}/p^n)$.
Then we have a filtered isomorphism 
\[
\rH^j(\mathcal{X}, \dR^{\wedge}_{-/W}/p^n) \hat{\otimes}_{W_n} \big(S/p^n\big) \xrightarrow{\cong} \rH^j(\mathcal{X}, \dR^{\wedge}_{-/\s}/p^n).
\]
Moreover it is compatible with the divided Frobenii $\varphi_m$ on $m$-th filtration
of both sides for all $m \leq p-1$.
\end{theorem}

Here again the ring $S/p^n$ is equipped with the divided power ideal filtration.
Concretely we have
\[
\Fil^\ell\rH^j(\mathcal{X}, \dR^{\wedge}_{-/\s}/p^n) = \sum_{r + s = \ell} \Fil^r\rH^j(\mathcal{X}, \dR^{\wedge}_{-/W}/p^n) \hat{\otimes}_{W_n} \big(\mathcal{I}^{[s]}/p^n\big)
\]
as sub-$W_n$-modules inside $\rH^j(\mathcal{X}, \dR^{\wedge}_{-/\s}/p^n) \hat{\otimes}_{W_n} \big(S/p^n\big) \xrightarrow{\cong} \rH^j(\mathcal{X}, \dR^{\wedge}_{-/\s}/p^n)$.

\begin{proof}
By \Cref{filtration auxiliary corollary}, it suffices to show the exact triangle obtained
induces short exact sequence after applying $\rH^q$.
To that end, it suffices to show the map 
\[
\rH^q(\mathcal{X},\Fil^\ell_{\rH}\dR^{\wedge}_{-/\s}/p^n) \hat{\otimes}_{W_n} \big(\mathcal{I}^{[j+1]}/p^n\big)
\to
\rH^q(\mathcal{X},\Fil^\ell_{\rH}\dR^{\wedge}_{-/\s}/p^n) \hat{\otimes}_{W_n} \big(\mathcal{I}^{[j]}/p^n\big)
\]
is injective for any $q,j,\ell,n$. But this follows from the fact that
$\big(\mathcal{I}^{[j]}/p^n\big)/\big(\mathcal{I}^{[j+1]}/p^n\big) \simeq W_n \cdot \gamma_j(u-p)$
is $p$-completely flat over $W_n$.
The compatibility with divided Frobenii was checked in \Cref{derived FL to B} (3).
\end{proof}

We arrive at the following result, which was already proved by Fontaine--Messing \cite[Cor. 2.7]{FontaineMessing}
and Kato \cite[II.Proposition 2.5]{Katovanishingcycles}.
In fact they did not need the existence of a lift all the way to $\Spf(W)$.

\begin{corollary}
\label{co-FM-revisited}
Let $\cX$ be a proper smooth $p$-adic formal scheme over $W$.
Let $j \leq p-1$ and $n \in \mathbb{N}$. 
Then the natural map $\rH ^j _{\cris} (\cX_n /W_n, \cI^{[i]}_{\cris}) \to \rH^j _{\cris}(\cX_n/W_n)$
is injective,
and the triple
\[
\left (\rH^j _{\cris}(\cX_n/W_n), \rH ^j _{\cris} (\cX_n /W_n, \cI^{[i]}_{\cris}), 
\varphi_i:  \rH ^j _{\cris} (\cX_n /W_n, \cI^{[i]}_{\cris}) \to \rH^j _{\cris}(\cX_n/W_n) \right )
\] is an object in $\FL_{W(k)}$. 
\end{corollary}

\begin{proof}
The injectivity follows from \Cref{integral HdRSS} (1).
The triple tensored up to $S$ is identified with 
\[
\left (\rH^j(\mathcal{X}, \dR^{\wedge}_{-/\s}/p^n), \rH^j(\mathcal{X},\Fil^j_{\rH}\dR^{\wedge}_{-/\s}/p^n), \varphi_j  \right ),
\]
by \Cref{MF and Breuil module}.
We have showed the map $\rH^j(\mathcal{X},\Fil^\ell_{\rH}\dR^{\wedge}_{-/\s}/p^n) \to \rH^j(\mathcal{X}, \dR^{\wedge}_{-/\s}/p^n)$
is injective, and the divided Frobenius $\varphi_j$ generates the image:
for $j \leq p-2$, this was the main result in our previous paper \cite[Theorem 7.22 and Corollary 7.25]{LL20};
and for $j = p-1$, use \Cref{Fil i injective} and \Cref{degree p-1 is Breuil}. 
Using the ``if'' part of \Cref{FL module equivalent to Breuil module}, we see that
$\left (\rH^j _{\cris}(\cX_n/W_n), \rH ^j _{\cris} (\cX_n /W_n, \cI^{[j]}_{\cris}), 
\varphi_i:  \rH ^j _{\cris} (\cX_n /W_n, \cI^{[i]}_{\cris}) \to \rH^j _{\cris}(\cX_n/W_n) \right )$
 is an object in $\FL_{W(k)}$. 
\end{proof}

\subsection{Comparison to \'etale cohomology}
In this section, we study how crystalline cohomology $\rH ^i _{\cris} (\mathcal X /S_n)$ compares to \'etale cohomology $\rH ^i_{\et} (\cX_{\C}, \Z/ p ^n \Z)$ 
in the boundary case $e \cdot i  = p-1 $. 
We shall freely use the notation and terminology from \Cref{modules and Galois rep}.

We first treat the case when $e=1$ and $p -1$,
in which case \Cref{co-FM-revisited} shows that 
\[
M : = \left (\rH^{p-1} _{\cris}(\cX_n/W_n), \rH ^{p-1} _{\cris} (\cX_n /W_n, \cI^{[p-1]}_{\cris}), \varphi_{p-1} \right )
\]
is an object in $\FL_{W(k)}$. 
\begin{theorem}
\label{Thm-FM-p-1}
Notations as the above, then  there exists a natural map 
$\eta \colon \rH ^{p-1} _{\et} (\cX_\C, \Z/ p ^n \Z)(p-1)  \to T_{\FL} (M )$ of $G_K$-representations such that 
\begin{enumerate}
    \item The $\ker(\eta)$ is an unramified representation of $G_K$ killed by $p$;
    \item The $\coker(\eta)$ sits in a natural exact sequence $ 0 \to W \to \coker (\eta)\to W' $,
    where $W \cong \ker(\eta)$ and $W' \cong \ker(\mathrm{Sp}_n^{p-1})$ is given by the kernel of specialization map in degree $(p-1)$.
\end{enumerate}
\end{theorem}

Note that by our \Cref{control ker Cosp} (3), $\ker(\mathrm{Sp}_n^{p-1})$ is also an unramified $G_K$-representation killed by $p$.
The $T_{\FL} (M) $ in the above theorem is what we meant by $\rho^{p-1}_{n, \mathrm{FL}}$ in \Cref{Thm 1.9}.

\begin{proof} 
Let $\m: = \rH^{p-1} _{\Prism} (\cX/ \s_n ) $ 
(note that here we do not have Frobenius twist) and 
$\calM = \rH^{p-1}_{\cris} (\cX/S_n)$. 
We have showed that the natural exact sequence \eqref{the sequence} 
induces a natural exact sequence in $\Mod_{S, \tor}^{\varphi, p-1 , \nabla}$:  
\[ \xymatrix{ 0 \ar[r] & \u \calM (\m) \ar[r] & \calM \ar[r] &  \bar M \ar[r]  & 0 },\]
see \Cref{residual filtration}, \Cref{Fil on tensor product}, \Cref{compute residual divided Frobenius}, \Cref{degree p-1 is Breuil},
\Cref{descend connection} and \Cref{residual connection} for descriptions of the filtrations, Frobenii action, and connections.
Furthermore, our \Cref{MF and Breuil module} says $\calM = \u\calM_{\FL} (M)$. 
Therefore, by left exactness of $T_S$, we have a natural sequence of $G_K$-representations:
\[
0 \to T_S (\u\calM(\m)) \hookrightarrow  T_S (\calM) = T_{\FL} (M) \to T_S(\overline{M}).
\]
On the other hand, we also have natural maps of $G_K$-representations:
$$\xymatrix{\eta: \rH^{p -1} _{\et} (\cX_\C, \Z/ p^n \Z)(p-1)\ar[r]^-{\sim} & T_\s(\m)(p-1) \ar[r]^-\sim_-\alpha & T^{p-1} _\s (\m ) \ar[r]^-\iota &  T_S (\u\calM(\m)).} $$ 
The first isomorphism is proved by \cite[Cor.~7.4,  Rem.~7.5]{LL20}. As explained before Lemma \ref{lem-Galois-unipotent}, the map $\iota \circ \alpha$ is a map compatible with $G_K$-actions if the natural map $f: \m \otimes _\s A_{\inf}\to \calM (\m) \otimes_S A_{\cris}$ is compatible with $G_K$-actions on the both sides, where the $G_K$-action on $\m \otimes _\s A_{\inf}$ given by $\m \otimes _\s A_{\inf}\simeq \rH^{p-1} _\Prism (\cX_{\O_\C}/ A_{\inf})$ and the $G_K$-action on $\u \calM(\m) \otimes_S A_{\cris}$ is defined by formula \eqref{eqn-G-action-by-nabla}. To prove that $f$ is compatible with $G_K$-actions, note that the natural map $f':  \rH^{p-1} _\s (\cX_{\O_\C}/A_{\inf}) \to \rH^{p-1}_{\cris} (\cX_{\O_\C}/ A_{\cris})$, which is compatible with $G_K$-actions,  factors through $f : \m \otimes _\s A_{\inf}\to \u \calM (\m) \otimes_S A_{\cris}$ by using inclusion $\u \calM (\m) \subset \calM $ and isomorphism $\beta: \calM \otimes_S A_{\cris}\simeq \rH^{p-1} _{\cris} (\cX_{\O_{\C}}/A_{\cris})$. So it suffices to check that 
$\u\calM (\m)\otimes_S A_{\cris}\to \calM \otimes_S A_{\cris} \simeq \rH^{p-1} _{\cris} (\cX_{\O_{\C}}/A_{\cris}) $ are compatible with $G_K$-actions. The compatibility of first map is due to that  $\u\calM(\m) \subset \calM$ is stable under $\nabla$ on $\calM$ by Corollary \ref{descend connection}, and the compatibility of second isomorphism is proved in \cite[\S 5.3]{LL20}. 
In summary, we obtain a natural map $\eta: \rH^{p -1} _{\et} (\cX_\C, \Z/ p^n \Z)(p-1)\to T_{\FL} (M)$ of $G_K$-representations. 

Now we shall justify the two extra statements concerning kernel and cokernel of $\eta.$
Since $T_S$ is left exact, $\ker (\eta) \simeq \ker (\iota)$ which is unramified
and killed by $p$, thanks to \Cref{cor-diff-killed-byp}. 

Easy diagram chase gives us a natural exact sequence:
\[
0 \to \coker(\iota) \to \coker(\eta) \to T_S (\overline M).
\]
By \Cref{cor-diff-killed-byp} we have $\coker(\iota) \cong \ker(\iota)$.
The fact that $T_S (\overline M) \cong \ker(\mathrm{Sp}_n^{p-1})$ follows from
\Cref{cor-rep-Mbar} and \Cref{ker Cosp and u-torsion}. 
\end{proof}
\begin{remark} 
\leavevmode
\begin{enumerate}
\item From the proof, we see that the appearance of $\ker(\eta)$ and $V$ is due to the defect of a key functor in integral $p$-adic Hodge theory,
and the potential $u$-torsion in degree $p$ (mod $p^n$) prismatic cohomology of $\cX$
is to be blamed for the appearance of $V'$.
\item It is unclear to us if the whole $\coker(\eta)$ is unramified and/or killed by $p$.
It could even very well be the case that the sequence $ 0 \to W \to \coker (\eta)\to W'$
is split exact (in particular, right-exact) as $G_K$-representations.
One would need extra input from integral $p$-adic Hodge theory, especially a further study of Breuil and Fontaine--Laffaille
modules in the boundary degree case, in order to obtain such refinements.
\end{enumerate}
\end{remark}

Now we discuss the case $e>1$ but $h \leq p-2$. 
We first recall that for $i \leq p -1$, in \cite[\S 5.2 ]{LL20} we have showed that $\calM^i_n : = (\rH ^i _{\cris} (\cX /S_n), \rH ^i _{\cris} (\cX/ S_n, \cI ^{[i]}), \varphi _i)$ is an object in $^\sim \Mod^{\varphi, i}_{S}$. By the discussion before equation (7.24) in \cite{LL20}, we get the following exact sequence for $i \leq h \leq p-2$: 
\begin{equation}\label{eqn-etale-Fil-2}
\cdots\rH ^{i -1} _{\cris} (\mathcal{X}_n/ A_{\cris, n}) \to  \rH^i_ {\et} (\mathcal{X}_{\C} , \Z/ p^n \Z (h)) \to \rH ^ i _{\cris} (\mathcal{X}_n/ A_{\cris, n}, \cI_{\cris} ^{[h]}) \overset{\varphi_h -1}{\longrightarrow} \rH ^i _{\cris} (\mathcal{X}_n / A_{\cris, n }),
\end{equation}
let us mention that the crucial input is \cite[Theorem F]{AMMN21}.
Thanks to $A_{\cris, n}$ being flat over $S_n$, we have
\[\rH^i_{\cris} (\cX_n /A_{\cris, n} , \cI ^{[h]}_{\cris}) \cong \rH^i _{\cris} (\cX_n /S_{ n} , \cI ^{[h]}_{\cris})\otimes_S A_{\cris}
\text{ and } 
\rH^i_{\cris} (\cX_n /A_{\cris, n}) \cong \rH^i _{\cris} (\cX_n /S_{ n})  \otimes_S A_{\cris},\]
In this case, we can still define \[T_S (\calM_n ^i): = \Fil^i (\calM_n^i\otimes  _S A_{\cris})^{\varphi_i =1} =\ker\{\varphi_i -1: \rH^i_{\cris} (\cX_n /A_{\cris, n} , \cI ^{[i]}_{\cris})\to \rH ^i _{\cris} (\mathcal{X}_n / A_{\cris, n })   \}.  \] 
The only difference is that the natural map $\rH^i _{\cris} (\cX_n /S_{ n} , \cI ^{[i]}_{\cris} ) \to \rH^i _{\cris} (\cX_n /S_{ n} )$ is not expected to be injective without the condition $e \cdot i < p-1$. 

\begin{proposition}\label{prop-TS}
Notation as above, we have a functorial isomorphism
$T_S (\calM_n^i ) \cong \rH^i_ {\et} (\mathcal{X}_{\C} , \Z/ p^n \Z (i)) $. 
\end{proposition}
\begin{proof} By \eqref{eqn-etale-Fil-2}, it suffices to show that $\varphi_i-1:  \rH ^ i _{\cris} (\mathcal{X}_n/ A_{\cris, n}, \cI_{\cris} ^{[i]}) {\longrightarrow} \rH ^i _{\cris} (\mathcal{X}_n / A_{\cris, n })$ is surjective for $i < p-2$. Choose an $m$ large enough so that $\varphi _i (\Fil^m S_n)= 0 $. 
So clearly $\varphi_i -1 $ restricted to the image of $\Fil^m S \otimes \rH^i_{\cris} (\cX_n / A_{\cris , n})$ is bijective.
\footnote{From now on, we abusively denote this image by $\Fil^m S\cdot \rH^i_{\cris} (\cX_n / A_{\cris , n})$.}
Hence it suffices to show that $$\varphi_{i}-1 : \rH ^ i _{\cris} (\mathcal{X}_n/ A_{\cris, n}, \cI_{\cris} ^{[i]})/ \Fil ^m  S\cdot \rH^i_{\cris} (\cX_n / A_{\cris , n}) \longrightarrow \rH^i_{\cris} (\cX_n / A_{\cris , n})/ \Fil^m S\cdot \rH^i_{\cris} (\cX_n / A_{\cris , n})$$ is surjective. Now we claim that both sides are finite generated $W_n(\O_{\C}^\flat)$-modules. Then the surjectivity of $\varphi_i-1$ follows Lemma \ref{cokernel of F-G} below. 

To check that both $\rH ^ i _{\cris} (\mathcal{X}_n/ A_{\cris, n}, \cI_{\cris} ^{[i]})/ \Fil ^m S\cdot \rH^i_{\cris} (\cX_n / A_{\cris , n})  $  and  $ \rH^i_{\cris} (\cX_n / A_{\cris , n})/ \Fil^m S\cdot \rH^i_{\cris} (\cX_n / A_{\cris , n})$ are finitely generated over $W_n (\O_{\C}^\flat)$, 
it suffices to check that $\rH ^ i _{\cris} (\mathcal{X}_n/ S_n, \cI_{\cris} ^{[i]})/ \Fil ^m  S\cdot \rH^i_{\cris} (\cX_n / S_n) $ and $\rH^i_{\cris} (\cX_n / S_{ n})/ \Fil^m S\cdot \rH^i_{\cris} (\cX_n / S_{ n})$ are finite generated $\s_n$-modules. 
This is clear for  $\rH^i_{\cris} (\cX_n / S_{ n})/ \Fil^m S\cdot \rH^i_{\cris} (\cX_n / S_{ n})$: 
it is known that  $\rH^i_{\cris} (\cX_n / S_{ n})$ is a finite generated $S_n$-module, see \cite[Proposition 7.19]{LL20}. 
For $\rH ^ i _{\cris} (\mathcal{X}_n/ S_n, \cI_{\cris} ^{[i]})/ \Fil ^m S\cdot \rH^i_{\cris} (\cX_n / S_n) $, consider the following diagram 
\[  \xymatrix{ \rH^{i-1}_{\qsyn} (\cX_n, \Prism^{(1)}/ \Fil^i_{\rm N} \Prism^{(1)} )\ar[r]^-{\alpha}\ar[d]^\wr & \rH ^i_{\qsyn} (\cX_n, \Fil^i_{\rm N} \Prism^{(1)}) \ar[r]^-{\beta}\ar[d]& \rH ^i _{\qsyn}(\cX_n, \Prism ^{(1)}_{-/\s} )\ar[d] ^\iota \ar[r]& \cdots \\  \rH^{i-1}_{\qsyn} (\cX_n , \dR^\wedge_{-/\s}/ \Fil^i_{\rm H} \dR ^\wedge _{-/ \s} )\ar[r]^-{\alpha'} & \rH ^i_{\qsyn} (\cX_n, \Fil^i_{\rm H} \dR^\wedge _{-/ \s})\ar[r] ^-{\beta'} & \rH ^i_{\qsyn} (\cX_n, \dR^\wedge_{-/\s} )\ar[r] &\cdots }\]
Since $\rH ^i _ \qsyn (\cX_n , \dR^\wedge_{\cX_n/\s_n} )$ is finitely generated over $S_n$, the image of $\rH ^ i _{\cris} (\mathcal{X}_n/ S_n, \cI_{\cris} ^{[i]})/ \Fil ^m S\cdot \rH^i_{\cris} (\cX_n / S_n) $ inside  $\rH^i_{\cris} (\cX_n / S_{ n})/ \Fil^m S\cdot \rH^i_{\cris} (\cX_n / S_{ n})$ is also finite $\s_n$-generated.
Here we have used the fact that $S_n/\Fil^m S_n$ is finitely generated over $\s_n$.
Note  that $\ker (\beta')= {\rm Im} (\alpha ')$ is also finitely generated over $\s_n$. 
So  $\rH^i_{\cris} (\cX_n / S_{ n}, \cI_{\cris}^{[i]})/ \Fil^m S\cdot \rH^i_{\cris} (\cX_n / S_{ n})$ is finite $\s_n$-generated. 
\end{proof}

\begin{lemma}
\label{cokernel of F-G in char p}
Let $C^\flat$ be a characteristic $p$ algebraically closed complete non-Archimedean field, denote its ring of integers by $\mathcal{O}_C^\flat$ with maximal ideal $m^\flat$ and residue field $k^\flat$.
Let $M$ and $N$ be two finitely generated $\mathcal{O}_C^\flat$ modules, let $F \colon M \to N$ be a Frobenius semilinear map,
and let $G \colon M \to N$ be a linear map.
The following are equivalent:
\begin{enumerate}
    \item The map $F-G \colon M \to N$ is surjective;
    \item The cokernel of $F-G \colon M \to N$ is finite;
    \item The cokernel of $\overline{F-G} \colon M/m^\flat M \to N/m^\flat N$ is surjective;    
    \item The induced map $\overline{F-G} \colon M/m^\flat M \to N/m^\flat N$ is finite.
\end{enumerate}
\end{lemma}

\begin{proof}
It is clear that $(1) \implies (2) , (3) \implies (4)$. Below we shall show $(4) \implies (1)$.

Without loss of generality we may assume that both of $M$ and $N$ are finite free over $\mathcal{O}_C^\flat$.
Indeed, let us choose maps from finite free modules, say $P$ and $Q$, to $M$ and $N$ such that 
it is an isomorphism after modulo $m^\flat$.
By Nakayama's lemma we see that these maps are surjective.
Lift the two maps $F$ and $G$, to get the following diagram
\[
\xymatrix{
P \ar[r]^-{\widetilde{F}-\widetilde{G}} \ar[d] & Q \ar[d] \\
M \ar[r]^-{F-G} & N.
}
\]
By our choice of $P$ and $Q$, condition (4) still holds for the top arrow.
Since vertical arrows are surjective, it suffices to show that the top arrow is surjective.
Therefore we may and do assume $M$ and $N$ are finite free.

Let us name the reduction of $M$ and $N$ by $V$ and $W$ which are finite dimensional $k^\flat$-vector spaces,
and denote the reduction of $F$ and $G$ by $f$ and $g$.
We claim there are exhaustive increasing filtrations $\Fil_i$ with $0 \leq i \leq \ell$ on $V$ and $W$ respectively such that
\begin{itemize}
    \item The maps $f$ and $g$ respect these two filtrations;
    \item The induced $f \colon \Fil_0 V \to \Fil_0 W$ is surjective;
    \item The induced $f \colon \gr_i V \to \gr_i W$ is $0$ for all $1 \leq i \leq \ell$; and
    \item The induced $g \colon \gr_i V \to \gr_i W$ is an isomorphism for all $1 \leq i \leq \ell$.
\end{itemize}
To see the existence of such filtrations, we consider the following process:
notice the image of $f \colon V \to W$ is a $k^\flat$ subspace, now look at the map $g \colon V \to W/\mathrm{Im}(f)$.
By assumption of $\mathrm{Coker}(f-g)$ being finite, this map must be surjective, lastly we let 
\[
\Fil^0 V = \mathrm{Ker}\left(g \colon V \to W/\mathrm{Im}(f)\right),~\Fil^0 W = \mathrm{Im}(f).
\]
Replace $V$ and $W$ with $\Fil^0 V$ and $\Fil^0 W$ and repeat the above steps. 
This process terminates when we arrive at $\mathrm{Im}(f) = W$, and
it will terminate as each time the dimension of $W$ will drop.
This way we get a decreasing filtration, after reversing indexing order we arrive at the desired increasing filtration.

Choose a sub-vector space $V_0 \subset \Fil_0(M/m^\flat M)$ on which $f$ is an isomorphism, now
lift the basis of $\gr_i (M/m^\flat M)$ for $1 \leq i \leq \ell$ and the basis of $V_0$ all the way to elements in $M$,
we generate a finite free submodule $\widetilde{M}$.
Now we contemplate the map $\widetilde{M} \to N$.

After choosing basis, we may regard both sides as $\mathcal{O}_C^{\flat}$ points of formal affine space over $\mathcal{O}_C^{\flat}$,
and the map $F-G$ can be promoted to an algebraic map 
$h \colon \Spf(\mathcal{O}_C^{\flat}\langle \underline{X} \rangle) \to \Spf(\mathcal{O}_C^{\flat}\langle \underline{Y} \rangle)$.
Note that by our choice of $\widetilde{M}$, these two formal affine spaces have the same dimension.
Our choice of $\widetilde{M}$ guarantees that the reduction of $h$ is finite,
due to next lemma.
Therefore the rigid generic fiber map $h^{\mathrm{rig}}$ is also finite by \cite[6.3.5 Theorem 1]{BGR},
which implies it is flat by miracle flatness \cite[\href{https://stacks.math.columbia.edu/tag/00R4}{Tag 00R4}]{stacks-project},
hence inducing a surjective map
at the level of $C^\flat$-points.\footnote{Note that $C^\flat$-points of rigid generic fibre of an admissible
formal scheme over $\mathcal{O}_C^{\flat}$ is the same as just $\mathcal{O}_C^{\flat}$-points of the 
formal scheme, see \cite[\S 8.3]{Bosch}.}
%
%
%
%
%
%
%
%
%
\end{proof}

The following lemma was used in the proof above, we thank Johan de Jong
for providing an elegant proof.

\begin{lemma}
Let $k$ be a field, let $m>1$ be an integer, and let $\big(a_{ij}\big)$ be an $n \times n$ matrix
with entries in $k$.
Let $\bar{h} \colon \mathbb{A}^n_k \to \mathbb{A}^n_k$ be the morphism given by
$\bar{h}^{\sharp}(y_i) = x_i^m + \sum_j a_{i,j} x_j$,
then $\bar{h}$ is a finite morphism.
\end{lemma}

\begin{proof}
This map can be compatified to a morphism between $\mathbb{P}^n_k$ preserving the
infinity hyperplane.
When restricted to the infinity hyperplane, the map becomes
$[x_1 \colon \ldots \colon x_n] \mapsto [x_1^m \colon \ldots \colon x_n^m]$,
which is non-constant.
Lastly just observe that any endomorphism of $\mathbb{P}^n_k$ is either finite or constant.
\end{proof}

Here we have crucially used algebraically closedness of $\mathcal{O}_C^\flat$.
Below is an example suggested to us by Johan de Jong illustrating the failure of $(3) \implies (4)$ when one drops the
algebraically closed assumption.
Start with the field $L_0 = \mathbb{F}_p(t^{1/p^\infty})$, pick a basis of $\rH^1_{\et}(L_0, \mathbb{F}_p)$
we may find a (ginormous!) Galois pro-$p$ infinite field extension $L_1$ such that the induced
map on $\rH^1_{\et}(-, \mathbb{F}_p)$ kills every basis vector except the first one.
Repeat this process we arrive at a perfect field $L$ such that $\rH^1_{\et}(L, \mathbb{F}_p)$ is $1$-dimensional
over $\mathbb{F}_p$.

From the above we immediately conclude the following.

\begin{lemma}
\label{cokernel of F-G}
Let $M$ and $N$ be two finitely generated $A_{\inf}$ modules, let $F \colon M \to N$ be a Frobenius linear map
and $G \colon M \to N$ be a linear map.
Then the cokernel of $F-G$ (which is a $\mathbb{Z}_p$-linear map) is finitely generated over $\mathbb{Z}_p$ if and only
if it is $0$.
\end{lemma}

\begin{proof}
The ``if'' part is trivial. For the ``only if part'': use right exactness of tensor and \Cref{cokernel of F-G in char p} we conclude that the 
cokernel is zero after modulo $p$. 
Now since finitely generated $\mathbb{Z}_p$ module is $0$
if and only if its reduction modulo $p$ is so, we get that the cokernel is zero.
\end{proof}

\section{An example}
\label{subsection example}

Inspired by the example in \cite[Subsection 2.1]{BMS1}, let us work out a direct generalization of their example
(as suggested in \cite[Remark 1.3]{BMS1})
in this subsection.
This example answers a question of Breuil \cite[Question 4.1]{BreuilIntegral} negatively.

Fix a positive integer $n$.\footnote{We suggest first-time readers to simply take $n=1$ which already has
the ``meat'' and the notations and formulas become much simpler.}
Let $\mathcal{E}_0$ be an ordinary elliptic curve over an algebraically closed field $k$
of characteristic $p > 0$.
Let $\mathcal{E}$ over $\Spec(W(k))$ be its canonical lift, in particular we have a closed immersion
$\mu_{p^n} \subset \mathcal{E}[p^n]$ of finite flat group schemes over $\Spec(W(k))$.
Let $\mathcal{O}_K \coloneqq W(k)[\zeta_{p^n}]$, choose $\zeta_{p^n} - 1$ to be the uniformizer
in order to get $\s \twoheadrightarrow \mathcal{O}_K$. To avoid confusion let us denote its Eisenstein
polynomial as
\[
E= d = \frac{(u+1)^{p^n} - 1}{(u+1)^{p^{n-1}} - 1} \in \s.
\]
On $\Spec(\mathcal{O}_K)$
we have the canonical group scheme homomorphism $\mathbb{Z}/p^n \to \mu_{p^n}$.

\begin{construction}
Let $\mathcal{X} \coloneqq [\mathcal{E}/(\mathbb{Z}/p^n)]$, a Deligne--Mumford stack
which is smooth proper over $\Spec(\mathcal{O}_K)$.
Here the action of $\mathbb{Z}/p^n$ on $\mathcal{E}$ is via $\mu_{p^n}$.
The generic fibre of $\mathcal{X}$ is the elliptic curve $\mathcal{E}_K/\mu_{p^n}$ (which in fact is isomorphic to $\mathcal{E}_K$ itself)
and the special fibre is $\mathcal{E}_0 \times B(\mathbb{Z}/p^n)$.
We have a factorization
\[
\mathcal{E} \to \mathcal{X} \to \mathcal{E}/\mu_{p^n} \cong \mathcal{E}
\]
of the lift of $n$-th Frobenius on $\mathcal{E}$ (exists because it is the canonical lift).
\end{construction}

We want to understand the various cohomology theories of $\mathcal{X}$. 
Since all cohomology theories that we will encounter are \'{e}tale sheaves and the quotient map 
$\mathcal{E} \to \mathcal{X}$, being a $\mathbb{Z}/p^n$-torsor,
is a finite \'{e}tale cover, we shall apply the Leray spectral sequence to this cover.
Let us first record the structure of the prismatic cohomology of $\mathcal{E}_{\mathcal{O}_K}$ relative to $\s$.
We need the following lemma explicating the Frobenius operator on the (-1) Breuil--Kisin
twist  $\s\{-1\}$, see \cite[Section 2.2]{APC}.
\begin{lemma}
The Frobenius module $\s\{-1\}$ has a generator $x$ such that $\varphi(x) = E(u) \cdot \frac{p}{E(0)} \cdot x$.
\end{lemma}

\begin{proof}
We know modulo $u$ the Breuil--Kisin prism $\s$ reduces to crystalline prism, whose (-1)-twist
has a canonical generator $\bar{x}$ satisfying $\varphi(\bar{x}) = p \cdot \bar{x}$.
Lifting this generator, we see that there is a generator $x'$ of $\s\{-1\}$ such that
$\varphi(x') = a \cdot x'$ with $a \equiv p$ mod $u$.
On the other hand we know $a$ is necessarily $E(u) \cdot \text{unit}$, due to \cite[Construction 2.2.14]{APC}.
Therefore we see that $a = E(u) \cdot \frac{p}{E(0)} \cdot v'$ where $v' \in \s^{\times}$ and reduces to $1$ mod $u$.
It is a simple exercise that $v'$ is of the form $\varphi(v)/v$ for some unit $v \in \s^{\times}$ satisfying $v \equiv 1$ mod $u$ as well.
Finally $x = x'/v$ is our desired generator.
\end{proof}

In our concrete situation, the Eisenstein polynomial $d$ of $\zeta_{p^n} - 1$ has constant term $p$.
Therefore our $\s\{-1\}$ has a generator $x$ such that $\varphi(x) = d \cdot x$.

\begin{proposition}
\label{prop pris of ellcurve}
We have isomorphism of Frobenius modules over $\s$:
\begin{enumerate}
    \item $\rH^0_{\Prism}(\mathcal{E}_{\mathcal{O}_K}/\s) \cong \s$;
    \item $\rH^2_{\Prism}(\mathcal{E}_{\mathcal{O}_K}/\s) \cong \s\{-1\}$; and
    \item $\rH^1_{\Prism}(\mathcal{E}_{\mathcal{O}_K}/\s) \simeq \s \cdot \{e_1, e_2\}$
    with its Frobenius action given by $\varphi(e_1) = e_1, \text{ and } \varphi(e_2) = a \cdot e_1 + d \cdot e_2$
    for some $a \in \s$.
\end{enumerate}
\end{proposition}

\begin{proof}
It is well-known that elliptic curve has torsion-free crystalline cohomology.
Therefore by \Cref{finite free follows from crystalline torsion free},
we know all these prismatic cohomology groups are finite free $\s$-modules.

The map $\mathcal{X} \to \Spf(\s/d)$ always induces an isomorphism on $\rH^0_{\Prism}$ by Hodge--Tate comparison,
this proves the first identification.

The second identification is well-known.
For instance, the relative prismatic Chern class \cite[Section 7.5]{APC} of 
(the line bundle associated with) the origin $0 \in \mathcal{E}_{\mathcal{O}_K}({\mathcal{O}_K})$
gives a map $c \colon \s\{-1\} \to \rH^2_{\Prism}(\mathcal{E}_{\mathcal{O}_K}/\s)$.
Reducing mod $u$ this reduces to the first Chern class map in crystalline cohomology which is well-known
to be an isomorphism. Since both source and target are finite free $\s$-module, the map $c$ is an isomorphism.

Cup product gives rise to a map of finite free Frobenius $\s$-modules: 
$\bigwedge^2_{\s} \rH^1_{\Prism}(\mathcal{E}_{\mathcal{O}_K}/\s) \to \rH^2_{\Prism}(\mathcal{E}_{\mathcal{O}_K}/\s)$.
Modulo $u$ this map reduces to the analogous map in crystalline cohomology
which is again well-known to be an isomorphism, hence it is an isomorphism before mod $u$.
Therefore it suffices to justify the existence of $e_1$. 
Since $\varphi(u) = u^p$, we see that 
\[
\left(\rH^1_{\Prism}(\mathcal{E}_{\mathcal{O}_K}/\s)\right)^{\varphi = 1} \cong 
\left(\rH^1_{\Prism}(\mathcal{E}_{\mathcal{O}_K}/\s)/u\right)^{\varphi = 1}.
\]
Now we may use the crystalline comparison $\rH^1_{\Prism}(\mathcal{E}_{\mathcal{O}_K}/\s)/u \cong \rH^1_{\cris}(\mathcal{E}_0/W)^{(-1)}$
and the fact that $\mathcal{E}_0$ is ordinary to conclude the existence of $e_1$. 
\end{proof}

Next let us compute the prismatic cohomology $\rH^{*}_{\Prism}(\mathcal{X}/\s)$.
We stare at the Leray spectral sequence
\[
E_2^{i,j} = \rH^i(\mathbb{Z}/p^n, \rH^j_{\Prism}(\mathcal{E}_{\mathcal{O}_K}/\s)) \Longrightarrow \rH^{i+j}_{\Prism}(\mathcal{X}/\s)
\]
which is compatible with Frobenius actions. 
In order to understand $E_2$ terms, we need the following:
\begin{lemma}
The action of $\mathbb{Z}/p^n$ on $\rH^j_{\Prism}(\mathcal{E}_{\mathcal{O}_K}/\s)$ is trivial.
\end{lemma}

\begin{proof}
Let us use the $p$-completely flat base change $\s \hookrightarrow W(C^{\flat})$.
Since our prismatic cohomology, as $\s$-modules, are free,
we get injections 
$\rH^j_{\Prism}(\mathcal{E}_{\mathcal{O}_K}/\s) \hookrightarrow \rH^j_{\Prism}(\mathcal{E}_{\mathcal{O}_K}/\s) \otimes_{\s} W(C^{\flat})$
compatible with the $\mathbb{Z}/p^n$-action.
Using the \'{e}tale comparison~\cite[Theorem 1.8.(iv)]{BMS1}, the target is canonically identified with
$\rH^j_{\et}(\mathcal{E}_C, \mathbb{Z}_p) \otimes_{\mathbb{Z}_p} W(C^{\flat})$.
We conclude the $\mathbb{Z}/p^n$-action on the target is trivial by comparing to the topological situation.
\end{proof}

Therefore the second page, which is the starting page, of the above spectral sequence looks like
\[
\label{ex spectral sequence}
\tag{\epsdice{4}}
\xymatrix{
\s\{-1\} \ar[rrd] & 0 & \s\{-1\}/p^n & \ldots \\
\rH^1_{\Prism}(\mathcal{E}_{\mathcal{O}_K}/\s) \ar[rrd]^{d_2} & 0 & \rH^1_{\Prism}(\mathcal{E}_{\mathcal{O}_K}/\s)/p^n & \ldots \\
\s & 0 & \s/p^n & \ldots \\
}
\]
To our interest is the differential
\[
d_2 \colon E_2^{0,1} = \rH^1_{\Prism}(\mathcal{E}_{\mathcal{O}_K}/\s) \longrightarrow E_2^{2,0} \cong \s/p^n.
\]
Using the multiplicative structure of the spectral sequence this arrow determines the rest arrows,
by degree reason the spectral sequence degenerates on the third page $E_3^{i,j} = E_{\infty}^{i,j}$.

\begin{lemma}
\label{lem divisible by u}
The differential $d_2$ is divisible by $u$.
In other words, it is zero after reduction modulo $u$.
\end{lemma}

\begin{proof}
Let us look at the reduction modulo $u$ of the spectral sequence \ref{ex spectral sequence},
which is computing the crystalline cohomology of $\mathcal{X}/W$ by the crystalline comparison.
Using the fact that $\rH^2_{\cris}(B(\mathbb{Z}/p^n)/W) \cong W/p^n$ (see for instance \cite[Theorem 1.2]{Mon20}),
we see that the $d_2$ modulo $u$ must be zero.
\end{proof}

\begin{lemma}
\label{lem d2 of e1}
We have $d_2(e_1) = 0$.
\end{lemma}

\begin{proof}
This is because $d_2$ is Frobenius-equivariant.
Now \Cref{prop pris of ellcurve} (3) implies $e_1 \in \rH^1_{\Prism}(\mathcal{E}_{\mathcal{O}_K}/\s)$ is fixed by Frobenius, 
yet \Cref{lem divisible by u} says its image under $d_2$ is divisible by $u$.
So its image is divisible by arbitrary powers of $u$, hence must be zero.
\end{proof}

\begin{lemma}
\label{lem d2 of e2}
After scaling $e_2$ by a unit in $\mathbb{Z}_p^{\times}$ we have $d_2(e_2) = (u+1)^{p^{n-1}} - 1$.
\end{lemma}

\begin{proof}
Note that $\varphi(e_2) = d \cdot e_2$ by \Cref{prop pris of ellcurve} (3),
its image must be an element $x \in \s/p^n$ satisfying the same Frobenius eigen-class condition.
Next lemma guarantees that $d_2(e_2) = b \cdot ((u+1)^{p^{n-1}} - 1)$
for some $b \in \mathbb{Z}/p^n$.
\'{E}tale comparison for prismatic cohomology says that
base changing the spectral sequence \ref{ex spectral sequence} along $\s \hookrightarrow W(C^{\flat})$
gives a spectral sequence computing \'{e}tale cohomology of $\mathcal{X}_C$ 
(base changed along $\mathbb{Z}_p \hookrightarrow W(C^{\flat})$).
Since $\mathcal{X}_C$ is an elliptic curve, its second \'{e}tale cohomology has no torsion,
hence the base changed $d_2$ must be surjective.
In particular we see that $b \not\in p \cdot \mathbb{Z}/p^n$,
hence $b$ must be a unit in $(\mathbb{Z}/p^n)^{\times}$.
\end{proof}

In the proof above, we have utilized the following:

\begin{lemma}
For any $m \leq n$, we have an exact sequence
\[
0 \to \mathbb{Z}/p^m \cdot \left((u+1)^{p^{n-1}} - 1\right) \to \s/p^m \xrightarrow{\varphi - d}
\s/p^m.
\]
\end{lemma}

\begin{proof}
First of all, let us check that $(u+1)^{p^{n-1}} - 1$ does satisfy the Frobenius action condition.
Recall $d = \frac{(u+1)^{p^n} - 1}{(u+1)^{p^{n-1}} - 1}$,
it suffices to know that $(u+1)^{p^n} \equiv (u^p+1)^{p^{n-1}}$ modulo $p^n$.
When $n = 1$ this is well-known, induction on $n$ proves the statement.

Next we verify this exact sequence for $m = 1$. In that situation $\s/p \cong k[\![u]\!]$,
and the Frobenius condition becomes $f^p = u^{p^{n-1}(p-1)} \cdot f$.
One immediately verifies that $f \in \mathbb{F}_p \cdot u^{p^{n-1}}$.

Lastly we finish the proof by induction on $m$ and applying the snake lemma to the following diagram:
\[
\xymatrix{
0 \ar[r] & \s/p^m \ar[r]^{\cdot p} \ar[d]^{\varphi - d} & \s/p^{m+1} \ar[r] \ar[d]^{\varphi - d} & \s/p \ar[r] \ar[d]^{\varphi - d} & 0 \\
0 \ar[r] & \s/p^m \ar[r]^{\cdot p} & \s/p^{m+1} \ar[r] & \s/p \ar[r] & 0.
}
\]
Notice that we have verified the kernel of middle vertical arrow surjects onto the kernel of right vertical arrow,
thanks to the previous two paragraphs.
The snake lemma tells us that the kernel of middle vertical arrow has length $m+1$,
but we also know $\mathbb{Z}/p^{m+1} \cdot \left((u+1)^{p^{n-1}} - 1\right)$ sits inside the kernel.
\end{proof}

From now on let us scale $e_2$ by the $p$-adic unit so that $d_2(e_2) = (u+1)^{p^{n-1}} - 1$.
Using multiplicativity of the spectral sequence \ref{ex spectral sequence}, we can compute the prismatic cohomology of $\mathcal{X}$.
Let us record the result below.

\begin{corollary}
The prismatic cohomology ring of $\mathcal{X}/\s$ is
\[
\rH^*_{\Prism}(\mathcal{X}/\s) \cong \s\langle e,f \rangle [g]/((u+1)^{p^{n-1}} - 1 \cdot g, p^n \cdot g, f \cdot g),
\]
where $e,f$ have degree $1$ and are pulled back to $e_1$ and $p^n \cdot e_2$ respectively inside 
$\rH^1_{\Prism}(\mathcal{E}_{\mathcal{O}_K}/\s)$,
and $g$ has degree $2$ being the generator of $E_3^{2,0} = E_{\infty}^{2,0}$.
Moreover the Frobenius action is given by 
\[
\varphi(e) = e,~\varphi(f) = p^n a' \cdot e +  d \cdot f, \text{ and } \varphi(g) = g.
\]

In particular we see that
\[
\rH^2_{\Prism}(\mathcal{X}/\s)[u^{\infty}] \cong \s/((u+1)^{p^{n-1}} - 1, p^n) \cdot g
\]
and 
\[
\rH^{\ell}_{\Prism}(\mathcal{X}/\s) \simeq \s/((u+1)^{p^{n-1}} - 1, p^n)
\]
for all $\ell \geq 3$ generated by either $g^{\ell/2}$ or $e \cdot g^{(\ell - 1)/2}$ depending on the parity of $\ell$.
\end{corollary}

Here $a'$ is a $p$-adic unit (that we used to scale $e_2$) times the constant $a$ from \Cref{prop pris of ellcurve} (3).
We remark that $g$ can be taken as a generator of the group cohomology $\rH^2(\mathbb{Z}/p^n, \mathbb{Z}_p)$.

Later on we will produce a schematic example using approximations of $B(\mathbb{Z}/p^n)$,
but before that let us observe that our stacky example matches with some predictions
made in 

\begin{remark}
\label{rmk stacky example}
The discussion in \Cref{Albanese subsection} extends to smooth proper Deligne--Mumford stacks,
such as our $\mathcal{X}$.
Since the generic fibre of $\mathcal{X}$ is $\big(\mathcal{E}/\mu_{p^n}\big)_K$,
the map $g \colon \mathcal{X} \to \text{ N\'{e}ron model of } \mathcal{X}_K$ becomes
the natural map $[\mathcal{E}/(\mathbb{Z}/p^n)] \to \mathcal{E}/\mu_{p^n}$.
Taking special fibre and factoring through $\mathrm{Alb}(\mathcal{X}_0) = \mathcal{E}_0$,
we see the map $f$ becomes the natural quotient map $\mathcal{E}_0 \to \mathcal{E}_0/\mu_{p^n}$
which has kernel $\mu_{p^n}$. 
Note that when $n = 1$, we have $e = p-1$,
and our \Cref{defect controlled by ramification index} (3) indeed predicts that $\ker(f)$ can be at worst a form of several copies of $\mu_p$.

Since $\mathcal{X}_0 \cong \mathcal{E}_0 \times B(\mathbb{Z}/p^n)$, we know its $\pi_1$ is abelian
with torsion given by $\mathbb{Z}/p^n$.
Consequently the torsion part in $\rH^2_{\et}(\mathcal{X}_0, \mathbb{Z}_p)$ is also given by $\mathbb{Z}/p^n$.
Since $\mathcal{X}_C$ is an elliptic curve, its \'{e}tale cohomology is torsion-free.
Hence the specialization map in degree $2$ for $p$-adic \'{e}tale cohomology has kernel given by $\mathbb{Z}/p^n$.
This matches up with what \Cref{ker Cosp and u-torsion} predicts.
Indeed since $\varphi(g) = g$, we see that 
\[
\left(\rH^2_{\Prism}(\mathcal{X}/\s)[u^{\infty}]\right)^{\varphi = 1}
= \left(\rH^2_{\Prism}(\mathcal{X}/\s)[u^{\infty}]/u\right)^{\varphi = 1}
= \left(W/p^n \cdot g\right)^{\varphi = 1} = \mathbb{Z}/p^n \cdot g.
\]
Here in the second identification we have used the fact that $u$ divides $(u+1)^{p^{n-1}} - 1$.
\end{remark}


The above stacky example can be turned into a scheme example, by the procedure of
approximation explained below.

\begin{construction}
\label{cons approximating stack}
Choose a representation $V$ of $\mathbb{Z}/p^n$ over $\mathbb{Z}_p$,
so that inside $\mathbb{P}(V)$ one can find a $\mathbb{Z}/p^n$-stable
complete intersection $3$-fold $\mathcal{Y}$ with no fixed point
and smooth proper over $\mathbb{Z}_p$, see \cite[2.7-2.9]{BMS1}.
Now we form $\mathcal{Z} \coloneqq (\mathcal{E} \times_{\mathbb{Z}_p} \mathcal{Y})_{\mathcal{O}_K}/(\mathbb{Z}/p^n)$,
which is a smooth proper relative $4$-fold over $\mathcal{O}_K$.
Here the action of $\mathbb{Z}/p^n$ is the diagonal action.
\end{construction}

Let us show the prismatic cohomology of $\mathcal{Z}/\s$ approximates that of $\mathcal{X}/\s$ in degrees $\leq 2$
in a suitable sense.

\begin{proposition}
\label{prop calculate pris of scheme example}
The natural $\mathbb{Z}/p^n$-equivariant projection
$(\mathcal{E} \times_{\mathbb{Z}_p} \mathcal{Y})_{\mathcal{O}_K} \to \mathcal{E}_{\mathcal{O}_K}$
gives rise a map $\mathcal{Z} \to \mathcal{X}$, which induces isomorphisms:
\[
\rH^0_{\Prism}(\mathcal{X}/\s) \xrightarrow{\cong} \rH^0_{\Prism}(\mathcal{Z}/\s) \text{ and }
\rH^1_{\Prism}(\mathcal{X}/\s) \xrightarrow{\cong} \rH^1_{\Prism}(\mathcal{Z}/\s).
\]
Together with the similarly defined map $\mathcal{Z} \to \mathcal{Y}_{\mathcal{O}_K}/(\mathbb{Z}/p^n)$,
we have
\[
\rH^2_{\Prism}(\mathcal{X}/\s) \oplus \s\{-1\} \xrightarrow{\cong} \rH^2_{\Prism}(\mathcal{Z}/\s).
\]
\end{proposition}

\begin{proof}
We want to apply the Leray spectral sequence to the finite \'{e}tale cover
$(\mathcal{E} \times_{\mathbb{Z}_p} \mathcal{Y})_{\mathcal{O}_K} \to \mathcal{Z}$.

First we claim the natural embedding $\mathcal{Y}_{\mathcal{O}_K} \to \mathbb{P}(V)_{\mathcal{O}_K}$
induces an isomorphism of prismatic cohomology in degrees $\leq 2$.
It suffices to show the same for Hodge--Tate cohomology, which in turn reduces us to showing
it for Hodge cohomology.
This follows from $\mathcal{Y}$ being a smooth complete intersection inside $\mathbb{P}(V)$,
see~\cite[Proposition 5.3]{ABM19}.
Lastly it is well-known that $\rH^2_{\Prism}(\mathbb{P}(V)_{\mathcal{O}_K}/\s) \cong \s\{-1\}$,
for instance see \cite[10.1.6]{APC}.

Since $\rH^1_{\Prism}(\mathcal{Y}_{\mathcal{O}_K}/\s) = 0$, the Leray spectral sequence
in degrees $\leq 2$ is the direct sum of the spectral sequences for $\mathcal{X}$
and $\mathcal{Y}/(\mathbb{Z}/p^n)$ respectively.
This gives the statement for cohomological degrees $\leq 1$.
Looking at the shape of the Leray spectral sequence for $\mathcal{Y}/(\mathbb{Z}/p^n)$,
one easily sees that the $E_2^{0,2}$ term:
\[
\big(\rH^2_{\Prism}(\mathcal{Y}_{\mathcal{O}_K}/\s)\big)^{\mathbb{Z}/p^n} \cong \rH^2_{\Prism}(\mathcal{Y}_{\mathcal{O}_K}/\s) \cong \s\{-1\}
\]
survives, hence proving the  statement in cohomological degree $2$.
\end{proof}

\begin{remark}
\label{rmk example}
\leavevmode
\begin{enumerate}
\item Since $\rH^2_{\Prism}(\mathcal{X}/\s) \oplus \s\{-1\} \cong \rH^2_{\Prism}(\mathcal{Z}/\s)$
we know the $\rH^2_{\Prism}(\mathcal{Z}/\s)[u^\infty] \cong \s/((u+1)^{p^{n-1}} - 1, p^n)$.
In particular its annihilator ideal is $((u+1)^{p^{n-1}} - 1, p^n) \in \s$, congruent to $(u^{p^{n-1}})$ modulo $(p)$.
The ramification index is $p^{n-1}(p-1)$, hence these examples demonstrate that the bound in \Cref{cor exponent inequality}
is sharp.
\item Now assume $p \geq 3$, then $p-2+1 \geq 2$, our previous result \cite[Theorem 7.22]{LL20}
together with the fact that $\rH^2_{\Prism}(\mathcal{Z}/\s)$ contains $u$-torsion implies
that Breuil's first crystalline cohomology of $\mathcal{Z}$, with mod $p^m$ coefficient for any $m$,
together with Frobenius action and filtration is \emph{not} a Breuil module.
When $n = 1$, we have $e = p-1$, which shows that our result \cite[Corollary 7.25]{LL20} is sharp.
Below we shall see that the first crystalline cohomology cannot even support a strongly divisible lattice structure
because it is torsion-free but not free.
\item Same reasoning as in \Cref{rmk stacky example} shows that the map 
$f \colon \mathrm{Alb}(\mathcal{Z}_0) \to \mathrm{Alb}(\mathcal{Z})_0$
is given by the quotient map $\mathcal{E}_0 \to \mathcal{E}_0/\mu_{p^n}$.
\item The special fibre $\mathcal{Z}_0 = \mathcal{E}_0 \times (\mathcal{Y}_0/(\mathbb{Z}/p^n))$ has abelian $\pi_1$,
with its torsion part being $\mathbb{Z}/p^n$.
Here we used the fact that complete intersections of dimension $\geq 3$ are simply connected,
see \cite[\href{https://stacks.math.columbia.edu/tag/0ELE}{Tag 0ELE}]{stacks-project}.
On the other hand the same argument as in \cite[proof of Proposition 2.2.(i)]{BMS1} shows that
$\pi_1(\mathcal{Z}_C) \cong \widehat{\mathbb{Z}}^{\oplus 2}$.
Hence we see again the specialization map $\rH^2_{\et}(\mathcal{Z}_0) \to \rH^2_{\et}(\mathcal{Z}_C)$
has kernel given by $\mathbb{Z}/p^n$,
c.f.~\cite[Remarks 2.3-2.4]{BMS1}.
\end{enumerate}
\end{remark}

In fact it was the desire of finding examples with non-trivial kernel under specialization,
together with inspiring discussions with Bhatt and Petrov separately, 
that leads us to analyze and generalize the example in \cite[Subsection 2.1]{BMS1}.
The Enriques surface used there turns out to be a little bit red herring, the actual purpose
it serves is just an approximation of classifying stack of $\mathbb{Z}/2$,
like our $(\mathcal{Y}/(\mathbb{Z}/p^n))$ here.

Finally let us explain why our example negates a prediction of Breuil \cite[Question 4.1]{BreuilIntegral}.
Let $S$ denote the $p$-adic PD envelope of $\s \twoheadrightarrow \mathcal{O}_K$.

\begin{proposition}
There is an exact sequence:
\[
0 \to \rH^1_{\cris}(\mathcal{Z}/S) \hookrightarrow S \cdot \{e_1, e_2\} \xrightarrow{d_2} S/p^n,
\]
where $d_2(e_1) = 0$ and $d_2(e_2) = (u+1)^{p^n} - 1$. 
In particular $\rH^1_{\cris}(\mathcal{Z}/S)$ is torsion-free rank $2$
but \emph{not} free unless $(n,p) = (1,2)$.
\end{proposition}

\begin{proof}
In \Cref{prop calculate pris of scheme example} we see that the map
$\mathcal{Z} \to \mathcal{X}$ induces isomorphism in degree $1$ prismatic cohomology and $u^\infty$-torsion
in degree $2$ prismatic cohomology.
The comparison between prismatic and crystalline cohomology \cite[Theorem 1.5]{LL20} (see also \cite[Theorem 5.2]{BS19})
tells us that $\rH^1_{\cris}(\mathcal{X}/S) \xrightarrow{\cong} \rH^1_{\cris}(\mathcal{Z}/S)$.
The same comparison result implies that 
after applying $-\otimes_{\s} \varphi_*S$ to the spectral sequence \ref{ex spectral sequence},
one calculates the crystalline cohomology of $\mathcal{Z}/S$.
Therefore the first statement follows from \Cref{lem d2 of e1} and \Cref{lem d2 of e2}.
Note that $\varphi\big((u+1)^{p^{n-1}} - 1\big) \equiv (u+1)^{p^n} - 1$ modulo $p^n$.

To see the second assertion, note that $\rH^1_{\cris}(\mathcal{Z}/S) \cong S \cdot e_1 \oplus J \cdot e_2$
where $J$ is the ideal 
\[
\{x \in S \mid p^n \text{ divides } x \cdot ((u+1)^{p^n} - 1)\}.
\]
If $J$ were free, then it would be generated by a particular such element, denoted below as $g$.
Let $g= \sum\limits _{i = 0}^\infty a_i \frac{u^i}{e(i)!}$ 
with $ a_i \in W(k)$ approaching $0$ and $e(i) = \lfloor\frac{i}{e} \rfloor$ where $e = p^{n-1}\cdot(p-1)$,
note that every element in $S$ can be uniquely expressed in this form.
Since $p^n$ trivially lies in $J$, it must also be divisible by this $g$. 
Therefore there exists $h_1\in S$ such that $g h_1 = p^{n}$.
Write $q_n = (u +1)^{p ^n}-1 $.
\begin{claim}
$a_0$ is nonzero and divisible by $p$.
\end{claim}
\begin{proof}
The fact that $a_0$ is nonzero follows from $g h_1 = p ^{n}$.
If $a_0$ is a unit in $W(k)$ then $g \in S^\times$ is a unit, 
which implies $q_n \in p^n S$. But this is equivalent to $n =1$ and $p =2$.
\end{proof}
So now we can assume that $a_0 = p a'_0$ with $a'_0 \not = 0$.
Pick $\frac{u ^m}{e(m)!}$ so that $\frac{u ^m}{e(m)!} q_n \in p ^n S$ (select $m = p ^n e -1$ for example). 
Then we have $ g h_2 = \frac{u ^m}{e(m)!} $ for some $h_2 \in S$. 
The above equation implies that $h_2 =  \sum\limits_{i = m}^{\infty}b_i\frac{u ^{i}}{e(i)!}$. 
But compare $u ^m$ term on both sides, we have $a_0 b_m = 1$ which contradicts $a_0 = p a'_0$.
This finishes the proof.

\end{proof}

\begin{remark}
In Breuil's terminology, this shows that the first crystalline cohomology of our examples are
\emph{not} strongly divisible lattices \cite[Definition 2.2.1]{BreuilIntegral}.
This contradicts the claimed \cite[Theorem 4.2.(2)]{BreuilIntegral},
in the proof of loc.~cit.~one is led to Faltings' paper \cite{Faltings}.
However Faltings was treating the case of $p$-divisible groups, hence Breuil's Theorem/proof should only
be applied to abelian schemes.
Now it is tempting to say smooth proper schemes over $\mathcal{O}_K$ and their Albanese should share the
same $\rH^1$ for whatever cohomology theory.\footnote{
To quote Sir Humphrey Appleby: ``It is not for a humble mortal such as I 
to speculate on the complex and elevated deliberations of the mighty.''
But we suspect this is what Breuil had in mind when he claimed that his conjecture holds for $\rH^1$.}
But our example clearly negates this philosophy: the stacky example is squeezed between
two abelian schemes and neither should really be the ``mixed characteristic $1$-motive'' of our stack (even though
these two abelian schemes are abstractly isomorphic).
Indeed the sequence $\mathcal{E}_{\mathcal{O}_K} \to \mathcal{X} \to \mathcal{E}_{\mathcal{O}_K}$
has the property that the first map only induces an isomorphism of first crystalline cohomology of the special fibre
(relative to $W$) and the second map only induces an isomorphism of first \'{e}tale cohomology of the (geometric) generic fibre.
\end{remark}

\subsection{Raynaud's theorem on prolongations}
\label{Raynaud subsection}
Lastly we give a geometric proof of Raynaud's theorem 
\cite[Th\'{e}or\`{e}me 3.3.3]{Ray74}
on uniqueness of prolongations of finite flat commutative group schemes
over a mildly ramified $\mathcal{O}_K$.

Let $G_K$ be a finite flat commutative group scheme over $K$.
A prolongation of $G_K$ is a finite flat commutative group scheme $G$
over $\mathcal{O}_K$ together with an isomorphism of its generic fiber with
$G_K$ (as finite flat commutative group schemes).
Once $G_K$ is fixed, its prolongations form a category with homomorphisms
given by maps of group schemes compatible with the isomorphisms of their generic
fiber.

Recall \cite[Corollaire 2.2.3.]{Ray74} that the
(possibly empty) category of prolongations
of a finite flat group scheme $G$ over $K$ has an initial $G_{\rm{min}}$
and a terminal object $G_{\rm{max}}$.
Moreover these two are interchanged under Cartier duality.

\begin{theorem}[{c.f.~\cite[Th\'{e}or\`{e}me 3.3.3]{Ray74}}]
Assume $G_K$ is a finite flat commutative group scheme
which has a prolongation over $\mathcal{O}_K$.
\begin{enumerate}
\item If $e < p-1$, then the prolongation is unique;
\item If $e < 2(p-1)$, then the reduction of the
canonical map $G_{\rm{min}} \to G_{\rm{max}}$ has kernel
and cokernel annihilated by $p$;
\item If $e = p-1$, then the reduction of the above map
has \'{e}tale kernel and multiplicative type cokernel.
\end{enumerate}
\end{theorem}

\begin{proof}
To ease the notation, let us denote $G_1 \coloneqq G_{\rm{min}}$
and $G_2 \coloneqq G_{\rm{max}}$.
Denote the canonical map by $\rho \colon G_1 \to G_2$.
Choose a group scheme embedding $G_2 \to \mathcal{A}$
of $G_2$ into an abelian scheme $\mathcal{A}$ over $\mathcal{O}_K$,
which is guaranteed by yet another Theorem of Raynaud
(see \cite[Th\'{e}or\`{e}me 3.1.1]{BBM82}).

We shall contemplate with the quotient stack $[\mathcal{A}/G_1]$,
which is a smooth proper Artin stack.
Similar to \Cref{cons approximating stack}, let us pick
a smooth complete-intersection $\mathcal{Y}$ with a fixed-point free
action by $G_1$, let 
$\mathcal{Z} \coloneqq (\mathcal{A} \times_{\mathcal{O}_K} \mathcal{Y})/G_1$,
which is a smooth projective scheme over $\mathcal{O}_K$, thanks to the second factor.
Moreover $\mathcal{Z}$ is pointed because it admits a map from $\mathcal{A}$,
which has a canonical point given by the identity section.

Let $H$ be the image group scheme of the map $\rho_k \colon G_{1,k} \to G_{2,k}$.
Applying the same reasoning as in \Cref{rmk stacky example} shows us that
the canonical map $f \colon \mathrm{Alb}(\mathcal{Z}_0) \to \mathrm{Alb}(\mathcal{Z})_0$
is identified with $\mathcal{A}_0/H \to \mathcal{A}_0/G_{2,k}$,
whose kernel group scheme is given by $G_{2,k}/H$,
which is none other than the cokernel of $\rho_0$.
Now our statements on $\coker(\rho_0)$ follows directly from applying 
\Cref{defect controlled by ramification index} to our $\mathcal{Z}$.
The statements on $\ker(\rho_0)$ follows from Cartier duality.
\end{proof}

\begin{remark}
Note that Raynaud first proved his theorem on prolongations,
then use it to prove statements concerning Picard scheme of a $p$-adic integral
scheme, which is directly related to statements concerning
natural map between Albanese of reduction and reduction of Albanese,
see \Cref{translation to Picard remark}.
Our roadmap is the exact opposite. 

Finally, we remark that  the estimate of $s$ so that $p ^s$ kills the corkernel of 
$G_{\rm min}\to G_{\rm max}$ has been studied before, see for example, 
\cite{Vasiu-Zink} and \cite{Bondako} (and the references therein), which used completely different methods than ours. 
Note that an affirmative answer to our \Cref{general question} for $i = 2$, when specialized to the construction made
in the above proof, agrees with Bondako's sharp estimate.
\end{remark}

\bibliographystyle{amsalpha}
\bibliography{main}

\end{document}